\documentclass[reqno]{amsart}
\usepackage{hyperref}
\usepackage{geometry}
\usepackage{xcolor}
\usepackage{tikz}
\usetikzlibrary{decorations.pathreplacing}
\usepackage[utf8]{inputenc}
\usepackage[english]{babel}
\usepackage{upgreek}
\usepackage{tipa}
\usepackage{lastpage}
\usepackage{float}
\usepackage[shortlabels]{enumitem}
\usepackage{subcaption}
\usepackage{multirow}
\usepackage{graphicx} 
\graphicspath{ {c:/Users/Owner/Desktop} }
\usepackage{amsmath}
\usepackage{amssymb}
\usepackage{bbm}
\usepackage{dsfont}
\usepackage{amsthm}
\usepackage[normalem]{ulem}

\let\oldsqrt\sqrt
\def\sqrt{\mathpalette\DHLhksqrt} \def\DHLhksqrt#1#2{%
\setbox0=\hbox{$#1\oldsqrt{#2\,}$}\dimen0=\ht0
\advance\dimen0-0.2\ht0
\setbox2=\hbox{\vrule height\ht0 depth -\dimen0}%
{\box0\lower0.4pt\box2}}

\def\b0{\boldsymbol{0}}

\newcommand{\R}     {\mathbb{R}} 
\newcommand{\Z}     {\mathbb{Z}} 
\newcommand{\N}     {\mathbb{N}}

\newcommand{\Bcal}   {{\mathcal B }}

\newcommand{\Ical}   {{\mathcal I }}

\newcommand{\m} {{\mathfrak m}}

\newcommand{\eps}{\varepsilon}

\newcommand{\Is}{{\operatorname {Is}\,}} 
\newcommand{\Exp}{\mathscr{E}\kern-0.2mm{\operatorname{xp}}}
\newcommand{\Log}{\mathscr{L}\kern-0.2mm{\operatorname{og}}}

\def\1{{\mathchoice {1\mskip-4mu\mathrm l}      
{1\mskip-4mu\mathrm l}
{1\mskip-4.5mu\mathrm l} {1\mskip-5mu\mathrm l}}}

\usepackage{mathtools}

\numberwithin{equation}{section}
\numberwithin{figure}{section}
\newtheoremstyle{plain}
  {6pt}
  {4pt}
  {\slshape}
  {}
  {\bfseries}
  {.}
  {0.5em}
  {}%
\newtheorem{thm}{\protect\theoremname}
  \newtheorem{defn}[thm]{\protect\definitionname}
  
  \newtheorem{prop}[thm]{\protect\propositionname}
  \newtheorem{rem}[thm]{\protect\remarkname}
  \newtheorem{cor}[thm]{\protect\corollaryname}
  \newtheorem{lem}[thm]{\protect\lemmaname}
  \numberwithin{thm}{section}

\usepackage[english]{babel}
  \addto\captionsbritish{\renewcommand{\corollaryname}{Corollary}}
  \addto\captionsbritish{\renewcommand{\definitionname}{Definition}}
  \addto\captionsbritish{\renewcommand{\factname}{Fact}}
  \addto\captionsbritish{\renewcommand{\propositionname}{Proposition}}
  \addto\captionsbritish{\renewcommand{\remarkname}{Remark}}
  \addto\captionsbritish{\renewcommand{\theoremname}{Theorem}}
  \addto\captionsenglish{\renewcommand{\corollaryname}{Corollary}}
  \addto\captionsenglish{\renewcommand{\definitionname}{Definition}}
  \addto\captionsenglish{\renewcommand{\factname}{Fact}}
  \addto\captionsenglish{\renewcommand{\propositionname}{Proposition}}
  \addto\captionsenglish{\renewcommand{\remarkname}{Remark}}
  \addto\captionsenglish{\renewcommand{\theoremname}{Theorem}}
  \providecommand{\corollaryname}{Corollary}
  \providecommand{\definitionname}{Definition}
  \providecommand{\factname}{Fact}
  \providecommand{\propositionname}{Proposition}
  \providecommand{\remarkname}{Remark}
\providecommand{\theoremname}{Theorem}
\providecommand{\lemmaname}{Lemma}

\usepackage{enumitem}

\usepackage{thmtools}
\usepackage{thm-restate}
\makeatletter
\newcommand*{\defeq}{\mathrel{\rlap{%
                     \raisebox{0.3ex}{$\m@th\cdot$}}%
                     \raisebox{-0.3ex}{$\m@th\cdot$}}%
                     =}
\makeatother

\title{
Combinatorial identities from an inhomogeneous Ising chain
}

\author{Jessica Jay}
\address[Jessica Jay]{School of Mathematics, University of Bristol, Bristol, BS8 1TW, UK, and, \newline \indent The Heilbronn
Institute for Mathematical Research, Bristol, UK.}
\email{jessica.jay@bristol.ac.uk}

\author{Benjamin Lees}
\address[Benjamin Lees]{School of Mathematics, University of Leeds, Leeds, LS2 9JT, UK}
\email{b.t.lees@leeds.ac.uk}
\date{}

\begin{document}

\begin{abstract}
    We study a family of inhomogeneous Ising chain models along with an equivalent family of nearest neighbour particle systems. By the correspondence between the two families we prove identities of combinatorial significance relating to certain integer and Frobenius partitions. In particular, for certain parameter values we see that one of our identities relates to generating functions for overpartitions. Using the identities we give a surprising product form of the partition function for an Ising chain with homogeneous interaction and an inhomogeneous external field. We also use the connection between the Ising chain and particle system to find interesting long-range reversible dynamics for the particle system that do not have a product form stationary measure. 
\end{abstract}

\maketitle

\section{Introduction}

Recent research has shown that studying natural distributional questions for certain interacting particle systems under their reversible stationary measures, known as blocking measures, can lead to probabilistic proofs and interpretations of both classical and new identities of combinatorial significance. In \cite{blocking}, Bal\'azs and Bowen gave a probabilistic proof of the Jacobi triple product identity. This identity arises as an equivalence of blocking measures for the Asymmetric Simple Exclusion Process (ASEP) and the Asymmetric Zero-Range Process (AZRP). Both of these process were introduced by Spitzer in 1970 \cite{spitzer} and have been used to model traffic jams and queues in series, for example. ASEP is a particle system on $\mathbb{Z}$ with at most one particle per site; the process evolves by particles performing nearest neighbour jumps under an exclusion rule. AZRP is a particle system on $\mathbb{Z}_{< 0}$ with no limit on the number of particles per site. Particles complete nearest neighbour jumps; also particles can enter and leave the system through an open right boundary. The proof follows from the exclusion-zero range correspondence. Other natural questions for ASEP under its blocking measure also lead to probabilistic proofs of known identities; in \cite{ad_bal_j_asep_proofs}, Adams, Bal\'azs and Jay, prove Euler's Identity, the $q$-Binomial Theorem and the Durfee Rectangles Identity. 
\par A natural question is what other identities can be proved by considering generalisations of ASEP? In \cite{MDJ}, Bal\'azs, Fretwell and Jay, considered exclusion systems with more particles allowed per site. In particular, by considering the family of 0-1-2 processes on $\mathbb{Z}$ with product blocking measures and an equivalent family of kinetically constrained models on $\mathbb{Z}_{<0}$, they proved new 3 variable Jacobi style identities relating to Generalised Frobenius Partitions (GFPs) with a 2 repetition condition. Specialising to particular processes in the 0-1-2 family, for example ASEP$(q,1)$, 2-exclusion and a particle-antiparticle process, specialises the identities to known identities studied by Andrews \cite{Andrews_GFP}, for example. In \cite{MDJ}, the authors also considered the $k$-exclusion process (for any $k\in\mathbb{Z}_{>0}$) and its equivalent kinetically constrained model. This gave a probabilistic proof to similar Jacobi style identities relating to GFPs with a $k$-repetition condition. 
\par In this article we take a similar approach, studying processes with at most 1 particle per site with more general interaction rules than simple exclusion. By introducing an interaction that creates a preference or aversion for particles to neighbour each other we have a natural correspondence to the Ising model.
\par The Ising model is a simple model of magnetism originally invented by Lenz for his student Ising. Each vertex of a graph has a degree of freedom called a \emph{spin}, taking values in $\{-1,+1\}$, that interacts with their neighbour's spins. The magnetism of the system, for example, is given by the average of these spin values. It is now probably the most studied model in statistical mechanics, with a vast literature. For an introduction to this model we direct the reader to \cite{FandV}. In this article we will introduce an Ising chain, i.e. an Ising model on $\Z$, with inhomogeneous interactions between spins and an inhomogeneous external field. The interactions will be chosen to ensure that our measure on spin configurations concentrates on a countable set where the 
energy of interactions is finite. This concentration occurs because our external field will strongly penalise $-1$ spins at sites $i\gg 1$ and $+1$ spins at sites $i\ll 0$. One can think physically of starting on a finite chain $\{-L,\dots,0,\dots,L\}$ with a magnet at $L+1$ attracting $+1$ spins and then simultaneously increasing $L$ and the strength of the magnet. A similar model of a driven lattice gas was considered by Katz, Lebowitz, and Spohn \cite{KLS} on hypercubic lattices where simulations suggest a non-equilibrium phase transition when the strength of the external field varies.
\par We consider Kawasaki dynamics on this Ising model. These dynamics exchange differing spins at a rate giving rise to stationary and reversible dynamics on the chain. By a process of \emph{standing up}, the direct analogue of the exclusion-zero range correspondence, we can map this Ising chain to a system of interacting particles on $\Z_{<0}$ with more general dynamics than the asymmetric zero-range process (AZRP). In fact, when the strength of interaction in the Ising chain is sent to zero we recover ASEP and, through standing up, AZRP. The results mentioned above that derive from considering multiple particle types also suggests the possibility of a natural extension in the interacting case to multiple particle types, and also to similar considerations for the Potts model. In this case multiple stand up maps are required ($q-1$ maps for the $q$-state Potts model) to recover the spin configuration from the resulting particle systems. There is also the possibility of using correspondences between quantum spin chains (such as those coming from the Jordan-Wigner transform) in place of the standing up process to derive identities, it would be very interesting to understand the nature of such identities. 

\par By equating stationary measures for the Ising chain and equivalent particle systems we find identities of combinatorial significance relating to generating functions of certain integer and Frobenius partitions (for an introduction to Frobenius partitions see Andrews \cite{Andrews_GFP}). In particular, we see that for certain parameter values we can relate one of the identities to generating functions of overpartitions, that is integer partitions where the first instance of each size of part can be either overlined or not. Overpartitions have been shown to relate to interesting $q$-series identities such as $q$-Gauss summation and Ramanujan's $_1\varphi_1$ summation (see works by Corteel, Lovejoy and Yee for example, \cite{Lovejoy_Corteel_OPs}, \cite{Lovejoy_Corteel_Yee_OPs_GFPs} and \cite{Lovejoy_OPs}). This link between Ising configurations and overpartitions is very interesting and the authors wonder if studying other questions for the Ising process would lead to new results in the theory of overpartitions. We also see that other parameter values relate the identities to generating functions for partitions where the first part of each size can be represented in one of $m$ ways, for some given $m\in\mathbb{Z}_{>0}$. We call these partitions ``$(m-1)$-overpartitions", in particular if $m=1$ these are just integer partitions and for $m=2$ these are classical overpartitions. These generating functions are given on OEIS (A321884) \cite{OEIS} with the interpretation coming from partitions into coloured blocks of equal parts from $m$ different colours. The authors would be interested to know if the generating functions for these more general $(m-1)$-overpartitions relate to other $q$-series identities similar to the $q$-Gauss and Ramanujan's $_1\varphi_1$ summations seen in the $m=2$ case.

\par \vspace{2mm} Often $q$-series identities, such as the Jacobi triple product identity, have interpretations in terms of the representation theory of affine Lie algebras. Some of the particle systems that lead to probabilistic proofs of such identities, such as ASEP and ASEP$(q,1)$ are known to have underlying algebraic structure. In particular, these two processes can be constructed from representations of $\mathcal{U}_q(\mathfrak{gl}_2)$, the quantum group associated to the Lie algebra $\mathfrak{gl}_2$, (see the works of Carinci, Giardin\'a, Redig and Sasamoto \cite{redig} and also Kuan \cite{kuan}). Going the other way and finding the underlying algebraic structure to a process is very difficult; some processes, such as 2-exclusion, that lead to probabilistic proofs of combinatorial identities are yet to be found algebraically. It would be interesting to know if there is an algebraic interpretation of the identities we find in this paper and also if the inhomogeneous Ising chain we study here has an underlying algebraic structure which it can be constructed from.

\par \vspace{2mm} It is natural and fairly simple to consider Ising models with longer range interactions. We hence generalise the Ising model above to allow interactions between spins up to some arbitrary distance from each other. Under mild conditions on this interaction it is possible to define stationary and reversible dynamics allowing the exchange of distant spins. These dynamics can then be transferred to the particle system through the standing up mapping, resulting in quite unusual long range dynamics that are stationary and reversible. This procedure for finding stationary and reversible long range dynamics also works ``in reverse". By starting with natural long range dynamics for the particle system we can then use the mapping to find a corresponding Ising-like model (this will be a measure on spins but not one taking the familiar Ising form) for which the dynamics are stationary and reversible. We can then use the mapping again to find the measure for the particle system for which the dynamics are stationary and reversible. We note that, interestingly, these dynamics do not have a product form. 
\newpage\subsection{Results}~
\par Our family of Ising chain models are parameterised by an interaction function $J:\mathbb{Z}\to\mathbb{R}$ which, to ensure countability of the state space of the process, satisfies conditions which we give in Lemma \ref{lem:nearestneighbourconcentration}. Many interaction functions satisfy these conditions, and in particular we consider two functions of interest, $J\equiv 1$ which corresponds to the homogeneous Ising interaction, and also $J(i)=i$ (an inhomogeneous Ising interaction). In addition the Ising chain will feature an inhomogeneous external field where, roughly speaking, a positive spin at site $i<0$ is subject to a factor $q^{-2i}$ and a negative spin at site $i\geq1$ is subject to a factor $q^{2i}$, for some $q\in(0,1)$. For details see Section \ref{sec:Ising}.
\par By considering the homogeneous Ising interaction and equivalent particle system we prove the following identity.
\begin{restatable}{thm}{constident}
\label{thm: J is 1 identity}
For $Q\in(0,1)$, $z>0$, and $y \in (0,1]$, 
  $$\sum\limits_{m\in\mathbb{Z}} Q^{\frac{m(m+1)}{2}}z^m=Z(Q,z,y)\prod\limits_{i=1}^\infty\frac{1-Q^i}{1+(y-1)Q^i}.$$
  Where, 
  \begin{align*}
    &Z(Q,z,y) =  1 + \sum_{\substack{L,R\geq0\\ L+R>0}}y^{(L+R)} \sum_{\substack{\ell_1,\dots,\ell_L\geq 1 \\ m_1,\dots,m_{R}\geq 1}}\prod_{j=1}^{L} \frac{Q^{\tfrac12 \ell_j(\ell_j-1)+j\ell_j+\ell_{j-1}(\ell_j+\cdots+\ell_{L})}z^{-l_j}}{1-Q^{\ell_j+\cdots + \ell_{L}}}
    \\
    &\qquad\qquad\qquad\qquad\prod_{j=1}^{R}\frac{Q^{\tfrac12 m_j(m_j-1)+jm_j+m_{j-1}(m_{j}+\cdots+m_{R})}z^{m_j}}{1-Q^{m_{j}+\cdots+m_{R}}}\bigg(\big(1-y^{-1}\big)Q^{\ell_1+\cdots + \ell_L}Q^{m_{1}+\cdots+m_{R}} + y^{-1}\bigg),
\end{align*}
with $\ell_0=-1$ and $m_0=0$.
\end{restatable}
In Section \ref{sec: comb j is 1} we discuss combinatorial interpretations of this identity. Firstly, we look at the constant in $z$, we see a 2 variable generating function for for integer partitions where the power of $Q$ is the number being partitioned and the power of $y$ is the number of different sizes of part in a partition. We then use a Wright bijection argument \cite{wright} to give interpretation to the other $z^m$ terms (for each $m\in\mathbb{Z}$). We see three variable generating functions for Frobenius partitions where the power of $Q$ is the number being partitioned, the power of $y$ counts a quantity for Frobenius partitions equivalent to the number of different sizes of part, and the power of $z$ gives the offset of the Frobenius partition (similar to the identities of Bal\'azs, Fretwell and Jay \cite{MDJ}).  If we take $y=2$ we see that the identity above relates to the generating function of overpartitions and furthermore if we take $y=m$ for any $m\in\mathbb{Z}_{>0}$ we see the generating function of $(m-1)$-overpartitions as discussed above.
\par By combining the identity of Theorem \ref{thm: J is 1 identity} and the Jacobi Triple Product identity we find a surprising product form for the Ising partition function with homogeneous interactions and inhomogeneous external field defined in \eqref{eq:IsingZ}.
\begin{restatable}{cor}{partitionfunc}\label{cor:partn}
    The Ising partition function given by \eqref{eq:IsingZ} (for $J(i)\equiv 1$) can be written as a product, 
    $$Z_{\beta,q,c}=e^{-\beta}\prod\limits_{i=1}^\infty(1+(e^{-2\beta}-1)q^{2i})(1+q^{2(i-c)})(1+q^{2(i-1+c)}).$$
\end{restatable}
\par \vspace{2mm} Similarly, by considering the inhomogeneous Ising chain with interaction function $J(i)=i$ and equivalent particle systems we prove the following family of identities (one for each $n \in\mathbb{Z}$).
\begin{restatable}{thm}{iiden} \label{thm: J is i identity}
Define, $\Omega\defeq \{\omega \in \mathbb{Z}_{\geq 0}^{\mathbb{Z}_{<0}} : \exists N > 0 \hspace{2mm}\textrm{ s.t } \omega_{-i}=0 \hspace{2mm} \forall i \geq N\}$. For $Q\in(0,1)$, $y\in(0,1]$, and each $n\in\mathbb{Z}$,

$$Z_n(Q,y)=\sum\limits_{z\in\Omega}\prod\limits_{i=1}^\infty y^{\1_{\{z_{-i}>0\}}(n+i-\sum\limits_{j=i}^\infty z_{-j})}Q^{iz_{-i}}.$$
Where, 
\begin{align*}
    &Z_n(Q,y) = 1 + \sum_{L,R>0} y^{nL+(n-1)R}\sum_{\substack{\ell_1,\dots,\ell_L\geq 1 \\ m_1,\dots,m_{R}\geq 1}}\1_{\big\{\sum_{j=1}^R m_j= \sum_{j=1}^L\ell_j\big\}} \prod_{j=1}^{L} \frac{Q^{\tfrac12 \ell_{j}(\ell_{j}-1)}\big(y^{-(L-j+1)}Q^{\ell_{j}+\cdots + \ell_{L}}\big)^{\ell_{j-1}+1}}{1-y^{-(L-j+1)}Q^{\ell_{j}+\cdots + \ell_{L}}} \\
    &\qquad\qquad\qquad\prod_{j=1}^{R}\frac{Q^{\tfrac12 m_{j}(m_{j}-1)}\big(y^{R-j+1}Q^{m_{j}+\cdots+m_{R}}\big)^{m_{j-1}+1}}{1-y^{R-j+1}Q^{m_{j}+\cdots+m_{R}}}\bigg(y^{-n}+\big(1-y^{-n}\big)y^{R-L}Q^{\ell_1+\cdots + \ell_L}Q^{m_{1}+\cdots+m_{R}}\bigg)
\end{align*}   
with $\ell_0=-1$ and $m_0=0$.
\end{restatable}
In Section \ref{sec: comb j is i} we discuss the combinatorial interpretations of this family of identities. For any $n\in\mathbb{Z}$, we see that the right side of the identity is a 2 variable generating function for integer partitions where the power of $Q$ is the number being partitioned and the power of $y$ can be split into three parts counting, roughly speaking:
\begin{itemize}
    \item The number of different sizes of part (as in Theorem \ref{thm: J is 1 identity}).
    \item The ``minimal" integer that can be partitioned given the part sizes (by taking exactly one part of each size).
    \item What we refer to as the ``partition of partial sums". Suppose we know how many parts there are of each of the $k$ sizes in the order of the sizes (but not the actual values of the sizes). Then the ``partition of partial sums" is given by:
    \begin{itemize}
        \item a part equal to the total number of parts of the original partitions,
        \item a part equal to the number of parts from the $k-1$ biggest sizes,
        \\ \vdots
        \item a part equal to the number of parts of the largest size.
    \end{itemize}
\end{itemize}
The fact that the $y$ power is a combination of three different properties of a partition, suggests it might be more natural for this to be written as a three or four variable generating function for integer partitions. This suggests that there may be a more general (reversible) process than the inhomogeneous Ising we have considered that naturally has a three or four variable stationary measure, whose partition function is the three or four variable version of this generating function.

\section*{Acknowledgements}
\par \vspace{1mm} \noindent The authors would like to thank Dan Fretwell and also Robert Osburn for useful discussions about combinatorial interpretations. We also thank Matthew Aldridge for a very nice suggestion that resulted in section \ref{sec:nicelong}.
\par \noindent J.\ Jay was funded by the Heilbronn Institute for Mathematical Research.

\section{An Asymmetric Blocking Ising process under Kawasaki dynamics}\label{sec:Ising}

We consider an Ising model with nearest-neighbour inhomogeneous coupling constants and external field.  States of this model are evolved under Kawasaki dynamics. The measure of our model is constructed so that there is concentration on blocking configurations of $\mathbb{Z}$, that is, on configurations $\mathcal{B}\subset \Omega^{\Is}:=\{-1,+1\}^{\mathbb{Z}}$ consisting of only $+1$ (or positive) spins to the right of some finite vertex $b$ and only $-1$ (or negative) spins to the left of some finite vertex $a$. Note that, necessarily, $a\leq b$. More precisely, 
\begin{equation}\label{eq:defB}
\mathcal{B}:=\{\sigma\in\Omega^{\Is}\, :\, \exists a,b\in\mathbb{Z}\text{ s.t. }\forall i\in\mathbb{N}\,\,\, \sigma_{a-i}=-1,\, \sigma_{b+i}=+1\}.
\end{equation}
This space is countable.
Each state in $\sigma\in\Omega^{\Is}$ has an associated energy, given by the hamiltonian 
\begin{equation}\label{eq:ham}
H_{J}(\sigma)=-\sum\limits_{i\in\mathbb{Z}}J(i)\frac{\sigma_i\sigma_{i+1}-1}{2} = \sum_{i\in\Z}J(i)\1_{\{\sigma_i\neq\sigma_{i+1}\}},
\end{equation}
where $J:\Z\to\R$ is some real valued function. Notice that, if the sign of $J(i)$ is constant, $|H_{J}(\sigma)|=\infty$ unless $\sigma$ has only finitely many disagreements of neighbouring spins. Each $\sigma\in\Bcal$ has only finitely many disagreements of neighbouring spins and we will see that our measure defined below in \eqref{eq:Ising measure} concentrates on $\Bcal$. 
\par The states evolve by neighbouring pairs of disagreeing spins swapping, when this happens the total energy of the system can change. The rate at which the spins swap depends on how the swap affects energy, which depends on four spins in total. Before we introduce the correct rates to satisfy detailed balance, we introduce an inhomogeneous external field so that our model will concentrate on $\Bcal$, where $H_{J}(\sigma)$ is finite. This external field is motivated by the study of the Asymmetric Simple Exclusion Process (ASEP) and will have the effect that when the energy term is removed the process will be precisely ASEP. We may hence want to consider this field as introducing an asymmetry in the rates of spin exchange which, energy change aside, favours spins with value +1 moving to the right (i.e. in the direction $i\rightarrow i+1$ for $i\in \mathbb{Z}$).

\par Any swap that causes a positive spin to move left will happen at a rate determined from the hamiltonian multiplied by $q\in(0,1)$ and swaps where a positive spin moves right by $q^{-1}$. To achieve this we introduce a function $f_c:\mathcal{B}\to \mathbb{Z}$, $c\in\mathbb{R}$, given by
\begin{equation}
f_c(\sigma)=\sum_{i=1}^{\infty}(i-c)(1-\sigma_i) - \sum_{i=-\infty}^0 (i-c)(1+\sigma_i).
\end{equation}
Note that 
\begin{equation}\label{eq:fc}
f_c(\sigma)=2\sum_{i=1}^\infty(i-c)\1_{\{\sigma_i=-1\}}-2\sum_{i=-\infty}^0(i-c)\1_{\{\sigma_i=1\}}.
\end{equation}
Now for $c\in \mathbb{R}$, $\beta\geq0,$ and $q\in(0,1)$ we define a probability measure $\mu^c_{J, \beta,q}=\mu^c_J$ on $\Omega^\Is$ as
\begin{equation}\label{eq:Ising measure}
\mu^c_J(\sigma) = \frac{1}{Z^{J}_{\beta,q,c}} e^{-\beta H_{J}(\sigma)}q^{f_c(\sigma)},
\end{equation}
where 
\begin{equation}\label{eq:IsingZ}
Z^{J}_{\beta,q,c} = \sum_{\sigma\in\Omega^{\Is}} e^{-\beta H_{J}(\sigma)}q^{f_c(\sigma)}.
\end{equation}
\begin{rem}
    We note that when we take $\beta=0$, for any interaction function $J$, we indeed see ASEP. Moreover this gives us that,
    $$Z_{0,q,c}^J=Z_\text{ASEP}(q,c)=\prod\limits_{i=1}^\infty(1+q^{2(i+c)})(1+q^{2(i-1+c)}).$$
    We will further study the Ising partition function, for any $\beta$, in Section \ref{sec: ising partitions function}. 
\end{rem}
\par In order to distinguish between the factors of $e^{-\beta}$ and $q$ in $\mu^c_J(\sigma)$ we will refer to the factor of $e^{-\beta}$ as the \emph{energy} term and call conditions or configurations energetically favourable or lower energy if they have a smaller value of $H_J(\sigma)$ than the condition or configuration they are being compared with.

The required swap rates for the dynamics are given by the detailed balance condition. If $w(\sigma,\sigma^{\prime})$ is the transition rate from state $\sigma$ to $\sigma^{\prime}$ then $w(\sigma,\sigma^{\prime})=0$ unless $\sigma$ and $\sigma^{\prime}$ differ by a single swap of neighbouring spins. If they do differ by a single swap of neighbouring spins then $w(\sigma,\sigma^{\prime})$ depends on properties of $\sigma$ and $\sigma^\prime$. The first property that determines swap rates is how the number and positions of neighbouring spin disagreements changes. The cases where the  number of disagreements increases, decreases, and stays the same (but possibly change location) will be called disagreement increasing, disagreement decreasing, and disagreement neutral, respectively. These types of swaps are usually called energy increasing/decreasing/neutral in the case of no, or homogeneous, external field. The second property is whether a positive spin moves to the left, or to the right, when transitioning from $\sigma$ to $\sigma^\prime$. In each case below we suppose that the proposed swap is of differing spins $\sigma_i$ and $\sigma_{i+1}$.

\begin{itemize}
\item \textbf{Disagreement increasing:} 
    \begin{itemize}
     \item A swap that is disagreement increasing and such that positive spin moves left, for example $--++\to -+-+$, happens at rate 
     $$
    w(\sigma,\sigma^{\prime})=\frac{1}{2}\Big(1-\tanh\big(\beta\tfrac{J(i-1)+J(i+1)}{2}\big)\Big)\cdot q.
    $$
     \item A swap that is disagreement increasing and such that a positive spin moves right, for example $++--\to+-+-$, happens at rate 
     $$
     w(\sigma,\sigma^{\prime})=\frac{1}{2}\Big(1-\tanh\big(\beta\tfrac{J(i-1)+J(i+1)}{2}\big)\Big)\cdot q^{-1}.
     $$
    \end{itemize}
    \item \textbf{Disagreement Neutral:} 
    \begin{itemize}
    \item \par  A swap involving two agreeing spins is energy neutral, we say that these happen at rate $0$ (we can give any rate to this swap since the measure of the configuration is unchanged).
    \item A swap that is disagreement neutral and such that a positive spin moves left, for example $--+-\to -+--$ or $+-++\to ++-+$, happens at rate
     $$
    w(\sigma,\sigma^{\prime})= \frac{1}{2}\Big(1+\tanh\big(\beta\tfrac{J(i+1)-J(i-1)}{2}\big)\Big)q.
    $$
    \item A swap that is disagreement neutral and such that a positive spin moves right, for example  $-+--\to --+-$ or $++-+\to +-++$, happens at rate 
    $$
    w(\sigma,\sigma^{\prime})= \frac{1}{2}\Big(1-\tanh\big(\beta\tfrac{J(i+1)-J(i-1)}{2}\big)\Big)q^{-1}.
    $$
     \end{itemize}
\item \textbf{Disagreement decreasing:} 
    \begin{itemize}
    \item A swap that is disagreement decreasing and such that a positive spin moves left, for example $+-+-\to ++--$, happens at rate 
    $$
    w(\sigma,\sigma^{\prime})=\frac{1}{2}\Big(1+\tanh\big(\beta\tfrac{J(i-1)+J(i+1)}{2}\big)\Big)\cdot q.
    $$
    \item A swap that is disagreement decreasing and such that a positive spin moves right, for example $-+-+\to --++$, happens at rate
    $$
    w(\sigma,\sigma^{\prime})=\frac{1}{2}\Big(1+\tanh\big(\beta\tfrac{J(i-1)+J(i+1)}{2}\big)\Big)\cdot q^{-1}.
    $$
    \end{itemize}
\end{itemize}

\begin{lem}\label{lem:nearestneighbourconcentration}
For every $c\in\mathbb{R}$, $\beta\geq 0$, and $q\in(0,1)$, the measure $\mu^c_{J}$ is stationary and reversible for the Kawasaki dynamics described above with rates given by  $w(\sigma,\sigma^{\prime})$ provided that it concentrates on $\mathcal{B}$. Moreover, $\mu^c_{J}$ concentrates on $\mathcal{B}$ if 
$$
\sum_{i=-\infty}^0 \left(1+q^{2(i-c)}\frac{e^{\beta J(i-1)}}{e^{\beta |J(i)|}+q^2 e^{\beta |J(i-2)|}}\right)^{-1} + \sum_{i=1}^\infty \left(1+q^{-2(i-c)}\frac{e^{\beta J(i-1)}}{e^{\beta |J(i)|}+q^{-2} e^{\beta |J(i-2)|}}\right)^{-1}<\infty
$$
\end{lem}
\begin{rem}
The summability condition above imposes only rather mild conditions on $J(i)$, for example $J(i)=a|i|+b$ satisfies the condition, and if $-log(q)>\beta$, then $J(i)=\pm i$ also satisfies the condition.

In the proof we will show that there are a.s. only finitely many disagreements of neighbouring spins using a Borel-Cantelli argument, similar to Bal\'azs and Bowen \cite{blocking}. Concentration on $\Bcal$ (rather than one of the other possibilities for spin values far to the left or right of $0$) is then immediate from the definition of $f_c$. 
\end{rem}

\begin{proof}
The rates of spin exchanges were chosen specifically to satisfy detailed balance as a straightforward calculation will verify. We leave this calculation to the interested reader.

Now to show that $\mu^c_J$ concentrates on $\mathcal{B}$ we first show that there are a.s. only finitely many pairs of neighbouring spins that differ under $\mu^c_j$. Concentration on $\mathcal{B}$ is then immediate as $f_c(\sigma)<\infty$ for any $\sigma\in\Bcal$ but having infinitely many $\sigma_i=+1$ for $i\leq 0$ or infinitely many $\sigma_i=-1$ for $i\geq 1$ results in $f_c(\sigma)=\infty$.

For $i<0$ define
$$
\mathcal{I}_i=\{\sigma\in\Omega^{\Is}\,:\,\sigma_i\neq\sigma_{i-1}\}
$$
to be the event that the spins at $i$ and $i-1$ differ. For $\Omega^{< i}=\{\pm1\}^{\{\dots,i-2,i-1\}}$ and $\Omega^{\geq i}=\{\pm1\}^{\{i,i+1,\dots\}}$ define
$$
S^{(i)}(\pm1)=\sum_{\sigma\in\Omega^{< i}}e^{-\beta H_{J}^{< i}(\sigma)}q^{f^{<i}_c(\sigma)}\mathbbm{1}_{\{\sigma_{i-1}=\pm1\}},\quad T^{(i)}(\pm1)=\sum_{\sigma\in\Omega^{\geq i}}e^{-\beta H_{J}^{\geq i}(\sigma)}q^{f^{\geq i}_c(\sigma)}\mathbbm{1}_{\{\sigma_i=\pm1\}}
$$
where $H_{J}^{< i}=\sum_{j<i-1}J(j)\1_{\{\sigma_j\neq\sigma_{j+1}\}}$, $H_{J}^{\geq i}=\sum_{j\geq i}J(j)\1_{\{\sigma_j\neq\sigma_{j+1}\}}$, $f^{<i}_c(\sigma)= -\sum_{j=-\infty}^{i-1} (j-c)(1+\sigma_j)$ and $f^{<i}_c(\sigma)+f^{\geq i}_c(\sigma)=f_c(\sigma)$.

\par We have that
$$
\sum_{\sigma\in\mathcal{I}_i}e^{-\beta H_{J}(\sigma)}q^{f_c(\sigma)}=e^{-\beta J(i-1)}\bigg(S^{(i)}(-1)T^{(i)}(1)+S^{(i)}(1)T^{(i)}(-1)\bigg)
$$
and similarly
$$
\sum_{\sigma\in\Omega^{\Is}}e^{-\beta H_{J}(\sigma)}q^{f_c(\sigma)}=e^{-\beta J(i-1)}\bigg(S^{(i)}(-1)T^{(i)}(1)+S^{(i)}(1)T^{(i)}(-1)\bigg)+\bigg(S^{(i)}(1)T^{(i)}(1)+S^{(i)}(-1)T^{(i)}(-1)\bigg).
$$
This means that
\begin{align*}
\mu^c_J(\mathcal{I}_i)&=\left(1+e^{\beta J(i-1)}\frac{S^{(i)}(1)T^{(i)}(1)+S^{(i)}(-1)T^{(i)}(-1)}{S^{(i)}(-1)T^{(i)}(1)+S^{(i)}(1)T^{(i)}(-1)}\right)^{-1}  
\\
&\leq \left(1+e^{\beta J(i-1)}\left(\frac{T^{(i)}(1)}{T^{(i)}(-1)}+\frac{S^{(i)}(1)}{S^{(i)}(-1)}\right)^{-1}\right)^{-1}  
\end{align*}
where the inequality arises from throwing away the term $S^{(i)}(1)T^{(i)}(1)$.
We hence see that we require an upper bound on $T^{(i)}(1)$ and $S^{(i)}(1)$. Fortunately, simple bounds are
$$
T^{(i)}(1)=\sum_{\sigma\in\Omega^{\geq i}}e^{-\beta H_{J}^{\geq i}(\sigma)}q^{f^{\geq i}_c(\sigma)}\mathbbm{1}_{\{\sigma_i=-1\}}q^{-2(i-c)}\Big(\mathbbm{1}_{\{\sigma_{i+1}=1\}}e^{\beta J(i)}+\mathbbm{1}_{\{\sigma_{i+1}=-1\}} e^{-\beta J(i)}\Big) \leq T^{(i)}(-1)q^{-2(i-c)}e^{\beta |J(i)|},
$$
where the equality used a simple map to flip the spin of $\sigma_i$ and then inserted the correct factors to give $H_{J}(\sigma)$ and $f_c(\sigma)$ for the un-flipped configuration. Similarly we find
$$
S^{(i)}(1)\leq S^{(i)}(-1)q^{-2(i-1-c)}e^{\beta |J(i-2)|}.
$$
Combining these bounds gives that
$$
\mu^c_J(\mathcal{I}_i)\leq \left(1+q^{2(i-c)}\frac{e^{\beta J(i-1)}}{e^{\beta |J(i)|}+q^2 e^{\beta |J(i-2)|}}\right)^{-1}.
$$
Hence by Borel-Cantelli, there will be a.s. only finitely many disagreements of spins in $\{\dots,-2,-1,0\}$ if this quantity is summable over $i\leq 0$.
\par
Similarly, if
$$
\sum_{i\geq 1}\left(1+q^{-2(i-c)}\frac{e^{\beta J(i-1)}}{e^{\beta |J(i)|}+q^{-2} e^{\beta |J(i-2)|}}\right)^{-1}<\infty
$$
then there are only finitely many disagreements of spins in $\{1,2,\dots\}$.
\end{proof}

From now on, we will assume that the summability condition in Lemma \ref{lem:nearestneighbourconcentration} holds, and therefore that $\mu^c_{J}$ concentrates on $\mathcal{B}$. Some choices of $J(i)$ are of particular interest to us.

\begin{rem} We limit ourselves to two $J(i)$ of interest.
\begin{enumerate}
\item $J(i)\equiv 1$. This is the simplest case and perhaps physically the most natural. Spins find agreement with neighbours energetically favourable. In this case we will drop the subscript and superscript $J$ from the notation.

\item $J(i)=i$. This case is inhomogeneous, spins to the right of 1 find agreement energetically favourable and spins to the left of 0 find disagreement with their neighbours energetically favourable. In order to satisfy the condition of Lemma \ref{lem:nearestneighbourconcentration} we require that $e^{-\beta}<q$. 

\end{enumerate}    
\end{rem}

\subsection{The conserved quantity for Kawasaki dynamics on blocking configurations}
\par

Recall that on a finite graph Kawasaki dynamics preserve the number of +1 spins (or equivalently the magnetism) of the configuration.
As we are working on an infinite chain, and in particular on $\Bcal$, the number of +1 spins (and -1 spins) is infinite, so the number of +1 spins is not a conserved quantity that is available to us. Instead, the blocking configurations suggest another, rather natural, conserved quantity. For some $\ell,r$ satisfying $-\infty\leq \ell\leq 0 < r\leq \infty$ and $\sigma\in\Omega^{\Is}$ define
$$
N_+^{\ell}(\sigma)=\sum_{i=\ell}^0\mathbbm{1}_{\{\sigma_i=+1\}}, \qquad N_-^{r}(\sigma)=\sum_{i=1}^r\mathbbm{1}_{\{\sigma_i=-1\}}
$$
to be the number of +1 (resp. -1) spins on $\{\ell,\ell+1\dots,0\}$ (resp. $\{1,2,\dots,r\}$). Notice that we have
$$
\mathcal{B}=\bigcup_{P,M\in \mathbb{N}}\bigcap_{\ell\leq 0,r>0} \big\{\sigma\in\Omega^{\Is}\,:\, N_+^{\ell}(\sigma)\leq P,\, N_-^r(\sigma)\leq M\big\},
$$
and hence that $\Bcal$ is countable.
We further define
$$
N_+(\sigma)=\lim_{\ell\to\infty}N_+^{\ell}(\sigma),\qquad N_-(\sigma)=\lim_{r\to\infty}N_-^{r}(\sigma).
$$
If $\sigma\in\mathcal{B}$ then both of these limits exist and are finite. Using these quantities we define a conserved quantity. For $\sigma\in\mathcal{B}$
\begin{equation}\label{def:N}
N(\sigma):=N_-(\sigma)-N_+(\sigma).
\end{equation}
This is indeed a conserved quantity for Kawasaki dynamics on $\mathcal{B}$, and so we can decompose $\mathcal{B}$ into disjoint subsets according to the value of $N(\sigma)$. For $n\in\mathbb{N}$
\begin{equation}\label{def:Bn}
\mathcal{B}_n:=\big\{\sigma\in\mathcal{B}\,:\, N(\sigma)=n\big\}.
\end{equation}
This gives us the following proposition.
\begin{prop}\label{stationary dist on B^n}
The unique stationary distribution on $\mathcal{B}_n$ is given by, 
$$\nu^n_{J,\beta,q}(\sigma)\defeq \mu^c_{J,\beta,q}(\sigma|N(\sigma)=n)= \frac{\mu^c_J(\sigma)\1_{\{N(\sigma)=n\}}}{\mu^c_J(\{N=n\})}.$$
\end{prop}

\subsection{Expression for $\mu^c_J(\{N=n\})$ using the shift operator}~
\par 

We now look at the expression $\mu^c_J(\{N=n\})$ in $\nu^n_{J,\beta,q}(\sigma)$. By deriving a convenient expression for this quantity we can obtain useful expressions for our identities in Theorems \ref{thm: J is 1 identity} and \ref{thm: J is i identity}.
\begin{lem}\label{lem:shift}
Take $c\in \R$, $\beta\geq0$, $q\in(0,1)$, and $J:\Z\to\R$ satisfying the conditions of Lemma \ref{lem:nearestneighbourconcentration}. For any $n\in \Z$ we have
$$
\mu_{J}^c(\{N=n\}) = q^{-2nc + n(n+1)}\mu_{J}^c\Big(\1_{\{N=0\}}e^{-\beta H^{(n)}_{J}}\Big),
$$
where 
$$
H_{J}^{(n)}(\sigma) = \sum\limits_{i\in\mathbb Z} \big(J(i+n) - J(i)\big)\1 _{\{\sigma_{i}\neq \sigma_{i+1}\}}.
$$
\end{lem}
\begin{proof}
Let $\tau:\mathcal{B}\to\mathcal{B}$ be the left shift operator, i.e. $(\tau\sigma)_i=\sigma_{i+1}$. We consider how this shift affects the values of $H_{J}$, $f_c$ and $N$. We start with $H_{J}$,
\begin{align*}
   H_{J}(\tau\sigma)&=\sum\limits_{i\in\mathbb Z} J(i)\1 _{\{(\tau\sigma)_i\neq (\tau\sigma)_{i+1}\}}=\sum\limits_{i\in\mathbb Z} J(i+1)\1 _{\{\sigma_{i+1}\neq \sigma_{i+2}\}}+\sum\limits_{i\in\mathbb Z} \big(J(i) - J(i+1)\big)\1 _{\{\sigma_{i+1}\neq \sigma_{i+2}\}}
   \\
   &=H_{J}(\sigma)+\sum\limits_{i\in\mathbb Z} \big(J(i) - J(i+1)\big)\1 _{\{\sigma_{i+1}\neq \sigma_{i+2}\}}.
\end{align*}
We can iterate this identity $n$ times which gives 
\begin{align*}
H_{J}(\tau^n\sigma) =& H_{J}(\sigma) + \sum\limits_{i\in\mathbb Z} \big(J(i) - J(i+n)\big)\1 _{\{\sigma_{i+n}\neq \sigma_{i+n+1}\}}
\\
=&H_{J}(\sigma) + \sum\limits_{i\in\mathbb Z} \big(J(i-n) - J(i)\big)\1 _{\{\sigma_{i}\neq \sigma_{i+1}\}}
\\
=:& H_{J}(\sigma) + H^{(-n)}_{J}(\sigma).
\end{align*}
Similarly we can define the right shift operator by $\tau^{-1}$ and find that
$$
H_{J}(\tau^{-n}\sigma)= H_{J}(\sigma) + H_{J}^{(n)}(\sigma).
$$

\noindent Now for $f_c$:
\begin{align*}   
   f_c(\tau\sigma)&=2\sum\limits_{i=1}^\infty(i-c)\1_{\{(\tau\sigma)_i=-1\}}-2\sum\limits_{i=-\infty}^0(i-c)\1_{\{(\tau\sigma)_i=1\}}\\
   &=2\sum\limits_{i=1}^\infty(i-c)\1_{\{\sigma_{i+1}=-1\}}-2\sum\limits_{i=-\infty}^0(i-c)\1_{\{\sigma_{i+1}=1\}}\\
   &=2\sum\limits_{i=1}^\infty(i-1-c)\1_{\{\sigma_i=-1\}}+2c\1_{\{\sigma_1=-1\}}-2\sum\limits_{i=-\infty}^0(i-1-c)\1_{\{\sigma_i=1\}}+2c\1_{\{\sigma_1=1\}}\\
   &=2c + 2\sum\limits_{i=1}^\infty(i-c)\1_{\{\sigma_i=-1\}}-2\sum\limits_{i=-\infty}^0(i-c)\1_{\{\sigma_i=1\}}-2\sum\limits_{i=1}^\infty \1_{\{\sigma_i=-1\}}+2\sum\limits_{i=-\infty}^0\1_{\{\sigma_i=1\}}\\
   &=2c-2N(\sigma)+f_c(\sigma),
\end{align*}
and for $N$:
\begin{align*}
   N(\tau\sigma)&=\sum\limits_{i=1}^\infty \1_{\{(\tau\sigma)_i=-1\}}-\sum\limits_{i=-\infty}^0\1_{\{(\tau\sigma)_i=1\}}\\
   &=\sum\limits_{i=1}^\infty \1_{\{\sigma_{i+1}=-1\}}-\sum\limits_{i=-\infty}^0\1_{\{\sigma_{i+1}=1\}}\\
   &=\sum\limits_{i=1}^\infty \1_{\{\sigma_i=-1\}}-\1_{\{\sigma_1=-1\}}-\sum\limits_{i=-\infty}^0\1_{\{\sigma_i=1\}}-\1_{\{\sigma_1=1\}}\\
   &=N(\sigma)-1.
\end{align*}
If we iterate these relations we have
\begin{align*}
f_c(\tau^n\sigma) =& 2nc - 2nN(\sigma) +n(n-1) + f_c(\sigma),
\\
N(\tau^n\sigma) =& N(\sigma) - n.
\end{align*}
Using these identities we find that, for $n\in\Z$, 
$
    \mu_{J}^c(\{N=n\}) = q^{-2nc + n(n+1)}\mu_{J}^c\Big(\1_{\{N=0\}}e^{-\beta H^{(n)}_{J}}\Big).
$
\end{proof}
Using this lemma for our two cases of interest gives an expression for $\nu^n_{J,\beta,q}(\sigma)$ that we will work with.
\begin{enumerate}
\item $J(i)\equiv 1$.  
In this case $H^{(n)}(\sigma)= 0$ and thus we can cancel a factor of $Z_{\beta,q,c}\mu^c(\{N=0\})$ from numerator and denominator to give
\begin{equation}
\mu^c(\{N=n\})=\frac{q^{n(n+1)-2nc}}{\sum\limits_{m=-\infty}^\infty q^{m(m+1)-2mc}}.
\end{equation}
\par This then gives that,
\begin{equation}
    \nu^n_{\beta,q}(\sigma)=\frac{\sum\limits_{m\in\mathbb Z}q^{m(m+1)-2mc}\cdot e^{-\beta H(\sigma)}q^{f_c(\sigma)}}{q^{n(n+1)-2nc}Z_{\beta,q,c}}\1_{\{N(\sigma)=n\}}.
\end{equation}
It remains to find an expression for $Z_{\beta,q,c}$.

\item $J(i)=i$. 
Recalling $H(\sigma)= \sum_{i\in \Z} \1_{\{\sigma_i\neq \sigma_{i+1}\}}$, the hamiltonian for the case $J(i)\equiv 1$, our identity gives us
\begin{equation}
\mu^c_{J}(\{N=n\}) = q^{-2nc+n(n+1)} \mu^c_{J}\big(\1_{\{N=0\}}e^{-\beta n H}\big).
\end{equation}
This then gives that, 
\begin{equation}
    \nu^n_{J,\beta,q}(\sigma)=\frac{e^{-H_J(\sigma)}q^{f_c(\sigma)}}{q^{n(n+1)-2nc}\mu^c_J(\1_{\{N=0\}}e^{-\beta nH})Z^J_{\beta,q,c}}.
\end{equation}

\end{enumerate} 

\section{Transferring dynamics to a family nearest neighbour particle systems}~
\par We will now introduce a family of nearest neighbour interacting particle systems that are equivalent to the Ising process on $\mathbb{Z}$ under Kawasaki dynamics. We transfer the dynamics of the Ising chain on $\mathcal{B}_n$ to that of an interacting particle system on
\begin{equation}
\Omega\defeq\{\omega \in \mathbb{Z}_{\geq 0}^{\mathbb{Z}_{<0}} : \exists N > 0 \hspace{2mm}\textrm{ s.t } \omega_{-i}=0 \hspace{2mm} \forall i \geq N\},
\end{equation}
using a standing up map (as in the paper of Bal\'azs and Bowen \cite{blocking}). By doing this we can obtain an alternative characterisation of the stationary measure $\nu^n_{J,\beta,q}$ and equating the two expressions will lead us to probabilistic proofs of identities of combinatorial significance (Section \ref{sec: identities}).
    \begin{defn}\label{def: stand up}
    Given $\sigma\in\mathcal{B}_n$, let $S_r(\sigma)$ be the site of the $r^\text{th}$ positive spin, counted from the left. The \textbf{standing up map} $T^n:\mathcal{B}_n\rightarrow \Omega$ is defined by, $T^n(\sigma)=\omega$ where for $r\in\mathbb{Z}_{>0}$, $\omega_{-r}=S_{r+1}(\sigma)-S_r(\sigma)-1$. In other words, the number of particles at site $-r$ in $\omega$ is equal to the number of negative spins between the $r^\text{th}$ and $(r+1)^\text{th}$ positive spin in $\sigma$.
    \end{defn}
    \begin{figure}[H]
    \centering
    \begin{subfigure}[b]{0.4\textwidth}
    \centering
    \begin{tikzpicture}[scale=0.7]
    \draw[thick, <->] (-6,0.8)--(6,0.8);
\foreach \x in {-5,-4,-3,-2,-1,0,1,2,3,4,5}
    \draw[thick, -](\x cm, 0.9)--(\x cm, 0.7) node[anchor=north]{$\x$};
 
 \filldraw [black] (-5.5,1.3) circle (1pt);
\filldraw [black] (-5.7,1.3) circle (1pt);
\filldraw [black] (-5.9,1.3) circle (1pt);   
\node (a) at (-5,1) [label=\textbf{--}]{};
\node (a) at (-4,1) [label=\textbf{+}]{};
\node (a) at (-4.35,1.15) [label=\tiny{1}]{};
\node (a) at (-3,1) [label=\textbf{--}]{};
\node (a) at (-2,1) [label=\textbf{+}]{};
\node (a) at (-2.35,1.15) [label=\tiny{2}]{};
\node (a) at (-1,1) [label=\textbf{--}]{};
\node (a) at (0,1) [label=\textbf{--}]{};
\node (a) at (1,1) [label=\textbf{+}]{};
\node (a) at (0.65,1.15) [label=\tiny{3}]{};
\node (a) at (2,1) [label=\textbf{+}]{};
\node (a) at (1.65,1.15) [label=\tiny{4}]{};
\node (a) at (3,1) [label=\textbf{--}]{};
\node (a) at (4,1) [label=\textbf{+}]{};
\node (a) at (3.65,1.15) [label=\tiny{5}]{};
\node (a) at (5,1) [label=\textbf{+}]{};
\node (a) at (4.65,1.15) [label=\tiny{6}]{};
 \filldraw [black] (5.5,1.3) circle (1pt);
\filldraw [black] (5.7,1.3) circle (1pt);
\filldraw [black] (5.9,1.3) circle (1pt); 
    \end{tikzpicture}
    \caption{Our starting Ising configuration, $\sigma$.}
    \end{subfigure}
    \hfill
       \begin{subfigure}[b]{0.4\textwidth}
       \centering
    \begin{tikzpicture}[scale=0.7]
    \draw[thick, <-] (-7,-1)--(0,-1);
\foreach \x in {-6,-5,-4,-3,-2,-1}
    \draw[thick, -](\x cm, -1.1)--(\x cm, -0.9) node[anchor=north]{$\x$};
    
\filldraw [black] (-2,-0.5) circle (4pt);
\filldraw [black] (-2,0) circle (4pt);
\filldraw [black] (-1,-0.5) circle (4pt);
\filldraw [black] (-4,-0.5) circle (4pt);
 \filldraw [black] (-6.5,-0.7) circle (1pt);
\filldraw [black] (-6.7,-0.7) circle (1pt);
\filldraw [black] (-6.9,-0.7) circle (1pt); 
    \end{tikzpicture}
    \caption{The resulting stood up configuration, $\omega$.}
    \end{subfigure}
    \caption{An example of the bijection $T^{-1}$.}
    \label{fig: stand up asep holes}
\end{figure}
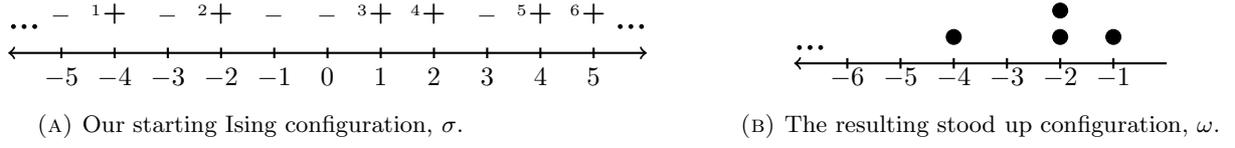

    \begin{lem}\label{lem: T^n bijection}
    For any $n\in\mathbb{N}$ the map $T^n:\mathcal{B}_n\rightarrow \Omega$ is a bijection.    
    \end{lem}
    \begin{proof}
       It is clear that the standing up map $T^n$ is an injection into $\mathbb{Z}_{\geq 0}^{\mathbb{Z}_{<0}}$. Since $\sigma_i=1$ for $i$ large enough, $T^n(\sigma)$ must coincide far to the left with $\omega^0$ which is such that $\omega_{-i}^0=0$ for all $i>0$. So the image of $T^n$ lies in, $$\Omega\defeq\{\omega \in \mathbb{Z}_{\geq 0}^{\mathbb{Z}_{<0}} : \exists N > 0 \hspace{2mm}\textrm{ s.t } \omega_{-i}=0 \hspace{2mm} \forall
    i \geq N\}.$$
    \par Now take some $\omega\in\Omega$ and some $n \in \mathbb{Z}$. We then construct the state $\sigma \in \mathcal{B}_n$ with leftmost positive spin at site $S^{(n)}_1(\omega)=n+1-\sum\limits_{i=1}^{\infty}\omega_{-i}$ and $r^\textrm{th}$ positive spin at site $S^{(n)}_r(\omega)=S_{r-1}(\omega)+\omega_{-(r-1)}+1$. In other words the $r^\text{th}$ positive spin is at site $S^{(n)}_r(\omega)=n+r-\sum\limits_{i=r}^\infty\omega_{-i}$ for any $r>0$. It is clear by construction that $T^n(\sigma)=\omega$. And so $T^n: \mathcal{B}_n \rightarrow \Omega$ is a bijection for any $n\in\mathbb{Z}$.
    \end{proof}
The dynamics for this corresponding particle system are inherited from the Kawasaki dynamics on $\mathcal{B}_n$. Note that in general the dynamics of the process depend on $n$. This means that for a given interaction function $J(i)$ in the Ising interaction we have a family of stood up processes, one for each $n\in\mathbb{Z}$. The jump rates for the stood up process, with a fixed value of $n$, are given in Tables \ref{table: bulk rates right}, \ref{table: bulk rates left} and \ref{table: boundary rates} below.  Where $S^{(n)}_r(\omega)$ is the function defined in the proof above. 
\begin{table}[H]
\centering
    \begin{tabular}{|c||c|c|}
         \hline
         & $\omega_{-r+1}=0$ & $\omega_{-r+1}\geq1$ \\
         \hline \hline
         $\omega_{-r}=0$ &0&0 \\
         \hline
        $\omega_{-r}=1$ &$\frac{1}{2}\left(1+\tanh(\beta\frac{J(S^{(n)}_{r}(\omega)+1)-J(S^{(n)}_r(\omega)-1)}{2})\right)q^{-1}$ &$\frac{1}{2}\left(1+\tanh(\beta\frac{J(S^{(n)}_r(\omega)-1)+J(S^{(n)}_r(\omega)+1)}{2})\right)q^{-1}$ \\
        \hline
        $\omega_{-r}\geq2$ &$\frac{1}{2}\left(1-\tanh(\beta\frac{J(S^{(n)}_r(\omega)-1)+J(S^{(n)}_r(\omega)+1)}{2})\right)q^{-1}$ &$\frac{1}{2}\left(1-\tanh(\beta\frac{J(S^{(n)}_r(\omega)+1)-J(S^{(n)}_r(\omega)-1)}{2})\right)q^{-1}$ \\
        \hline
    \end{tabular}
\caption{The right jump rates over $(-r,-r+1)$ for $r\geq2$ of the stood up process.}
\label{table: bulk rates right}
\end{table}
\begin{table}[H]
    \centering
     \begin{tabular}{|c||c|c|c|}
         \hline
        &$\omega_{-r+1}=0$ & $\omega_{-r+1}=1$ & $\omega_{-r+1}\geq2$\\
         \hline \hline
        $\omega_{-r}=0$ &0 &$\frac{1}{2}\left(1-\tanh(\beta\frac{J(S^{(n)}_r(\omega))-J(S^{(n)}_{r-1}(\omega))}{2})\right)q$ & $\frac{1}{2}\left(1-\tanh(\beta\frac{J(S^{(n)}_r(\omega)-2)+J(S^{(n)}_r(\omega))}{2})\right)q$\\
        \hline
        $\omega_{-r}\geq1$ &0 &$\frac{1}{2}\left(1+\tanh(\beta\frac{J(^{(n)}S_r(\omega)-2)+J(S^{(n)}_r(\omega))}{2})\right)q$  &  $\frac{1}{2}\left(1+\tanh(\beta\frac{J(S^{(n)}_r(\omega))-J(S^{(n)}_r(\omega)-2)}{2})\right)q$\\
        \hline
    \end{tabular}
    \caption{The left jump rates over $(-r,-r+1)$ for $r\geq2$ of the stood up process.}
    \label{table: bulk rates left}
\end{table}
\begin{table}[H]
    \centering
    \begin{tabular}{|c||c|c|}
         \hline
         &  Rate into the boundary &  Rate out of the boundary \\
         \hline \hline
          $\omega_{-1}=0$&0 & $\frac{1}{2}\left(1-\tanh(\beta\frac{J(S^{(n)}_1(\omega)-2)+J(S^{(n)}_1(\omega))}{2})\right)q$  \\
         \hline
       \multicolumn{1}{|c||}{ $\omega_{-1}=1$} & \multicolumn{1}{|c|}{$\frac{1}{2}\left(1+\tanh(\beta\frac{J(S^{(n)}_1(\omega)-1)+J(S^{(n)}_1(\omega)+1)}{2})\right)q^{-1}$} & \multirow{2}{*}{$\frac{1}{2}\left(1+\tanh(\beta\frac{J(S^{(n)}_1(\omega))-J(S^{(n)}_1(\omega)-2)}{2})\right)q$} \\ \cline{1-2}
        \multicolumn{1}{|c||}{$\omega_{-1}\geq2$} & \multicolumn{1}{|c|}{$\frac{1}{2}\left(1-\tanh(\beta\frac{J(S^{(n)}_1(\omega)+1)-J(S^{(n)}_1(\omega)-1)}{2})\right)q^{-1}$} &\\
        \hline
    \end{tabular}
    \caption{Boundary jump rates for the stood up process.}
    \label{table: boundary rates}
\end{table}
\begin{rem}
    When we consider the Ising chain with homogeneous interaction, $J(i)\equiv 1$, the dynamics for the stood up process are independent of $n$ and so we have a single particle system rather than a family of systems.
\end{rem}

\begin{prop}\label{prop: stationary dist stood up}
 Given $n\in\mathbb{Z}$, the unique stationary measure for the process on $\Omega$ with the dynamics given in Tables \ref{table: bulk rates right}, \ref{table: bulk rates left} and \ref{table: boundary rates} is given by, 
$$\pi_J^{(n)}(\omega)=\frac{e^{-\beta(\sum\limits_{j=1}^\infty (J(S^{(n)}_j(\omega))+J(S^{(n)}_{j+1}(\omega)-1))\1\{\omega_{-j}>0\}+J(S^{(n)}_1(\omega)-1))}q^{2\sum\limits_{i=1}^\infty i\omega_{-i}}}{\sum\limits_{z\in\Omega}e^{-\beta(\sum\limits_{j=1}^\infty (J(S^{(n)}_j(z))+J(S^{(n)}_{j+1}(z)-1))\1\{z_{-j}>0\}+J(S^{(n)}_1(z)-1))}q^{2\sum\limits_{i=1}^\infty iz_{-i}}}$$
\end{prop}
\begin{proof}
  It is easy to check that $\pi_J^{(n)}$ satisfies detailed balance both in the bulk and over the boundary edge for the process with the jump rates given in Tables \ref{table: bulk rates right},\ref{table: bulk rates left} and \ref{table: boundary rates}. The conditions given in Lemma \ref{lem:nearestneighbourconcentration} give that $\pi_J^{(n)}$ concentrates on $\Omega$.
 \end{proof}    
\begin{rem}\label{rem: pi} In terms of our two $J(i)$'s of interest we have:
 \begin{itemize}
     \item If $J(i)\equiv 1$ then the dynamics are independent of $n$ and for any $n\in\mathbb{Z}$, $\pi^{(n)}\equiv \pi$ where,
 $$\pi(\omega)=\prod\limits_{i=1}^\infty \frac{e^{-2\beta\1\{\omega_{-i}>0\}}q^{2i\omega_{-i}}}{\sum\limits_{z=0}^\infty e^{-2\beta\1\{z>0\}}q^{2iz}}=\prod\limits_{i=1}^\infty \frac{e^{-2\beta\1\{\omega_{-i}>0\}}q^{2i\omega_{-i}}}{\frac{e^{-2\beta}}{1-q^{2i}}-e^{-2\beta}+1}=\prod\limits_{i=1}^\infty \frac{e^{-2\beta\1\{\omega_{-i}>0\}}q^{2i\omega_{-i}}(1-q^{2i})}{1+(e^{-2\beta}-1)q^{2i}}.$$ 
 \item If $J(i)=i$ then we can write, 
 \begin{align*}
 \pi_J^{(n)}(\omega)&=\frac{e^{-\beta(\sum\limits_{j=1}^\infty ((2(n+j-\sum\limits_{i=j+1}^\infty\omega_{-i})-\omega_{-j})\1\{\omega_{-j}>0\})+n-\sum\limits_{i=1}^\infty \omega_{-i})}q^{2\sum\limits_{i=1}^\infty i\omega_{-i}}}{\sum\limits_{z\in\Omega}e^{-\beta(\sum\limits_{j=1}^\infty ((2(n+j-\sum\limits_{i=j+1}^\infty z_{-i})-z_{-j})\1\{z_{-j}>0\})+n-\sum\limits_{i=1}^\infty z_{-i})}q^{2\sum\limits_{i=1}^\infty i z_{-i}}}\\
 \\
&=\frac{\prod\limits_{i=1}^\infty e^{-2\beta\1_{\{\omega_{-i}>0\}}(n+i-\sum\limits_{j=i}^\infty \omega_{-j})}q^{2i\omega_{-i}}}{\sum\limits_{z\in\Omega}\prod\limits_{i=1}^\infty e^{-2\beta\1_{\{z_{-i}>0\}}(n+i-\sum\limits_{j=i}^\infty z_{-j})}q^{2iz_{-i}}}.
 \end{align*}
 \end{itemize}    
 \end{rem}

\section{Identities found by equating measures for the corresponding families}\label{sec: identities}
\par For any $n\in\mathbb{Z}$, the map $T^n$ (Definition \ref{def: stand up}) is a bijection between states in $\mathcal{B}_n$ and states in $\Omega$. This gives an equivalence of the two processes and thus an equivalence of stationary measures. That is, for any $\sigma \in \mathcal{B}_n$ such that $T^n(\sigma)=\omega\in\Omega$ we have $\nu_{J,\beta,q}^n(\sigma)=\pi_J^{(n)}(\omega)$. 
\par We can consider a convenient state for both the Ising process and the equivalent particle system. First let us fix an $n\in\mathbb{Z}$. For the Ising process on $\mathcal{B}$, convenient states will be those states where all the positive spins are as far to the right as possible. On $\mathcal{B}_n$ we take, $\sigma^n_i=\begin{cases}
    -1 &\text{ if } i\leq n\\
    1 &\text{ if } i>n
\end{cases}$ . 
We have that $f_c(\sigma^n)=n(n+1)-2nc$ and $H_J(\sigma^n)=J(n)$. This means that
\begin{equation}
\mu^c_J(\sigma^n) = \frac{1}{Z^{J}_{\beta,q,c}} e^{-\beta J(n)}q^{n(n+1)-2nc},
\end{equation}
and so, 
\begin{equation}
\nu^n_{J,\beta,q}(\sigma^n)=\frac{e^{-\beta J(n)}q^{n(n+1)-2nc}}{Z^J_{\beta,q,c}\mu^c_J(\{N=n\})}.
\end{equation}

\par \noindent Under the bijection $T^n(\sigma^n)=\omega^0$, where $\omega^0_{-i}=0$ for all $i>0$ is the empty state. For this state, we have that $S_1^{(n)}(\omega^0)=n+1$ and so, 
$$\pi_J^{(n)}(\omega^0)=\frac{e^{-\beta J(n)}}{\sum\limits_{z\in\Omega}e^{-\beta(\sum\limits_{j=1}^\infty (J(S^{(n)}_j(z))+J(S^{(n)}_{j+1}(z)-1))\mathbb{I}\{z_{-j}>0\}+J(S^{(n)}_1(z)-1))}q^{2\sum\limits_{i=1}^\infty iz_{-i}}}.$$ 
This equivalence of measures evaluated at the states $\sigma^n$ and $\omega^0$, gives us the following identity, 
\begin{equation}\label{eq: general identity}
   \frac{e^{-\beta J(n)}q^{n(n+1)-2nc}}{Z^J_{\beta,q,c}\mu^c_J(\{N=n\})} =\frac{e^{-\beta J(n)}}{\sum\limits_{z\in\Omega}e^{-\beta(\sum\limits_{j=1}^\infty (J(S^{(n)}_j(z))+J(S^{(n)}_{j+1}(z)-1))\mathbb{I}\{z_{-j}>0\}+J(S^{(n)}_1(z)-1))}q^{2\sum\limits_{i=1}^\infty iz_{-i}}}.
\end{equation}
\par By considering $J(i)$ functions of interest we obtain two identities for further study.
\begin{itemize}
    \item For $J(i)\equiv 1$,
    \begin{equation}\label{eq: J=1 iden}
        \frac{e^{-\beta}\sum\limits_{m\in\mathbb{Z}}q^{m(m+1)-2mc}}{Z_{\beta,q,c}}=\prod\limits_{i=1}^\infty\frac{1-q^{2i}}{1+(e^{-2\beta}-1)q^{2i}}.
    \end{equation}
    \item For $J(i)=i$,
    \begin{equation}\label{eq: J=i iden}
        \frac{e^{-\beta n}}{Z^J_{\beta,q,c}\mu^c_J(\1_{\{N=0\}}e^{-n\beta H})}=\frac{1}{\sum\limits_{z\in\Omega}\prod\limits_{i=1}^\infty e^{-2\beta\1_{\{z_{-i}>0\}}(n+i-\sum\limits_{j=i}^\infty z_{-j})}q^{2iz_{-i}}}.
    \end{equation}
\end{itemize}

\noindent We need to better understand $Z_{\beta,q,c}$ and $Z^J_{\beta,q,c}\mu^c_J(\1_{\{N=0\}}e^{-n\beta H})$ for $J(i)=i$, we will study these in the next section.
\begin{rem}
    We see that for any $n\in\mathbb{Z}$, when we take $J(i)\equiv 1$ we get the same identity. This is not surprising since the dynamics of the equivalent particle system in this case are independent of $n$.
\end{rem}

\subsection{An expression for the partition functions in terms of runs of consecutive spins}~ \label{sec: ising partitions function}

\par In  this section we will find two expressions for $\mu^c_{J}(\{N=n\})$ in the case $J(i)=i$ (see Corollaries \ref{cor:Jiexpression1} and \ref{cor:Jiexpression2}), and an expression $Z^{J}_{\beta,q,c}$ in the case $J(i)=1$ (see Corollory \ref{cor:J1expression1}), in terms of the runs of consecutive negative spins to the right of 1 and the runs of consecutive positive spins to the left of 0. 
The first expression for $\mu^c_{J}(\{N=n\})$ will be derived from the expression in Lemma \ref{lem:shift} that made use of the shift operator and be used to derive the expression for $Z_{\beta,q,c}$ in Corollary \ref{cor:partn}. The second expression will be derived directly from $\mu^c_{J}(\{N=n\})$ using the same method and be more convenient for our combinatorial identities.

We first define quantities that appear in these expressions that are dependent on $J$. Define $Q=q^2$ and take $\ell_1,\dots,\ell_L,m_1,\dots,m_R\geq 1$ and $n\in\N_0$ fixed. Let us write $A^{(n)}_0=B^{(n)}_0=a^{(n)}_0=b^{(n)}_0=1$ and for $R,L>0$
\begin{align}
    A^{(n)}_R &:= \sum_{\substack{s_1,\dots,s_{R}\geq 1 \\ s_{i+1}-s_i \geq m_i+1}}\prod_{j=1}^RQ^{s_jm_j}(e^{-\beta})^{J(s_j+n-1)+J(s_j+m_j+n-1)}, \label{eq:AR}
    \\
    B^{(n)}_L &:= \sum_{\substack{r_1,\dots,r_L\leq 0 \\ r_{i+1}-r_i \leq -\ell_i -1}}\prod_{j=1}^{L}Q^{-r_j \ell_j}(e^{-\beta})^{J(r_j+n)+J(r_j-\ell_j+n)}, \label{eq:BL}
    \\
    a^{(n)}_R &:= \sum_{\substack{s_1,\dots,s_{R}\geq 1 \\ s_{i+1}-s_i \geq m_i+1}}\1_{\{s_1>1\}}\prod_{j=1}^RQ^{s_jm_j}(e^{-\beta})^{J(s_j+n-1)+J(s_j+m_j+n-1)}, \label{eq:aR}
    \\
    b^{(n)}_L &:= \sum_{\substack{r_1,\dots,r_L\leq 0 \\ r_{i+1}-r_i \leq -\ell_i -1}}\1_{\{r_1<0\}}\prod_{j=1}^{L}Q^{-r_j \ell_j}(e^{-\beta})^{J(r_j+n)+J(r_j-\ell_j+n)}.    \label{eq:bl}
\end{align}
For some choices of the function $J:\Z\to\R$ these sums are explicitly computable in terms of elementary functions, for example the case $J(i)=1$ and $J(i)=i$ considered below. Considering different choices of $J$ will lead to further combinatorial identities than the ones considered here, however in many cases the sums are either not explicitly computable or extremely cumbersome.

\begin{lem}\label{lem:Zexpression}
Take $c\in \R$, $\beta\geq0$, $q\in(0,1)$, $Q=q^2$, and $J:\Z\to\R$ satisfying the summability condition in Lemma \ref{lem:nearestneighbourconcentration}. Let $A^{(n)}_R,B^{(n)}_L,a^{(n)}_R,b^{(n)}_L$ be defined as in equations \eqref{eq:AR}-\eqref{eq:bl}. For any $n\in \Z$ we have
\begin{equation*}
    \begin{aligned}
  Z^{J}_{\beta,q,c}\mu_{J}^c\hspace{-1pt}\Big(\1_{\{N=0\}}e^{-\beta H^{(n)}_{J}}\hspace{-2pt}\Big)      
  &= e^{-\beta J(n)} + \sum_{L,R>0}\sum_{\substack{\ell_1,\dots,\ell_L\geq 1 \\ m_1,\dots,m_{R}\geq 1}}\1_{\big\{\sum_{j=1}^R m_j= \sum_{j=1}^L\ell_j\big\}}
  \\
  &\prod_{j=1}^L Q^{\tfrac12 \ell_j(\ell_j-1)}\prod_{j=1}^{R}Q^{\tfrac12 m_j(m_j-1)}\Big(e^{\beta J(n)}A^{(n)}_RB^{(n)}_L + \big(e^{-\beta J(n)}-e^{\beta J(n)}\big)a^{(n)}_Rb^{(n)}_L\Big).
    \end{aligned}
\end{equation*}
\end{lem}

\begin{proof}
For a configuration $\sigma\in \Bcal$ we consider the lengths $m_1,m_2,\dots,m_R$ of runs of consecutive negative spins to the right of 0. Here $m_1$ is the length of the left most run, $m_2$ the length of next run to the right and so on. 
We denote by $s_i$ the left-most vertex of the $i^{th}$ run of consecutive negative spins (the run of length $m_i$). Notice that we require $s_{i+1}-s_i\geq m_i+1$ for $i=1,\dots,R-1$, and $s_1\geq 1$ in order for these runs to be separate (i.e. separated by at least one spin with value +1). 
We can similarly consider the runs of negative spins to the left of 0 of length $\ell_1,\ell_2,\dots$ and with right-most vertices $r_1,r_2,\dots$.

A run of negative spins right of 1 or a run of positive spins left of 0 with left-most vertex $a$ and right-most vertex $b$ contributes a factor $(e^{-\beta})^{J_{a-1}+J_b}$ to $e^{-\beta H_J(\sigma)}$, with the possible exception of the run closest to 0, here the cases $s_1=1$ and $r_1=0$ must be taken into account to consider whether or not $\sigma_0=\sigma_1$. 
The $i^{th}$ run of negative spins right of 1 also contributes a factor $q^{2\big(\tfrac12 m_i(m_i-1)-c m_j +s_jm_j\big)}$ to the function $f_c(\sigma)$. There are corresponding factors for runs of positive spins left of 0.

The configuration $\tilde\sigma$ where $\tilde\sigma_i = (-1)^{\1_{\{i\leq 0\}}}$ has no consecutive runs of negative spins right of 1 or positive spins left of 0. Its measure is $\mu^c_J(\sigma)=\tfrac{e^{-\beta J(0)}}{Z^J_{\beta,q,c}}$. All other configurations have at least one such run of consecutive spins on each side of $0$ (we cannot have runs of spins on only one side of 0 as then it is not possible to satisfy $N=0$). We can now write an expression for $\mu^c_{J}(\{N=n\})$. We denote by $L$ the number of runs of consecutive positive spins to the left of 0 and by $R$ the number of runs of consecutive negative spins to the right of $0$. Writing $Q=q^2$ and using the considerations above we have the general expression
\begin{equation}
\begin{aligned}
    &Z^{J}_{\beta,q,c}\mu_{J}^c\Big(\1_{\{N=0\}}e^{-\beta H^{(n)}_{J}}\Big)=e^{-\beta J(n)} + \sum_{L,R>0}\sum_{\substack{\ell_1,\dots,\ell_L\geq 1 \\ m_1,\dots,m_{R}\geq 1}}\sum_{\substack{r_1,\dots,r_L\leq 0 \\ r_{i+1}-r_i \leq -\ell_i -1}}\sum_{\substack{s_1,\dots,s_{R}\geq 1 \\ s_{i+1}-s_i \geq m_i+1}}\1_{\big\{\sum_{j=1}^R m_j= \sum_{j=1}^L\ell_j\big\}}
    \\
    & \prod_{j=1}^L Q^{\tfrac12 \ell_j(\ell_j-1)+c\ell_j - r_j\ell_j}(e^{-\beta})^{J(r_j+n)+J(r_j-\ell_j+n)}\prod_{j=1}^{R}Q^{\tfrac12 m_j(m_j-1)-cm_j + s_j m_j}(e^{-\beta})^{J(s_j+n-1)+J(s_j+m_j+n-1)}
        \\
    & 
\quad\qquad\qquad\qquad\qquad\qquad\qquad\Big(\big(e^{\beta J(n)}-e^{-\beta J(n)}\big)\big(\1_{\{r_1<0,s_1=1\}}\hspace{-2pt}+\hspace{-2pt}\1_{\{r_1=0,s_1>1\}}\hspace{-2pt}+\hspace{-2pt}\1_{\{r_1=0,s_1=1\}}\big)+
e^{-\beta J(n)}\Big).
\end{aligned}     
\end{equation}
The factors of $Q^c$ cancel due to the condition $\sum_{j=1}^R m_j= \sum_{j=1}^L\ell_j$.
Then our expression is
\begin{equation}
\begin{aligned}
&Z^{J}_{\beta,q,c}\mu_{J}^c\hspace{-1pt}\Big(\1_{\{N=0\}}e^{-\beta H^{(n)}_{J}}\hspace{-2pt}\Big)=e^{-\beta J(n)} \hspace{-2pt}+\hspace{-2pt} \sum_{L,R>0}\sum_{\substack{\ell_1,\dots,\ell_L\geq 1 \\ m_1,\dots,m_{R}\geq 1}}\hspace{-2pt}\1_{\big\{\sum_{j=1}^R m_j= \sum_{j=1}^L\ell_j\big\}}
\\
&\quad\qquad\qquad\qquad\qquad\qquad\prod_{j=1}^L Q^{\tfrac12 \ell_j(\ell_j-1)}\prod_{j=1}^{R}Q^{\tfrac12 m_j(m_j-1)}\Big(e^{\beta J(n)}A^{(n)}_RB^{(n)}_L + \big(e^{-\beta J(n)}-e^{\beta J(n)}\big)a^{(n)}_Rb^{(n)}_L\Big).
\end{aligned}     
\end{equation}
\end{proof}
We now look at our $J(i)$'s of interest.

\begin{cor}\label{cor:J1expression1}
Take $c\in \R$, $\beta\geq0$, $q\in(0,1)$, $Q=q^2$, $y=e^{-2\beta}$ and $J:\Z\to\R$ given by $J(i)=1$ for every $i\in \Z$. We have
\begin{align*}
Z_{\beta,q,c} &= y^{1/2} + \sum_{\substack{L,R\geq0\\ L+R>0}}y^{L+R} \sum_{\substack{\ell_1,\dots,\ell_L\geq 1 \\ m_1,\dots,m_{R}\geq 1}}\prod_{j=1}^{L} \frac{Q^{\tfrac12 \ell_j(\ell_j-1)+c\ell_j+j\ell_j+\ell_{j-1}(\ell_{j}+\cdots+\ell_{L})}}{1-Q^{\ell_{j}+\cdots + \ell_{L}}} 
\\
&\qquad\prod_{j=1}^{R}\frac{Q^{\tfrac12 m_j(m_j-1)-cm_j+jm_j+m_{j-1}(m_{j}+\cdots+m_{R})}}{1-Q^{m_{j}+\cdots+m_{R}}}\Big(y^{-1/2} + \big(y^{1/2}-y^{-1/2}\big)Q^{\ell_1+\cdots+\ell_L}Q^{m_1+\cdots+m_R}\Big).  
\end{align*}
\end{cor}
\begin{proof}
The case $J(i)\equiv 1$ is slightly simpler than the general case treated above. We need to find an expression for $Z^{J}_{\beta,q,c}$, this can be done in a very similar way to the expression in Lemma \ref{lem:Zexpression}. We no longer have the condition $\sum_{j=1}^R m_j= \sum_{j=1}^L\ell_j$ and hence also have an extra factor of $c\big(\sum_{j=1}^L\ell_j-\sum_{j=1}^R m_j\big)$ in the exponent of $Q$. In addition, we only require that $L+R>0$ in the first sum not that $L,R>0$, as the total length of runs to the left and right of 0 no longer need to balance.

Recall the quantities $A^{(n)}_R,B^{(n)}_L,a^{(n)}_R,b^{(n)}_L$ defined in equations \eqref{eq:AR}-\eqref{eq:bl}. The sum over $r_i$'s and $s_i$'s can be computed. For $\ell_1,\dots,\ell_L\geq 1$  fixed (and assuming $L>0$, otherwise $B^{(n)}_L=b^{(n)}_L=1$) we have 
\begin{equation}\label{eq:J1BL}
B^{(n)}_L=e^{-2\beta L}\sum_{\substack{r_1,\dots,r_L\leq 0 \\ r_{i+1}-r_i \leq -\ell_i -1}}\prod_{j=1}^L Q^{- r_j\ell_j} = e^{-2\beta L} \prod_{j=0}^{L-1} \frac{Q^{(\ell_{j}+1)(\ell_{j+1}+\cdots+\ell_{L})}}{1-Q^{\ell_{j+1}+\cdots + \ell_{L}}},
\end{equation}
where we define $\ell_0:=-1$. We also have 
\begin{equation}
\begin{aligned}
b^{(n)}_L=e^{-2\beta L}\sum_{\substack{r_1,\dots,r_L\leq 0 \\ r_{i+1}-r_i \leq -\ell_i -1}}\1_{\{r_1<0\}}\prod_{j=1}^L Q^{- r_j\ell_j} &=  e^{-2\beta L}\sum_{ r_{1} < 0}Q^{-r_1(\ell_1+\cdots+\ell_L)}\prod_{j=1}^{L-1} \frac{Q^{(\ell_{j}+1)(\ell_{j+1}+\cdots+\ell_{L})}}{1-Q^{\ell_{j+1}+\cdots + \ell_{L}}} 
\\
&= e^{-2\beta L}Q^{\ell_1+\cdots + \ell_L}\prod_{j=0}^{L-1} \frac{Q^{(\ell_{j}+1)(\ell_{j+1}+\cdots+\ell_{L})}}{1-Q^{\ell_{j+1}+\cdots + \ell_{L}}}. \label{eq:J1bL}
\end{aligned}
\end{equation}
We note that  $b^{(n)}_L=B^{(n)}_LQ^{\ell_1+\cdots + \ell_L}$.
Now for $m_1,\dots,m_{R}\geq 1$ (and assuming $R>0$, otherwise $A^{(n)}_R=a^{(n)}_R=1$) we have
\begin{equation}\label{eq:J1AR}
A^{(n)}_R=e^{-2\beta R}\sum_{\substack{s_1,\dots,s_{R}\geq 1 \\ s_{i+1}-s_i \geq m_i+1}}\prod_{j=1}^{R}Q^{s_j m_j} = e^{-2\beta R} \prod_{j=0}^{R-1}\frac{Q^{(m_{j}+1)(m_{j+1}+\cdots+m_{R})}}{1-Q^{m_{j+1}+\cdots+m_{R}}},
\end{equation}
where we define $m_0:=0$. With the indicators above we have
\begin{equation}
\begin{aligned}
a^{(n)}_R=e^{-2\beta R} \sum_{\substack{s_1,\dots,s_{R}\geq 1 \\ s_{i+1}-s_i \geq m_i+1}}\1_{\{s_1>1\}}\prod_{j=1}^{R}Q^{s_j m_j} & = e^{-2\beta R} \sum_{s_1> 1}Q^{s_1(m_1+\cdots+m_R)}\prod_{j=1}^{R-1}\frac{Q^{(m_{j}+1)(m_{j+1}+\cdots+m_{R})}}{1-Q^{m_{j+1}+\cdots+m_{R}}} 
\\
&= e^{-2\beta R}Q^{m_{1}+\cdots+m_{R}} \prod_{j=0}^{R-1}\frac{Q^{(m_{j}+1)(m_{j+1}+\cdots+m_{R})}}{1-Q^{m_{j+1}+\cdots+m_{R}}}. \label{eq:J1aR}
\end{aligned}
\end{equation}
We note that $a^{(n)}_R=A^{(n)}_RQ^{m_{1}+\cdots+m_{R}}$.
Putting this together gives
\begin{equation} \label{eq:Z2expression}
\begin{aligned}
Z_{\beta,q,c} 
= e^{-\beta} + \sum_{\substack{L,R\geq0\\ L+R>0}} \sum_{\substack{\ell_1,\dots,\ell_L\geq 1 \\ m_1,\dots,m_{R}\geq 1}}\prod_{j=1}^{L} Q^{\tfrac12 \ell_j(\ell_j-1)+c\ell_j} &\prod_{j=1}^{R}Q^{\tfrac12 m_j(m_j-1)-cm_j} 
\\
&A^{(n)}_RB^{(n)}_L\Big(e^\beta + \big(e^{-\beta}-e^\beta\big)Q^{\ell_1+\cdots+\ell_L+m_1+\cdots+m_R}\Big).
\end{aligned}    
\end{equation}
Now we insert our expression for $A^{(n)}_R$ and $B^{(n)}_L$ in \eqref{eq:J1AR} and \eqref{eq:J1BL} and also use that 
\begin{align*}
\sum_{j=1}^{L}(\ell_{j-1}+1)(\ell_{j}+\cdots+\ell_L)&=\sum_{j=1}^{L}j\ell_j+\sum_{j=1}^{L}\ell_{j-1}(\ell_{j}+\cdots+\ell_L),
\\
\sum_{j=1}^R(m_{j-1}+1)(m_j+\cdots + m_R)&= \sum_{j1}^Rjm_j + \sum_{j=1}^Rm_{j-1}(m_j+\cdots+m_R).
\end{align*}
Our final expression for $J(i)\equiv 1$, as claimed, is 
\begin{equation} \label{eq:Z3expression}
\begin{aligned}
Z_{\beta,q,c} &= e^{-\beta} + \sum_{\substack{L,R\geq0\\ L+R>0}}(e^{-\beta})^{2(L+R)} \sum_{\substack{\ell_1,\dots,\ell_L\geq 1 \\ m_1,\dots,m_{R}\geq 1}}\prod_{j=1}^{L} \frac{Q^{\tfrac12 \ell_j(\ell_j-1)+c\ell_j+j\ell_j+\ell_{j-1}(\ell_{j}+\cdots+\ell_{L})}}{1-Q^{\ell_{j}+\cdots + \ell_{L}}} 
\\
&\qquad\prod_{j=1}^{R}\frac{Q^{\tfrac12 m_j(m_j-1)-cm_j+jm_j+m_{j-1}(m_{j}+\cdots+m_{R})}}{1-Q^{m_{j}+\cdots+m_{R}}}\Big(e^\beta + \big(e^{-\beta}-e^\beta\big)Q^{\ell_1+\cdots+\ell_L}Q^{m_1+\cdots+m_R}\Big).
\end{aligned}    
\end{equation}
\end{proof}
\begin{cor}\label{cor:Jiexpression1}
Take $c\in \R$, $\beta\geq0$, $q\in(0,1)$, $Q=q^2$, $y=e^{-2\beta}$, and $J:\Z\to\R$ given by $J(i)=i$. We have
\begin{align*}
    &Z_{\beta,q,c}^{J}\mu_{J}^c(\1_{\{N=0\}}e^{-n\beta H_1}) = e^{-n\beta} + \sum_{L,R>0} y^{nL+(n-1)R}\sum_{\substack{\ell_1,\dots,\ell_L\geq 1 \\ m_1,\dots,m_{R}\geq 1}}\1_{\big\{\sum_{j=1}^R m_j= \sum_{j=1}^L\ell_j\big\}} 
    \\
    &\qquad\qquad\prod_{j=1}^{L} \frac{Q^{\tfrac12 \ell_{j}(\ell_{j}-1)}\big(y^{-(L-j+1)}Q^{\ell_{j}+\cdots + \ell_{L}}\big)^{\ell_{j-1}+1}}{1-y^{-(L-j+1)}Q^{\ell_{j}+\cdots + \ell_{L}}} \prod_{j=1}^{R}\frac{Q^{\tfrac12 m_{j}(m_{j}-1)}\big(y^{R-j+1}Q^{m_{j}+\cdots+m_{R}}\big)^{m_{j-1}+1}}{1-y^{R-j+1}Q^{m_{j}+\cdots+m_{R}}}
    \\
    &\qquad\qquad\qquad\qquad\qquad\qquad\qquad\qquad\qquad\qquad\bigg(y^{-n/2}+\big(y^{n/2}-y^{-n/2}\big)y^{R-L}Q^{\ell_1+\cdots + \ell_L}Q^{m_{1}+\cdots+m_{R}}\bigg). 
\end{align*}
\end{cor}

\begin{proof}
From Lemma \ref{lem:shift} we need to find an expression for $Z_{\beta,q,c}^{J}\mu_{J}^c(\1_{\{N=0\}}e^{-n\beta H_1})$. We first define $y=e^{-2\beta}$.
We can compute the sums over $r_j$'s and $s_j$'s in the $A^{(n)}_R,a^{(n)}_R,B^{(n)}_L,b^{(n)}_L$ terms  defined as in equations \eqref{eq:AR}-\eqref{eq:bl} so that we now need to compute the following sums.
\begin{equation}
  B^{(n)}_L=  \sum_{\substack{r_1,\dots,r_L\leq 0 \\ r_{i+1}-r_i \leq -\ell_i -1}}\prod_{j=1}^L (yQ^{-\ell_j})^{r_j}y^{n-\tfrac{\ell_j}{2}} = y^{Ln-\tfrac{1}{2}\sum_{j=1}^L\ell_j}\prod_{j=0}^{L-1} \frac{\big(y^{-(L-j)}Q^{\ell_{j+1}+\cdots + \ell_{L}}\big)^{\ell_{j}+1}}{1-y^{-(L-j)}Q^{\ell_{j+1}+\cdots + \ell_{L}}}
\end{equation}
where $\ell_0=-1$ as above. With the indicators above we have 
\begin{equation}
\begin{aligned}
 b^{(n)}_L  &=   \sum_{\substack{r_1,\dots,r_L\leq 0 \\ r_{i+1}-r_i \leq -\ell_i -1}}\1_{\{r_1<0\}}\prod_{j=1}^L (yQ^{-\ell_j})^{r_j}y^{n-\tfrac{\ell_j}{2}}  
 \\
 &= y^{Ln-\tfrac{1}{2}\sum_{j=1}^L\ell_j}y^{-L}Q^{\ell_{L}+\cdots + \ell_{1}}\prod_{j=0}^{L-1} \frac{\big(y^{-(L-j)}Q^{\ell_{j+1}+\cdots + \ell_{L}}\big)^{\ell_{j}+1}}{1-y^{-(L-j)}Q^{\ell_{j+1}+\cdots + \ell_{L}}}
 \\
 & = y^{-L}Q^{\ell_1+\cdots\ell_L}B^{(n)}_L.
 \end{aligned}
\end{equation}
We also have
\begin{equation}
A^{(n)}_R  =  \sum_{\substack{s_1,\dots,s_{R}\geq 1 \\ s_{i+1}-s_i \geq m_i+1}}\prod_{j=1}^{R}(yQ^{m_j})^{s_j} y^{n-1+\tfrac{m_j}{2}}= y^{R(n-1)+\tfrac{1}{2}\sum_{j=1}^Rm_j}\prod_{j=0}^{R-1}\frac{(y^{R-j}Q^{m_{j+1}+\cdots+m_{R}})^{m_{j}+1}}{1-y^{R-j}Q^{m_{j+1}+\cdots+m_{R}}}
\end{equation}
where $m_0=0$. With the indicators above we have 
\begin{equation}
\begin{aligned}
   a^{(n)}_R   =  &\sum_{\substack{s_1,\dots,s_{R}\geq 1 \\ s_{i+1}-s_i \geq m_i+1}}\1_{\{s_1>1\}}\prod_{j=1}^{R}(yQ^{m_j})^{s_j}y^{n-1+\tfrac{m_j}{2}} 
   \\
   = &y^{R(n-1)+\tfrac{1}{2}\sum_{j=1}^Rm_j}y^{R}Q^{m_{R}+\cdots+m_{1}}\prod_{j=0}^{R-1}\frac{(y^{j}Q^{m_{R}+\cdots+m_{R-j+1}})^{m_{R-j}+1}}{1-y^{j}Q^{m_{R}+\cdots+m_{R-j+1}}}
   \\
   & = y^RQ^{m_1+\cdots+m_R}A^{(n)}_R.
\end{aligned}
\end{equation}
So that our expression for the case $J(i)=i$ is
\begin{equation}
\begin{aligned}
    &Z_{\beta,q,c}^{J}\mu_{J}^c(\1_{\{N=0\}}e^{-n\beta H_1}) = e^{-n\beta} + \sum_{L,R>0} y^{nL+(n-1)R}\sum_{\substack{\ell_1,\dots,\ell_L\geq 1 \\ m_1,\dots,m_{R}\geq 1}}\1_{\big\{\sum_{j=1}^R m_j= \sum_{j=1}^L\ell_j\big\}} 
    \\
    &\qquad\qquad\prod_{j=1}^{L} \frac{Q^{\tfrac12 \ell_{j}(\ell_{j}-1)}\big(y^{-(L-j+1)}Q^{\ell_{j}+\cdots + \ell_{L}}\big)^{\ell_{j-1}+1}}{1-y^{-(L-j+1)}Q^{\ell_{j}+\cdots + \ell_{L}}} \prod_{j=1}^{R}\frac{Q^{\tfrac12 m_{j}(m_{j}-1)}\big(y^{R-j+1}Q^{m_{j}+\cdots+m_{R}}\big)^{m_{j-1}+1}}{1-y^{R-j+1}Q^{m_{j}+\cdots+m_{R}}}
    \\
    &\qquad\qquad\qquad\qquad\qquad\qquad\qquad\qquad\qquad\qquad\bigg(y^{-n/2}+\big(y^{n/2}-y^{-n/2}\big)y^{R-L}Q^{\ell_1+\cdots + \ell_L}Q^{m_{1}+\cdots+m_{R}}\bigg),
\end{aligned}    
\end{equation} 
recalling that we defined $\ell_0=-1$ and $m_0=0$.
\end{proof}
We now state the identities derived directly from $\mu_J^c(\{N=n\})$.
\begin{lem}\label{lem:Zexpression2}
Take $c\in \R$, $\beta\geq0$, $q\in(0,1)$, $Q=q^2$, and $J:\Z\to\R$. Let $A^{(n)}_R,B^{(n)}_L,a^{(n)}_R,b^{(n)}_L$ be defined as in equations \eqref{eq:AR}-\eqref{eq:bl}. For any $n\in \Z$ we have
\begin{equation*}
    \begin{aligned}
Z^{J}_{\beta,q,c}\mu_{J}^c\Big(N=n\Big) = &e^{-\beta J(0)} + \sum_{L,R>0}\sum_{\substack{\ell_1,\dots,\ell_L\geq 1 \\ m_1,\dots,m_{R}\geq 1}}\1_{\big\{\sum_{j=1}^R m_j - \sum_{j=1}^L\ell_j =n\big\}} Q^{-cn} 
\\
&\qquad\prod_{j=1}^L Q^{\tfrac12 \ell_j(\ell_j-1)}\prod_{j=1}^{R}Q^{\tfrac12 m_j(m_j-1)}\Big(e^{\beta J(0)}A^{(0)}_RB^{(0)}_L - \big(e^{\beta J(0)}-e^{-\beta J(0)}\big) a^{(0)}_Rb^{(0)}_L\Big).
    \end{aligned}
\end{equation*}
\end{lem}

\begin{cor}\label{cor:Jiexpression2}
Take $c\in \R$, $\beta\geq0$, $q\in(0,1)$, $Q=q^2$, $y=e^{-2\beta}$, and $J:\Z\to\R$ given by $J(i)=i$. We have
\small
\begin{align*}
& Z_{\beta,q,c}^{J}\mu_{J}^c(N=n) = 1 + \sum_{L,R>0} y^{\tfrac12 n}\sum_{\substack{\ell_1,\dots,\ell_L\geq 1 \\ m_1,\dots,m_{R}\geq 1}}\1_{\big\{\sum_{j=1}^R m_j- \sum_{j=1}^L\ell_j=n\big\}} Q^{-cn} y^{\tfrac12 R(R+1)-\tfrac12 L(L+1)}
\\
&\prod_{j=1}^{L} \frac{Q^{\tfrac12 \ell_{j}(\ell_{j}-1) + j\ell_j + \ell_{j-1}(\ell_j + \cdots + \ell_L)}y^{-(L+1)\ell_{j-1} + j\ell_{j-1}}}{1-y^{-(L-j+1)}Q^{\ell_{j}+\cdots + \ell_{L}}} 
\prod_{j=1}^{R}\frac{Q^{\tfrac12 m_{j}(m_{j}-1)+ jm_j+m_{j-1}(m_j+\cdots+m_R)}y^{(R+1)m_{j-1}-jm_{j-1}}}{1-y^{R-j+1}Q^{m_{j}+\cdots+m_{R}}}
\end{align*}
\end{cor}

\section{Combinatorial interpretation of identities}
We now look to give combinatorial interpretations to the identities found by equating measures for the family of Ising process and equivalent family of particle systems. We will consider the identities for our interaction functions of interest, $J(i)\equiv1$ and $J(i)=i$, and show that they can be viewed as equivalences of generating functions for certain integer and Frobenius partitions.

\subsection{Combinatorial Identity for $J(i)\equiv1$}\label{sec: comb j is 1}~

\par Let us first consider the case when $J(i)\equiv 1$, using equation \eqref{eq: J=1 iden} and Corollary \ref{cor:J1expression1}. If we let $Q\defeq q^2$, $z\defeq q^{-2c}$ and $y\defeq e^{-2\beta}$ and rearrange we have the following identity.
\constident*
\begin{rem} This holds as a formal identity. This can be seen from the combinatorial interpretation in terms of generating functions, discussed below. The probabilistic nature of the proof for Theorem \ref{thm: J is 1 identity} means it holds for $Q\in(0,1)$, $z>0$ and $y\in(0,1]$, but the identity should hold for any $Q,z,y\in\mathbb{R}$ with $z,y\neq 0$.
\end{rem}
Using Theorem \ref{thm: J is 1 identity} and the Jacobi triple product identity we can write the Ising partition function as a product. We note that this form seems non-obvious by just studying the normalising factor of the measure.
\partitionfunc*
\begin{proof}
Recall that the Jacobi triple product identity, for $Q\in(0,1)$ and $z\neq 0$,     $$\sum\limits_{m\in\mathbb{Z}}Q^{\frac{m(m+1)}{2}}z^m=\prod\limits_{i=1}^\infty (1-Q^i)(1+Q^i z)(1+Q^{i-1}z^{-1}).$$
By substituting this into Theorem \ref{thm: J is 1 identity} we have that,     $$Z(Q,z,y)=\prod\limits_{i=1}^\infty(1+(y-1)Q^i)(1+Q^iz)(1+Q^{i-1}z^{-1}).$$
Thus the normalising factor of the Ising blocking measure (when $J(i)\equiv 1$) is given by,
    $$Z_{\beta,q,c}=e^{-\beta}Z(q^2,q^{-2c},e^{-2\beta})=e^{-\beta}\prod\limits_{i=1}^\infty(1+(e^{-2\beta}-1)q^{2i})(1+q^{2(i-c)})(1+q^{2(i-1+c)}).$$   
\end{proof}
\par Now we will discuss the combinatorial interpretation of the identity in Theorem \ref{thm: J is 1 identity}. First let us rearrange to give,
\begin{equation}\label{eq: J=1 iden rearranged}
    \sum\limits_{m\in\mathbb{Z}}Q^{\frac{m(m+1)}{2}}z^m\prod\limits_{i=1}^\infty\frac{1+(y-1)Q^i}{1-Q^i}=Z(Q,z,y).
\end{equation}
We will first consider the left side of this identity. If we look at the constant term in $z$, that is
\begin{equation}
\prod\limits_{i=1}^\infty\frac{1+(y-1)Q^i}{1-Q^i},
\end{equation}
we can view this product as $\sum\limits_{n\geq 0,k\geq1}a_{n,k}Q^ny^k$ where $a_{n,k}$ counts the number of integer partitions of $n$ with $k$ different sizes of part. That is for any integer partition of length $l>0$, $(\lambda_1,...,\lambda_l)$, with $\lambda_i\geq \lambda_{i+1}$, the number of different sizes of part is $k=1+\sum\limits_{i=1}^{l-1}\1_{\{\lambda_i>\lambda_{i+1}\}}$. This is easy to see by going back to our calculation in Remark \ref{rem: pi},
\begin{equation}
\prod\limits_{i=1}^\infty\left\{\sum\limits_{z=0}^\infty y^{\1_{\{z>0\}}}(Q^i)^z\right\}=\prod\limits_{i=1}^\infty \frac{1+(y-1)Q^i}{1-Q^i}.
\end{equation}
\par\noindent For example, consider integer partitions of $n=7$, we know that there are 15 in total, and we see that $a_{7,1}=2$, $a_{7,2}=11$ and $a_{7,3}=2$, indeed:
\begin{align*}
    \text{The partitions with 1 size of part are: } &\Big\{(7), (1,1,1,1,1,1,1)\Big\}.\\
    \text{The partitions with 2 sizes of part are: } &\Big\{(6,1), (5,2), (5,1,1), (4,3), (4,1,1,1), (3,3,1), (3,2,2),\\
    &\hspace{7mm} (3,1,1,1,1), (2,2,2,1), (2,2,1,1,1),(2,1,1,1,1,1)\Big\}.\\
    \text{The partitions with 3 sizes of part are: } &\Big\{(4,2,1), (3,2,1,1)\Big\}.\\
\end{align*}
\par To understand the combinatorial meaning of the other terms, we look at Frobenius partitions and a Wright bijection argument \cite{wright}. First we will recall what Frobenius partitions are (for more details see \cite{Andrews_GFP}). Any integer partition can be represented by a two-rowed array known as a Frobenius partition. This is found by considering the Young diagram of the partition (for example see Figure \ref{fig: fp example}): 
\begin{itemize}
    \item Eliminate the diagonal (say it is of length $s>0$).
    \item The first row of the Frobenius partition is then $(a_1,...,a_s)$ where $a_i$ is equal to the number of boxes to the right of the $i^\text{th}$ diagonal box. 
    \item The second row of the Frobenius partition is then $(b_1,...,b_s)$ where $b_i$ is equal to the number of boxes below the $i^\text{th}$ diagonal box. 
\end{itemize}
From this construction we see that the two rows of the Frobenius partition are strictly decreasing. From the Frobenius partition we can write the number we are partitioning as $n=s+\sum\limits_{i=1}^s(a_i+b_i)$. Using this construction above we have that integer partitions and all 2, equal length, rowed arrays with strictly decreasing rows (Frobenius Partitions) are in 1:1 correspondence. We will denote the set of Frobenius partitions of $n$ by $\text{FP}_0(n)$.
\begin{figure}[H]
    \centering
\begin{tikzpicture}[scale=0.6]
\filldraw[red](0,0)rectangle(0.5,0.5);
\draw[black](0.5,0)--(4,0);
\draw[black](0.5,0.5)--(4,0.5);
\draw[black](4,0)--(4,0.5);
\draw[black](1,0)--(1,0.5);
\draw[black](1.5,0)--(1.5,0.5);
\draw[black](2,0)--(2,0.5);
\draw[black](2.5,0)--(2.5,0.5);
\draw[black](3,0)--(3,0.5);
\draw[black](3.5,0)--(3.5,0.5);

\draw[black](0,-0.5)rectangle(0.5,0);
\filldraw[red](0.5,-0.5)rectangle(1,0);
\draw[black](1,-0.5)--(3.5,-0.5);
\draw[black](3.5,-0.5)--(3.5,0);
\draw[black](1.5,-0.5)--(1.5,0);
\draw[black](2,-0.5)--(2,0);
\draw[black](2.5,-0.5)--(2.5,0);
\draw[black](3,-0.5)--(3,0);

\draw[black](0,-1)rectangle(0.5,-0.5);
\draw[black](0.5,-1)rectangle(1,-0.5);
\filldraw[red](1,-1)rectangle(1.5,-0.5);
\draw[black](1.5,-1)--(3.5,-1);
\draw[black](3.5,-1)--(3.5,-0.5);
\draw[black](1.5,-0.5)--(1.5,-1);
\draw[black](2,-0.5)--(2,-1);
\draw[black](2.5,-0.5)--(2.5,-1);
\draw[black](3,-0.5)--(3,-1);

\draw[black](0,-1)rectangle(0.5,-1.5);
\draw[black](0.5,-1)rectangle(1,-1.5);
\draw[black](1,-1)rectangle(1.5,-1.5);

\draw[black](0,-2)rectangle(0.5,-1.5);

\draw[black](0,-2)rectangle(0.5,-2.5);

\node (a) at (2.5,-4) [label=\small{$\begin{pmatrix}
7&5&4\\
5&2&1
\end{pmatrix}$}]{};
\end{tikzpicture}
    \caption{The integer partition $(8,7,7,3,1,1)$ and its Frobenius partition.}
    \label{fig: fp example}
\end{figure}
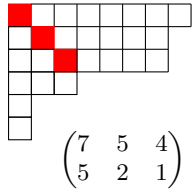
\par As we have seen, in the identity of Theorem \ref{thm: J is 1 identity}, the power of $y$ counts the number of different sizes of part for integer partitions. We will now see what the equivalent quantity is for Frobenius partitions. Consider some integer partition $(\lambda_1,...,\lambda_l)$ and suppose that the equivalent Frobenius partition is of length $s\leq l$, this means that $\lambda_s\geq s$ and $\lambda_{s+1}\leq s$ (if it exists).
\begin{itemize}
    \item Above the diagonal ($i<s$), if $\lambda_i>\lambda_{i+1}$ this corresponds to $a_i>a_{i+1}+1$.
    \item Below the diagonal, a change in the size of part corresponds to $b_i+i-s>b_{i+1}+i+1-s$ for some $i<s$ or equivalently, $b_i>b_{i+1}+1$.
\end{itemize}
We must also consider what happens at the diagonal. We note that if both $a_s\neq 0$ and $b_s\neq 0$ the conditions above under count the number of distinct sizes by 1 (see Figure \ref{fig_FP_distinct_sizes} for examples). Thus, for a given integer partition and its equivalent Frobenius partition we have that,    
\begin{equation}\label{eq: y power FP no offset} 
    1+\sum\limits_{i=1}^{l-1}\1_{\{\lambda_i>\lambda_{i+1}\}}=1+\sum\limits_{i=1}^{s-1}\left(\1_{\{a_i>a_{i+1}+1\}}+\1_{\{b_i>b_{i+1}+1\}}\right)+\1_{\{a_s\neq0\}}\1_{\{b_s\neq 0\}}.
    \end{equation}
We will denote the set of Frobenius partitions of $n$ with $\eqref{eq: y power FP no offset}=k$ by $\text{FP}_{0,k}(n)$. 

\newpage 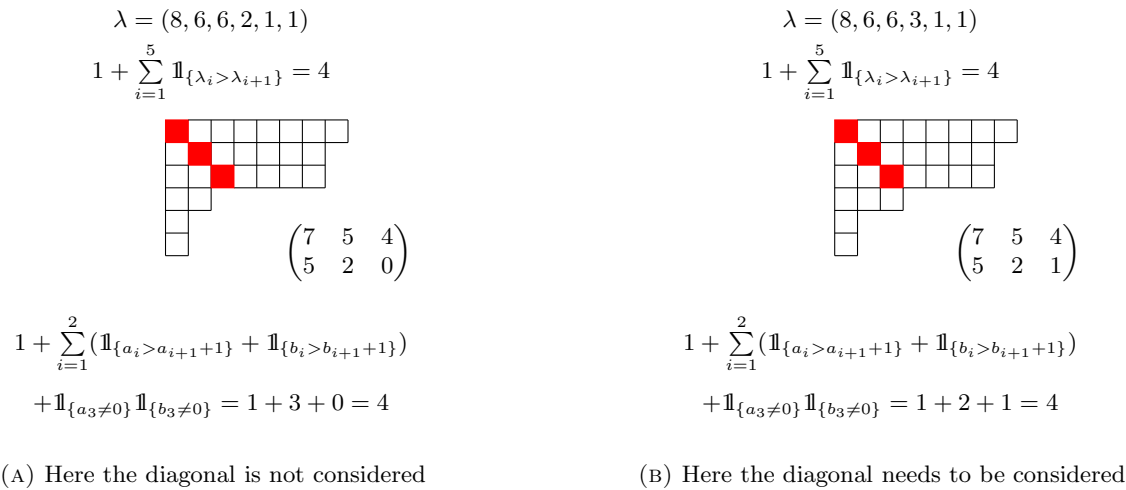
\begin{figure}[H]
    \centering
\begin{subfigure}[b]{0.45\textwidth}
    \centering
\begin{tikzpicture}[scale=0.6]
\filldraw [red] (0,3) rectangle (0.5,2.5);
\draw[black](0.5,3) -- (4,3);
\draw[black](0.5,2.5) -- (4,2.5);
\draw[black] (1,3)--(1,1);
\draw[black] (1.5,3)--(1.5,2);
\draw[black] (2,3)--(2,1.5);
\draw[black] (2.5,3)--(2.5,1.5);
\draw[black] (3,3)--(3,1.5);
\draw[black] (3.5,3)--(3.5,1.5);
\draw[black] (4,3)--(4,2.5);
\filldraw [red] (0.5,2.5) rectangle (1,2);
\draw[black] (0,2.5)--(0,0);
\draw[black] (0,2) --(0.5,2);
\draw[black] (1.5,2) --(3.5,2);
\draw[black] (0,1.5) --(1,1.5);
\draw[black] (1.5,1.5) --(3.5,1.5);
\draw[black] (0.5,2)--(0.5,0);
\filldraw [red] (1,2) rectangle (1.5,1.5);
\draw[black] (0,1)--(1,1);
\draw[black](0,0.5)--(0.5,0.5);
\draw[black](0,0)--(0.5,0);

\node (a) at (4,-1)[label=\small{$\begin{pmatrix}
7&5&4\\
5&2&0
\end{pmatrix}$}]{};  
\node(a) at (1,4.5) [label=\small{$\lambda=(8,6,6,2,1,1)$}]{}; 
\node(a) at (1,3) [label=\small{$1+\sum\limits_{i=1}^5\1_{\{\lambda_i>\lambda_{i+1}\}}=4$}]{}; 
\node(a) at (1,-3) [label=\small{$1+\sum\limits_{i=1}^2(\1_{\{a_i>a_{i+1}+1\}}+\1_{\{b_i>b_{i+1}+1\}})$}]{};
\node(a) at (1,-4) [label=\small{$+\1_{\{a_3\neq 0\}}\1_{\{b_3\neq 0\}}=1+3+0=4$}]{};
\end{tikzpicture}
\caption{Here the diagonal is not considered}
\end{subfigure}
\hfill
\begin{subfigure}[b]{0.45\textwidth}
\centering
\begin{tikzpicture}[scale=0.6]
\filldraw [red] (0,3) rectangle (0.5,2.5);
\draw[black](0.5,3) -- (4,3);
\draw[black](0.5,2.5) -- (4,2.5);
\draw[black] (1,3)--(1,1);
\draw[black] (1.5,3)--(1.5,2);
\draw[black] (2,3)--(2,1.5);
\draw[black] (2.5,3)--(2.5,1.5);
\draw[black] (3,3)--(3,1.5);
\draw[black] (3.5,3)--(3.5,1.5);
\draw[black] (4,3)--(4,2.5);
\filldraw [red] (0.5,2.5) rectangle (1,2);
\draw[black] (0,2.5)--(0,0);
\draw[black] (0,2) --(0.5,2);
\draw[black] (1.5,2) --(3.5,2);
\draw[black] (0,1.5) --(1,1.5);
\draw[black] (1.5,1.5) --(3.5,1.5);
\draw[black] (0.5,2)--(0.5,0);
\filldraw [red] (1,2) rectangle (1.5,1.5);
\draw[black](1.5,1.5)--(1.5,1);
\draw[black] (0,1)--(1.5,1);
\draw[black](0,0.5)--(0.5,0.5);
\draw[black](0,0)--(0.5,0);

\node (a) at (4,-1)[label=\small{$\begin{pmatrix}
7&5&4\\
5&2&1
\end{pmatrix}$}]{};  
\node(a) at (1,4.5) [label=\small{$\lambda=(8,6,6,3,1,1)$}]{}; 
\node(a) at (1,3) [label=\small{$1+\sum\limits_{i=1}^5\1_{\{\lambda_i>\lambda_{i+1}\}}=4$}]{}; 
\node(a) at (1,-3) [label=\small{$1+\sum\limits_{i=1}^2(\1_{\{a_i>a_{i+1}+1\}}+\1_{\{b_i>b_{i+1}+1\}})$}]{};
\node(a) at (1,-4) [label=\small{$+\1_{\{a_3\neq 0\}}\1_{\{b_3\neq 0\}}=1+2+1=4$}]{};
\end{tikzpicture}
\caption{Here the diagonal needs to be considered}
\end{subfigure}
\caption{Examples of the number of distinct sizes in certain integer partitions and the equivalent value for the corresponding Frobenius partitions.}
\label{fig_FP_distinct_sizes}
\end{figure} 

\par So far we have discussed Frobenius partitions where the rows are of equal length, we say these are Frobenius partitions with offset 0. We can consider Frobenius partitions with rows of differing lengths. For example, for some $s_1$ and $s_2$, 
$$\left(\begin{array}{cccc}a_1 & a_2 & ... & a_{s_1}\\ b_1 & b_2 & ... & b_{s_2}\end{array}\right)$$
such that $a_i>a_{i+1}\geq 0$ and $b_j>b_{j+1}\geq 0$, is a Frobenius partition with offset $m=s_1-s_2$. As in \cite{MDJ} we will denote the (first) $|s_1-s_2|$ ``empty" entries in the shorter row by dashes. We will denote the set of Frobenius partitions $n$ which have offset $m$ by $\text{FP}_m(n)$.
\par Frobenius partitions of some offset $m$ can be constructed from ordinary Frobenius partitions (offset 0); in particular for each $m$ there is a bijection, often attributed to 
 Wright, $\phi_m:\text{FP}_{0}(n) \rightarrow \text{FP}_{m}\left(n+\frac{m(m+1)}{2}\right)$. Given an element of $\text{FP}_{0}(n)$ we have a corresponding ordinary partition of $n$. Adjoin a right angled triangle of size $\frac{|m|(|m|+1)}{2}$ to either the left or top edge of its Young Diagram, depending on whether $m\geq 0$ or $m<0$ respectively (see Figure \ref{fig: wright biject example}). Now use the new leading diagonal implied by the triangle to read off an element of $\text{FP}_{m}\left(n+\frac{m(m+1)}{2}\right)$ by letting $s_1$ be the length of the diagonal if $m\geq 0$ ($s_2$ if $m<0$), the $a_i$ are the sizes of rows to the right of the diagonal and the $b_i$ are the sizes of columns under the diagonal (the $|m|$ empty rows/columns coming from the triangle supply the required $|m|$ ``empty" entries in the corresponding Frobenius partition of offset $m$).
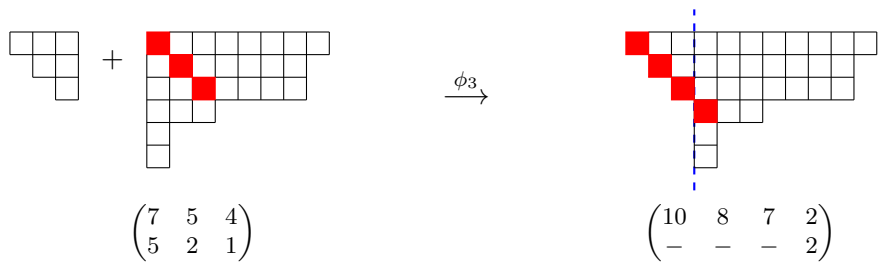
\begin{figure}[H]
    \centering
\begin{subfigure}[b]{\textwidth}
    \centering
\begin{tikzpicture}[scale=0.6]
\draw[black](-3,3) -- (-1.5,3);
\draw[black] (-3,3) -- (-3,2.5);
\draw[black] (-2.5,3) -- (-2.5,2);
\draw[black] (-2,3) -- (-2,1.5);
\draw[black] (-1.5,3) -- (-1.5,1.5);
\draw[black](-3,2.5) -- (-1.5,2.5);
\draw[black](-2.5,2) -- (-1.5,2);
\draw[black](-2,1.5) -- (-1.5,1.5);
\node (a) at (-0.75,1.75)[label=\large{+}]{};
\filldraw [red] (0,3) rectangle (0.5,2.5);
\draw[black](0.5,3) -- (4,3);
\draw[black](0.5,2.5) -- (4,2.5);
\draw[black] (1,3)--(1,1);
\draw[black] (1.5,3)--(1.5,2);
\draw[black] (2,3)--(2,1.5);
\draw[black] (2.5,3)--(2.5,1.5);
\draw[black] (3,3)--(3,1.5);
\draw[black] (3.5,3)--(3.5,1.5);
\draw[black] (4,3)--(4,2.5);
\filldraw [red] (0.5,2.5) rectangle (1,2);
\draw[black] (0,2.5)--(0,0);
\draw[black] (0,2) --(0.5,2);
\draw[black] (1.5,2) --(3.5,2);
\draw[black] (0,1.5) --(1,1.5);
\draw[black] (1.5,1.5) --(3.5,1.5);
\draw[black] (0.5,2)--(0.5,0);
\filldraw [red] (1,2) rectangle (1.5,1.5);
\draw[black](1.5,1.5)--(1.5,1);
\draw[black] (0,1)--(1.5,1);
\draw[black](0,0.5)--(0.5,0.5);
\draw[black](0,0)--(0.5,0);

\node (a) at (1,-2.5)[label=\small{$\begin{pmatrix}
7&5&4\\
5&2&1
\end{pmatrix}$}]{};
\node (a) at (7,1)[label=\large{$\overset{\phi_{3}}{\longrightarrow}$}]{};
\draw [dashed, thick, blue] (12,-0.5) -- (12,3.5);
\filldraw [red] (10.5,3) rectangle (11,2.5);
\draw[black](11,3)--(16,3);
\draw[black](11,2.5)--(16,2.5);
\filldraw [red] (11,2.5) rectangle (11.5,2);
\draw[black](11.5,2)--(15.5,2);
\draw[black](12,1.5)--(15.5,1.5);
\draw[black](12.5,1)--(13.5,1);
\draw[black](12,0.5)--(12.5,0.5);
\draw[black](12,0)--(12.5,0);
\draw[black](11.5,3)--(11.5,2.5);
\draw[black](12,3)--(12,0);
\draw[black](12.5,3)--(12.5,0);
\filldraw [red] (11.5,2) rectangle (12,1.5);
\filldraw [red] (12,1.5) rectangle (12.5,1);
\draw[black](13,3)--(13,1);
\draw[black](13.5,3)--(13.5,1);
\draw[black](14,3)--(14,1.5);
\draw[black](14.5,3)--(14.5,1.5);
\draw[black](15,3)--(15,1.5);
\draw[black](15.5,3)--(15.5,1.5);
\draw[black](16,3)--(16,2.5);

\node (a) at (13,-2.5)[label=\small{$\begin{pmatrix}
10&8&7&2\\
-&-&-&2
\end{pmatrix}$}]{};
\end{tikzpicture}
    \caption{An element of $\text{FP}_{0,4}(27)$ and its image in $\text{FP}_{3,4}(33)$ }
\end{subfigure}
 \caption{Examples of the Wright bijections $\phi_3$ and $\phi_{-3}$ (adapted from \cite{MDJ} Figure 5).}
    \label{fig: wright biject example}
\end{figure}

\newpage\begin{figure}\ContinuedFloat
\begin{subfigure}[b]{\textwidth}
    \centering
    \vspace{10mm}
    \begin{tikzpicture}[scale=0.6]

\draw[black](0,6) --(0.5,6); 
\draw[black](0,5.5)--(1,5.5);
\draw[black](0,5)--(1.5,5);
\draw[black](0,4.5) --(1.5,4.5);
\draw[black] (0,6)--(0,4.5);
\draw[black] (0.5,6)--(0.5,4.5);
\draw[black] (1,5.5)--(1,4.5);
\draw[black] (1.5,5)--(1.5,4.5);
\node (a) at (0.75,3)[label=\large{+}]{};
\filldraw [red] (0,3) rectangle (0.5,2.5);
\draw[black](0.5,3) -- (4,3);
\draw[black](0.5,2.5) -- (4,2.5);
\draw[black] (1,3)--(1,1);
\draw[black] (1.5,3)--(1.5,2);
\draw[black] (2,3)--(2,1.5);
\draw[black] (2.5,3)--(2.5,1.5);
\draw[black] (3,3)--(3,1.5);
\draw[black] (3.5,3)--(3.5,1.5);
\draw[black] (4,3)--(4,2.5);
\filldraw [red] (0.5,2.5) rectangle (1,2);
\draw[black] (0,2.5)--(0,0);
\draw[black] (0,2) --(0.5,2);
\draw[black] (1.5,2) --(3.5,2);
\draw[black] (0,1.5) --(1,1.5);
\draw[black] (1.5,1.5) --(3.5,1.5);
\draw[black] (0.5,2)--(0.5,0);
\filldraw [red] (1,2) rectangle (1.5,1.5);
\draw[black](1.5,1.5)--(1.5,1);
\draw[black] (0,1)--(1.5,1);
\draw[black](0,0.5)--(0.5,0.5);
\draw[black](0,0)--(0.5,0);
\node (a) at (1,-2.5)[label=\small{$\begin{pmatrix}
7&5&4\\
5&2&1
\end{pmatrix}$}]{};
\node (a) at (7,1)[label=\large{$\overset{\phi_{-3}}{\longrightarrow}$}]{};
\draw [dashed, thick, blue] (9.5,3) -- (14.5,3);
\draw[red](10,4.5) rectangle (10.5,4);
\draw[red](10.5,4)rectangle(11,3.5);
\draw[black](10,4) rectangle (10.5,3.5);
\draw[red](11,3.5) rectangle (11.5,3);
\draw[black](10,3.5)rectangle (10.5,3);
\draw[black](10.5,3.5) rectangle (11,3);
\draw[black] (10,3)--(14,3);
\draw[black] (10,2.5)--(14,2.5);
\draw[black] (10,2)--(13.5,2);
\draw[black] (10,1.5)--(13.5,1.5);
\draw[black] (10,1)--(11.5,1);
\draw[black] (10,0.5)--(10.5,0.5);
\draw[black] (10,0)--(10.5,0);
\draw[black] (10,3) --(10,0);
\draw[black] (10.5,3) --(10.5,0);
\draw[black] (11,3) --(11,1);
\draw[black] (11.5,3) --(11.5,1);
\draw[black] (12,3) --(12,1.5);
\draw[black] (12.5,3) --(12.5,1.5);
\draw[black] (13,3) --(13,1.5);
\draw[black] (13.5,3) --(13.5,1.5);
\draw[black] (14,3) --(14,2.5);
\filldraw[red](11.5,3)rectangle(12,2.5);
\filldraw[red](12,2.5)rectangle(12.5,2);
\filldraw[red](12.5,2)rectangle(13,1.5);
\node (a) at (12,-2.5)[label=\small{$\begin{pmatrix}
-&-&-&4&2&1\\
8&5&4&2&1&0
\end{pmatrix}$}]{};
\end{tikzpicture}
    \caption{An element of $\text{FP}_{0,4}(27)$ and its image in $\text{FP}_{-3,4}(30)$ }
\end{subfigure}
    \caption{Examples of the Wright bijections $\phi_3$ and $\phi_{-3}$ (adapted from \cite{MDJ} Figure 5).}
\end{figure}
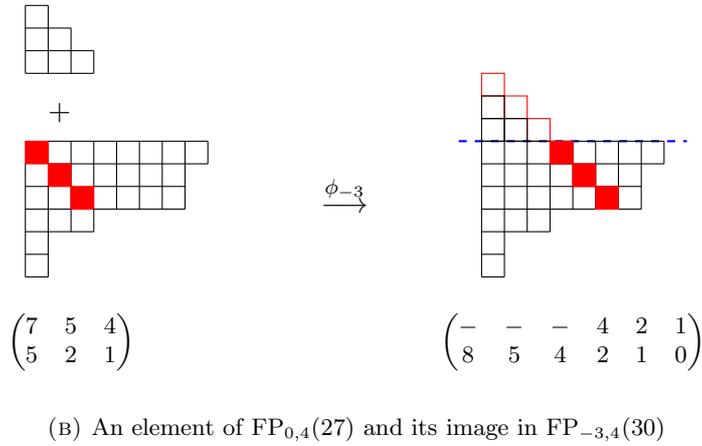
\par These Wright bijections explain the other terms in, 
\begin{equation}
\sum\limits_{m\in\mathbb{Z}}Q^{\frac{m(m+1)}{2}}z^m\prod\limits_{i=1}^\infty\frac{1+(y-1)Q^i}{1-Q^i}.
\end{equation}
We can see that in the sum above, $y$ is only in the product we have already explained. This means that the Wright bijection should not change the value of \eqref{eq: y power FP no offset} that we start with. Thus, for any Frobenius partition (no matter the offset) the quantity that the power of $y$ counts is,
\begin{equation}\label{eq: y FP}
1+\sum\limits_{i=1}^{s_1-1}\1_{\{a_i>a_{i+1}\}}+\sum\limits_{i=1}^{s_2-1}\1_{\{b_i>b_i+1\}}+\1_{\{a_{s_1}\neq 0\}}\1_{\{b_{s_2}\neq 0\}}. 
\end{equation}
Note that when $s_1=s_2$ this is the same as \eqref{eq: y power FP no offset}. We denote the set of Frobenius partitions of $n$ with offset $m$ and $\eqref{eq: y FP}=k$ by $\text{FP}_{m,k}(n)$. As desired, for any $m$, the bijection $\phi_m$ preserves the value of $k$, i.e. $\phi_m:\text{FP}_{0,k}(n) \rightarrow \text{FP}_{m,k}\left(n+\frac{m(m+1)}{2}\right)$.
\par Putting all this together we have that, 
$$Z(Q,z,y)=\sum\limits_{m\in\mathbb{Z}}Q^{\frac{m(m+1)}{2}}z^m\prod\limits_{i=1}^\infty\frac{1+(y-1)Q^i}{1-Q^i}=\sum\limits_{\substack{n\geq 0,\\ m\in\mathbb{Z},\\k>0}}a_{n,m,k}Q^nz^my^k $$
is the 3 variable generating function for Frobenius partitions, where $a_{n,m,k}=|\text{FP}_{m,k}(n)|$.
\begin{rem}
We note here that if we take $y=2$ we have that, 
$$\prod\limits_{i=1}^\infty\frac{1+(y-1)Q^i}{1-Q^i}=\prod\limits_{i=1}^\infty \frac{1+Q^i}{1-Q^i},$$
which can be seen as the generating function for overpartitions (for example see works by Corteel, Lovejoy and Yee \cite{Lovejoy_Corteel_OPs}, \cite{Lovejoy_Corteel_Yee_OPs_GFPs}, \cite{Lovejoy_OPs}). An overpartition of an integer $n$ is an integer partition of $n$ where the first instance of each part size can be overlined or not. For example there are 14 overpartitions of 4,
$$\bigg\{(4),(\overline{4}), (3,1), (\overline{3},1), (3,\overline{1}), (\overline{3},\overline{1}), (2,2), (\overline{2},2), (2,1,1), (\overline{2},1,1),(2,\overline{1},1) ,(\overline{2},\overline{1},1), (1,1,1,1), (\overline{1},1,1,1)\bigg\}.$$
Recall that, 
$$\prod\limits_{i=1}^\infty \frac{1+(y-1)Q^i}{1-Q^i}=\sum\limits_{n\geq 0, k\geq 1} a_{n,k}Q^ny^k$$
is the two variable generating function of integer partitions where the power of $y$ counts the number of sizes of parts in a given partition. Thus it is clear that if we take $y=2$ this is the same as saying for each size of part there are two possible ways of writing the parts i.e. the first instance of that part size can be overlined or not. If we let $\overline{p}(n)$ denote the number of overlined partitions of $n$ we see that $\sum\limits_{k=1}^\infty 2^ka_{n,k}=\overline{p}(n)$, e.g.\ $a_{4,1}=3$, $\left\{(4),(2,2),(1,1,1,1)\right\}$, $a_{4,2}=2$, $\{(3,1),(2,1,1)\}$ and $a_{4,k}=0$ for any $k\geq3$ and we have that, $\sum\limits_{k=1}^\infty 2^ka_{4,k}=14=\overline{p}(4)$.
\par By similar reasoning we can see that by taking $y$ to be any positive integer, say $m$, we have that
$$\prod\limits_{i=1}^\infty\frac{1+(m-1)Q^i}{1-Q^i}=\sum\limits_{n\geq 0, k\geq 1}m^ka_{n,k}Q^n=\sum\limits_{n\geq0}\overline{p}^{(m-1)}(n)Q^n$$
where we say that $\overline{p}^{(m-1)}(n)$ is the number of ``$(m-1)$-overpartitions of $n$", i.e.\ the first part of each size can be written in one of $m$ ways, $\{a,\overline{a},\overline{a}^{(2)},...,\overline{a}^{(m-1)}\}$. For example with $m=3$, $\sum\limits_{k=1}^\infty 3^ka_{4,k}=27=\overline{p}^{(2)}(4)$ which we confirm by looking at the $2$-overpartitions of $4$,
\begin{align*}
    \bigg\{(4),&(\overline{4}), (\overline{\overline{4}}),
    (3,1), (\overline{3},1), (\overline{\overline{3}},1),
    (3,\overline{1}), (3,\overline{\overline{1}}),
    (\overline{3},\overline{1}), (\overline{3},\overline{\overline{1}}),
    (\overline{\overline{3}},\overline{1}),(\overline{\overline{3}},\overline{\overline{1}}),(2,2), (\overline{2},2), (\overline{\overline{2}},2),
    (2,1,1),\\
    &(\overline{2},1,1),(\overline{\overline{2}},1,1), (2,\overline{1},1) , (2,\overline{\overline{1}}),
    (\overline{2},\overline{1},1), (\overline{2},\overline{\overline{1}},1),(\overline{\overline{2}},\overline{1},1),(\overline{\overline{2}},\overline{\overline{1}},1),
    (1,1,1,1), (\overline{1},1,1,1), (\overline{\overline{1}},1,1,1)\bigg\}.
\end{align*}
These generating functions are given on OEIS A321884 \cite{OEIS}, in terms of partitions into coloured blocks of equal parts from $m$ colours.
\end{rem}

\par \vspace{3mm} We have seen above that, by using the identity in Theorem \ref{thm: J is 1 identity}, $Z(Q,z,y)$ can be interpreted as the generating function for certain Frobenius partitions. This is not immediately clear from the form, 
\begin{align}
        &Z(Q,z,y) =  1 + \sum_{\substack{L,R\geq0\\ L+R>0}}y^{(L+R)} \sum_{\substack{\ell_1,\dots,\ell_L\geq 1 \\ m_1,\dots,m_{R}\geq 1}}\prod_{j=1}^{L} \frac{Q^{\tfrac12 \ell_j(\ell_j-1)+j\ell_j+\ell_{j-1}(\ell_j+\cdots+\ell_{L})}z^{-l_j}}{1-Q^{\ell_j+\cdots + \ell_{L}}}\notag 
    \\ &\qquad\qquad\qquad\qquad\prod_{j=1}^{R}\frac{Q^{\tfrac12 m_j(m_j-1)+jm_j+m_{j-1}(m_{j}+\cdots+m_{R})}z^{m_j}}{1-Q^{m_{j}+\cdots+m_{R}}}\bigg(\big(1-y^{-1}\big)Q^{\ell_1+\cdots + \ell_L}Q^{m_{1}+\cdots+m_{R}} + y^{-1}\bigg). \label{eq: Z(q,z,y)}
\end{align}
However we can see that this sum does look rather combinatorial. It will be useful to pull the constant $1$ into the sum, we see that it is simply the value of the summand when $L=R=0$.
\par Let us start by considering, 
$$y^L\sum\limits_{\ell_1,\cdots,\ell_L\geq 1}\prod\limits_{j=1}^L\frac{Q^{\frac{\ell_j(\ell_j-1)}{2}+j\ell_j+\ell_{j-1}(\ell_j+\cdots+\ell_L)}z^{-\ell_j}}{1-Q^{\ell_j+\cdots+\ell_L}}$$
for some fixed $L>0$. By writing $\frac{1}{1-Q^{\ell_j+\cdots+\ell_L}}$ as a geometric series we have, 
\begin{equation}
    y^L\sum\limits_{\ell_1,\cdots,\ell_L\geq 1}\prod\limits_{j=1}^L\frac{Q^{\frac{\ell_j(\ell_j-1)}{2}+j\ell_j+\ell_{j-1}(\ell_j+\cdots+\ell_L)}z^{-\ell_j}}{1-Q^{\ell_j+\cdots+\ell_L}}=y^L\sum\limits_{\substack{\ell_1,\cdots,\ell_L\geq 1\\ \alpha_1,\cdots,\alpha_L\geq0}}\prod\limits_{j=1}^L Q^{\frac{\ell_j(\ell_j-1)}{2}+(\alpha_j+1+\ell_{j-1})(\ell_j+\cdots+\ell_L)}z^{-\ell_j},
\end{equation}
with $l_0=-1$.
If we expand the product and group terms in the correct way, we see that the power of $Q$ forms an integer partition with:
\begin{itemize}
    \item $\alpha_1$ parts of size $\ell_1+\cdots+\ell_L$,
    \item distinct parts of sizes $(\ell_1-1+\ell_2+\cdots+\ell_L)$, $(\ell_1-2+\ell_2+\cdots+\ell_L)$, ..., $(\ell_2+\cdots+\ell_L+1)$, 
    \item $\alpha_2+2$ parts of size $\ell_2+\cdots+\ell_L$,
    \item distinct parts of sizes $(\ell_2-1+\ell_3+\cdots+\ell_L)$, $(\ell_2-2+\ell_3+\cdots+\ell_L)$, ..., $(\ell_3+\cdots+\ell_L+1)$,
    \\ \vdots
    \item $\alpha_L+2$ parts of size $\ell_L$,
    \item distinct parts of sizes $\ell_L-1$, $\ell_L-2$, ..., $1$.    
\end{itemize}
Where $\alpha_1,\cdots,\alpha_L\geq 0$, so there could possibly be no parts of size $\ell_1+\cdots+\ell_L$ etc. See Figure \ref{fig: J is 1 Z partition} for an example when $L=3$. The power of $z^{-1}$ is $\ell_1+\cdots+\ell_L$. If $\alpha_1>0$ this is the size of the largest part in the partition or, equivalently, the number of parts in the conjugate partition. When $\alpha_1=0$ this is one more than the largest part (or number of parts in the conjugate partition). The power of $y$ is $L$, which we can see is the number of runs of parts that differ by exactly 1, in other words the number of ``triangles" we see in the Young diagram (see Figure \ref{fig: J is 1 Z partition} for example). Equivalently, this is the number of sizes of parts that are repeated in the above partition (or one more than that). If $\alpha_1\geq 2$ then there are exactly $L$ sizes for which there are at least 2 parts, these are $\ell_1+\cdots+\ell_L$, $\ell_2+\cdots+\ell_L$, ... and, $\ell_L$. If $\alpha_1<2$ then there are $L-1$ part sizes that are repeated.  
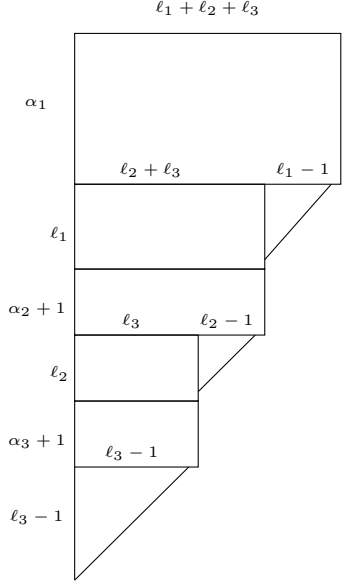
\begin{figure}[H]
    \centering
    \begin{tikzpicture}[scale=0.5]
    \draw (0,0) rectangle (7,-4); 
    \draw (0,-4) rectangle (5,-6.25);
    \draw(6.75,-4) -- (5,-6);
    \draw (0,-6.25) rectangle (5,-8);
    \draw (0,-8) rectangle (3.25, -9.75);
    \draw (4.75,-8) -- (3.25,-9.5);
    \draw (0,-9.75) rectangle (3.25, -11.5);
    \draw (0,-11.5) -- (0,-14.5);
    \draw (0,-14.5) -- (3,-11.5);
    \node (a) at (3.5,0) [label=\tiny{$\ell_1+\ell_2+\ell_3$}]{};
    \node(a) at (-1,-2.5) [label=\tiny{$\alpha_1$}]{};
    \node (a) at (2,-4.3) [label=\tiny{$\ell_2+\ell_3$}]{};
    \node (a) at (6,-4.3) [label=\tiny{$\ell_1-1$}]{};
    \node(a) at (-0.4,-6) [label=\tiny{$\ell_1$}]{};
    \node(a) at (-1,-8) [label=\tiny{$\alpha_2+1$}]{};
    \node (a) at (1.5,-8.3) [label=\tiny{$\ell_3$}]{};
    \node (a) at (4,-8.3) [label=\tiny{$\ell_2-1$}]{};
    \node(a) at (-0.4,-9.7) [label=\tiny{$\ell_2$}]{};
    \node(a) at (-1,-11.5) [label=\tiny{$\alpha_3+1$}]{};
    \node (a) at (1.5,-11.8) [label=\tiny{$\ell_3-1$}]{};
    \node(a) at (-1,-13.5) [label=\tiny{$\ell_3-1$}]{};
    \end{tikzpicture}
    \caption{The partition $\sum\limits_{j=1}^L\left\{\frac{\ell_j(\ell_j-1)}{2}+(\alpha_j+1+\ell_{j-1})(\ell_j+\cdots+\ell_L)\right\}$, for $L=3$.}
    \label{fig: J is 1 Z partition}
\end{figure}
We also have the term, 
$$y^L\sum\limits_{\ell_1,\cdots,\ell_L\geq 1}\prod\limits_{j=1}^L\frac{Q^{\frac{\ell_j(\ell_j-1)}{2}+j\ell_j+\ell_{j-1}(\ell_j+\cdots+\ell_L)}z^{-\ell_j}}{1-Q^{\ell_j+\cdots+\ell_L}}\cdot Q^{\ell_1+\cdots+\ell_L}.$$
By multiplying through by $Q^{\ell_1+\cdots+\ell_L}$, we see that the partitions described above are now guaranteed to have at least one part of size $\ell_1+\cdots+\ell_L$ and so the power of $z^{-1}$ is exactly the number of parts in the conjugate partition. The power of $y$ still gives the number of runs of parts that differ by 1 in size (i.e.\ the number of triangles like in Figure \ref{fig: J is 1 Z partition}).
\par By similar reasoning we have, 

\small $$  y^R\sum\limits_{m_1,\cdots,m_R\geq 1}\prod\limits_{j=1}^R\frac{Q^{\frac{m_j(m_j-1)}{2}+jm_j+m_{j-1}(m_j+\cdots+m_R)}z^{m_j}}{1-Q^{m_j+\cdots+m_R}}=y^R\sum\limits_{\substack{m_1,\cdots,m_R\geq 1\\ \beta_1,\cdots,\beta_R\geq 0}}\prod\limits_{j=1}^RQ^{\frac{m_j(m_j-1)}{2}+(\beta_j+1+m_{j-1})(m_j+\cdots+m_R)}z^{m_j},$$
\normalsize with $m_0=0$. Again we see that the power of $Q$ forms an integer partition with:
\begin{itemize}
    \item $\beta_1+1$ parts of size $m_1+\cdots+m_R$,
    \item distinct parts of sizes $(m_1-1+m_2+\cdots+m_R)$, $(m_1-2+m_2+\cdots+m_R)$, ..., $(m_2+\cdots+m_R+1)$, 
    \item $\beta_2+2$ parts of size $m_2+\cdots+m_R$,
    \item distinct parts of sizes $(m_2-1+m_3+\cdots+m_R)$, $(m_2-2+m_3+\cdots+m_R)$, ..., $(m_3+\cdots+m_R+1)$,
    \\ \vdots
    \item $\beta_R+2$ parts of size $m_R$,
    \item distinct parts of sizes $m_R-1$, $m_R-2$, ..., $1$.    
\end{itemize}
The power of $z$ is $m_1+\cdots+m_R$, which is the size of the largest part or equivalently the number of parts in the conjugate partition. The power of $y$ is R which is the number of runs of parts that differ by 1 in size. Or equivalently we can see this as the number of sizes that are used more than once in the partition if $\beta_1>0$, and one more than this if $\beta_1=0$.
\par This time when we multiply through by $Q^{m_1+\cdots+m_R}$ we have that the partition must have at least $2$ parts of size $m_1+\cdots+m_R$. Thus, the power of $y$ is now exactly the number of sizes that are used more than once.
\par We now use a ``General Principle" due to Andrews \cite{Andrews_GFP}, to see that the products 
\begin{equation}\label{eq: J=1 Z part 1}
\sum_{L,R\geq0}y^{(L+R-1)}\sum_{\substack{\ell_1,\dots,\ell_L\geq 1 \\ m_1,\dots,m_{R}\geq 1}}\prod_{j=1}^{L} \frac{Q^{\tfrac12 \ell_j(\ell_j-1)+j\ell_j+\ell_{j-1}(\ell_j+\cdots+\ell_{L})}z^{-l_j}}{1-Q^{\ell_j+\cdots + \ell_{L}}}
\prod_{j=1}^{R}\frac{Q^{\tfrac12 m_j(m_j-1)+jm_j+m_{j-1}(m_{j}+\cdots+m_{R})}z^{m_j}}{1-Q^{m_{j}+\cdots+m_{R}}}
\end{equation}
and, 
\begin{align}
\sum_{L,R\geq0}y^{(L+R)}\sum_{\substack{\ell_1,\dots,\ell_L\geq 1 \\ m_1,\dots,m_{R}\geq 1}}&\prod_{j=1}^{L} \frac{Q^{\tfrac12 \ell_j(\ell_j-1)+j\ell_j+\ell_{j-1}(\ell_j+\cdots+\ell_{L})}z^{-l_j}}{1-Q^{\ell_j+\cdots + \ell_{L}}}\notag
\\&\cdot\prod_{j=1}^{R}\frac{Q^{\tfrac12 m_j(m_j-1)+jm_j+m_{j-1}(m_{j}+\cdots+m_{R})}z^{m_j}}{1-Q^{m_{j}+\cdots+m_{R}}} \cdot (1-y^{-1})Q^{\ell_1+\cdots+\ell_L}Q^{m_1+\cdots+m_R} \label{eq: J=1 Z part 2}\end{align}
($\ell_0=-1$ and $m_0=0$) are generating functions for certain (non-trivial) generalised Frobenius partitions. As we have seen, the rows of a Frobenius partition are strictly decreasing. By relaxing this condition, or considering other conditions, we get generalised Frobenius partitions (GFPs). The ``General Principle" gives a way of writing the generating function for generalised Frobenius partitions with certain conditions. Suppose that $f_{A}(Q,z) = \sum_{n,k}a_{n,k}Q^nz^k$ and $f_{B}(Q,z) = \sum_{n,k}b_{n,k}Q^nz^k$ are generating functions for ordinary partitions of $n$ with $k$ parts and satisfying some conditions $A$ and $B$, respectively. Then the ``General Principle" states that the formal series $f_{A}(Q,Qz)f_{B}(Q,z^{-1}) = \sum_{k} f_{A,B,k}(Q)z^k$ is the generating function for generalised Frobenius partitions with first row satisfying condition $A$ and second row satisfying $B$, with the power of $z$ giving the offset.
\par We have a third variable $y$ which is not considered in the ``General Principle", but we see that in both of the above the $y$ factor can be pulled out. Let us first consider \eqref{eq: J=1 Z part 1}, in particular for given $L$ and $R$,
\small \begin{align*}
&\sum_{\substack{\ell_1,\dots,\ell_L\geq 1 \\ m_1,\dots,m_{R}\geq 1}}\prod_{j=1}^{L} \frac{Q^{\tfrac12 \ell_j(\ell_j-1)+j\ell_j+\ell_{j-1}(\ell_j+\cdots+\ell_{L})}z^{-l_j}}{1-Q^{\ell_j+\cdots + \ell_{L}}}
\prod_{j=1}^{R}\frac{Q^{\tfrac12 m_j(m_j-1)+jm_j+m_{j-1}(m_{j}+\cdots+m_{R})}z^{m_j}}{1-Q^{m_{j}+\cdots+m_{R}}}\\
&= \left\{\sum\limits_{m_1,\cdots,m_R\geq 1}\prod\limits_{j=1}^R\frac{Q^{\tfrac12 m_j(m_j-1)+jm_j+\hat{m}_{j-1}(m_{j}+\cdots+m_{R})}(Qz)^{m_j}}{1-Q^{m_{j}+\cdots+m_{R}}}\right\}\cdot\left\{\sum\limits_{\ell_1,\cdots,\ell_L\geq 1}\prod_{j=1}^{L} \frac{Q^{\tfrac12 \ell_j(\ell_j-1)+j\ell_j+\ell_{j-1}(\ell_j+\cdots+\ell_{L})}z^{-l_j}}{1-Q^{\ell_j+\cdots + \ell_{L}}}\right\}
\end{align*}
\normalsize where $\hat{m}_0=-1$ and $\hat{m}_j=m_j$ ($j\in\{1,\cdots,L\}$). For \eqref{eq: J=1 Z part 2}, for given $L$ and $R$ we have, 
\small\begin{align*}
    &\sum_{\substack{\ell_1,\dots,\ell_L\geq 1 \\ m_1,\dots,m_{R}\geq 1}}\prod_{j=1}^{L} \frac{Q^{\tfrac12 \ell_j(\ell_j-1)+j\ell_j+\ell_{j-1}(\ell_j+\cdots+\ell_{L})}z^{-l_j}}{1-Q^{\ell_j+\cdots + \ell_{L}}}
\prod_{j=1}^{R}\frac{Q^{\tfrac12 m_j(m_j-1)+jm_j+m_{j-1}(m_{j}+\cdots+m_{R})}z^{m_j}}{1-Q^{m_{j}+\cdots+m_{R}}}\cdot Q^{\ell_1+\cdots+\ell_L}Q^{m_1+\cdots+m_R}\\
&=\left\{\sum\limits_{m_1,\cdots,m_R\geq 1}\prod\limits_{j=1}^R\frac{Q^{\tfrac12 m_j(m_j-1)+jm_j+m_{j-1}(m_{j}+\cdots+m_{R})}(Qz)^{m_j}}{1-Q^{m_{j}+\cdots+m_{R}}}\right\}\cdot\left\{\sum\limits_{\ell_1,\cdots,\ell_L\geq 1}\prod_{j=1}^{L} \frac{Q^{\tfrac12 \ell_j(\ell_j-1)+j\ell_j+\hat{\ell}_{j-1}(\ell_j+\cdots+\ell_{L})}z^{-l_j}}{1-Q^{\ell_j+\cdots + \ell_{L}}}\right\}
\end{align*}
\normalsize for $\hat{\ell}_0=0$ and $\hat{\ell}_j=\ell_j$ ($j\in\{1,\cdots,R\}$).
\par So $Z(Q,z,y)=\eqref{eq: J=1 Z part 1}+\eqref{eq: J=1 Z part 2}$ is a linear combination of three 3-variable generating functions for certain generalised Frobenius partitions. 
\eqref{eq: J=1 Z part 2} is the difference of two generating functions. Both give generalised Frobenius partitions where the rows are the conjugate partitions of partitions as given in Figure \ref{fig: J is 1 Z partition} but with at least one part of size $\ell_1+\cdots+\ell_L$ (or equivalently $m_1+\cdots+m_R$) rather than potentially none. Again in both generating functions, the power of $z$ gives the offset of these GFPs. In these two generating functions the powers of $y$ are $(L+R)$ and $(L+R-1)$, where $L+R$ is the number of runs of parts that differ by 1 (each row treated separately). \eqref{eq: J=1 Z part 1} is similar to the second generating function in \eqref{eq: J=1 Z part 2} with $y^{L+R-1}$ except the power of $z$ counts the offset or the offset $\pm1$, depending on the values of $\alpha_1$ and $\beta_1$. The four cases are:
\begin{itemize}
    \item $\beta_1>0$ and $\alpha_1>0$ then the power of $z$ is:
    $$\{\text{number of parts in first row}\}-\{\text{number of parts in the second row}\}=\text{offset},$$
    \item $\beta_1>0$ and $\alpha_1=0$ then the power of $z$ is:
    $$\{\text{number of parts in first row}\}-\{\text{number of parts in the second row}+1\}=\text{offset}-1,$$
    \item $\beta_1=0$ and $\alpha_1>0$ then the power of $z$ is:
    $$\{\text{number of parts in first row}+1\}-\{\text{number of parts in the second row}\}=\text{offset}+1, $$
    \item $\beta_1=0$ and $\alpha_1=0$ then the power of $z$ is:
    $$\{\text{number of parts in first row}+1\}-\{\text{number of parts in the second row}+1\}=\text{offset}.$$
\end{itemize}

\par \vspace{3mm} So the identity we have proved in Theorem \ref{thm: J is 1 identity} can be seen as a rather surprising equivalence of generating functions for certain (generalised) Frobenius partitions. One side of the rearranged identity \eqref{eq: J=1 iden rearranged}, 
$$\sum\limits_{m\in\mathbb{Z}}Q^{\frac{m(m+1)}{2}}z^m\prod\limits_{i=1}^\infty\frac{1+(y-1)Q^i}{1-Q^i}=\sum\limits_{\substack{n\geq 0,\\ m\in\mathbb{Z},\\k>0}}a_{n,m,k}Q^nz^my^k $$
is the 3 variable generating function for Frobenius partitions where $a_{n,m,k}=|\text{FP}_{m,k}(n)|$ is the number of offset $m$ Frobenius partitions of $n$ with quantity \eqref{eq: y FP}$=k$. The other side, $Z(Q,z,y)$, is a linear combination of generating functions for generalised Frobenius partitions with rows forming partitions like in Figure \ref{fig: J is 1 Z partition}. With the offset being given by the power of $z$ and the power of $y$ telling us the number runs of parts which differ in size by exactly 1 (each row treated separately).

\subsection{Combinatorial Identity for $J(i)=i$}\label{sec: comb j is i}~
\par Now we consider the identity when $J(i)=i$, \eqref{eq: J=i iden}. Again we let $Q\defeq q^2$, $z\defeq q^{-2c}$ and $y\defeq e^{-2\beta}$.
\iiden*
\begin{rem}
    Again these identities (for each $n$) hold as formal identities. The probabilistic nature of the proof for Theorem \ref{thm: J is i identity} means they hold for $Q\in(0,1)$ and $y\in(0,1]$, but the identities should hold for any $Q,y\in\mathbb{R}$ with $y\neq 0$.
\end{rem}

We will now give combinatorial interpretations to this identity. Let us first consider the RHS, 
$$\sum\limits_{z\in\Omega}\prod\limits_{i=1}^\infty y^{\1_{\{z_{-i}>0\}}(n+i-\sum\limits_{j=i}^\infty z_{-j})}Q^{iz_{-i}}.$$
We can see that the states in $\Omega$ can be though of as integer partitions, this is clear by looking at the power of $Q$ in the above, $\sum\limits_{i=1}^\infty iz_{-i}$. The occupation at site $-i$, $z_{-i}$, gives the number of parts of size $i$ in the partition. It is also clear to see that these states and integer partitions are in 1:1 correspondence.
\par Now we focus on what the power of $y$ counts for a given integer partition (state in $\Omega$), 
$$\sum\limits_{i=1}^\infty \1_{\{z_{-i}>0\}}\bigg((n+i)-\sum\limits_{j=1}^i z_{-j}\bigg)=n\sum\limits_{i=1}^\infty \1_{\{z_{-i}>0\}}+\sum\limits_{i=1}^\infty i\1_{\{z_{-i}>0\}}-\sum\limits_{i=1}^\infty\1_{\{z_{-i}>0\}}\sum\limits_{j=i}^\infty z_{-j}.$$
The three sums each tell us something about the integer partition:
\begin{itemize}
\item $\sum\limits_{i=1}^\infty\1_{\{z_{-i}>0\}}$ is the number of distinct sizes of parts in the partition (as in Theorem \ref{thm: J is 1 identity}).
\item $\sum\limits_{i=1}^\infty i\1_{\{z_{-i}>0\}}$ tells us the ``minimal" integer that is partitioned by knowing the sizes of these parts (i.e.\ the new partition obtained when there is a single part of each size from the original partition).
\item Now we will look at, $\sum\limits_{i=1}^\infty\1_{\{z_{-i}>0\}}\sum\limits_{j=i}^\infty z_{-j}$. For some state $z\in\Omega$ suppose that $\sum\limits_{i=1}^\infty \1_{\{z_{-i}>0\}}=k$. Denote the $k$ occupied sites by $-i_1>\cdots>-i_k$. Then, $\sum\limits_{i=1}^\infty\1_{\{z_{-i}>0\}}\sum\limits_{j=i}^\infty z_{-j}=\sum\limits_{j=1}^k\sum\limits_{\ell=j}^kz_{-i_\ell}$, so that state $z$ has a corresponding integer partition with $k$ distinct parts, $(\sum\limits_{\ell=1}^k z_{-i_\ell},\sum\limits_{\ell=2}^k z_{-i_\ell},\cdots, z_{-i_k})$. Or equivalently, $\sum\limits_{i=1}^\infty\1_{\{z_{-i}>0\}}\sum\limits_{j=i}^\infty z_{-j}=\sum\limits_{j=1}^kjz_{-i_j}$, which describes the conjugate partition with parts of sizes exactly $1$ through $k$ (at least one of each size). We will call this the ``partition of partial sums" associated with the original partition. The conjugate pair of partitions are represented pictorially in Figure \ref{fig: partiton of partial sums} below. 
\end{itemize}

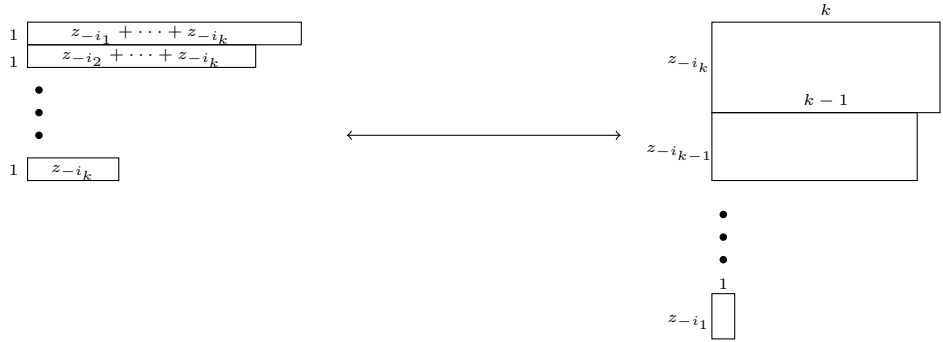
\begin{figure}[H]
    \centering 
    \begin{tikzpicture}[scale=0.6]
    \draw(-10,1.5)rectangle(-4,2);
    \draw(-10,1.5)rectangle(-5,1);
    \filldraw[black] (-9.75,0) circle (2pt);
    \filldraw[black] (-9.75,-0.5) circle (2pt);
    \filldraw[black] (-9.75,0.5) circle (2pt);
    \draw(-10,-1)rectangle(-8,-1.5);
    \node(a) at (-7.3,1.1)[label=\tiny{$z_{-i_1}+\cdots +z_{-i_k}$}]{};
    \node(a) at (-10.3,1.2)[label=\tiny{$1$}]{};
    \node(a) at (-7.5,0.6)[label=\tiny{$z_{-i_2}+\cdots +z_{-i_k}$}]{};
    \node(a) at (-10.3,0.6)[label=\tiny{$1$}]{};
    \node(a) at (-9,-1.9)[label=\tiny{$z_{-i_k}$}]{};
    \node(a) at (-10.3,-1.8)[label=\tiny{$1$}]{};
    \draw[<->] (-3,-0.5) -- (3,-0.5);
    \draw (5,0) rectangle (10,2);
    \draw (5,0) rectangle (9.5,-1.5);
    \draw (5,-4) rectangle (5.5, -5);
    \filldraw[black] (5.25,-2.25) circle (2pt);
    \filldraw[black] (5.25,-2.75) circle (2pt);
    \filldraw[black] (5.25,-3.25) circle (2pt);
    \node(a) at (7.5,1.75)[label=\tiny{$k$}]{};
    \node(a) at (7.5,-0.3)[label=\tiny{$k-1$}]{};
    \node(a) at (5.25,-4.3)[label=\tiny{$1$}]{};
    \node(a) at (4.5,0.5)[label=\tiny{$z_{-  i_k}$}]{};
    \node(a) at (4.3,-1.5)[label=\tiny{$z_{-i_{k-1}}$}]{};
    \node(a) at (4.5,-5.2)[label=\tiny{$z_{-i_1}$}]{};
    \end{tikzpicture}
    \caption{The ``partition of partial sums" and its conjugate associated to the partition $\sum\limits_{j=1}^kj\cdot z_{-i_j}$ where $i_k>\cdots>i_1$.}
    \label{fig: partiton of partial sums}
\end{figure}

So we can see that this is a generating function,
\begin{equation}\label{eq: gen func J is i}   \sum\limits_{z\in\Omega}\prod\limits_{i=1}^\infty y^{\1_{\{z_{-i}>0\}}(n+i-\sum\limits_{j=i}^\infty z_{-j})}Q^{iz_{-i}}=\sum\limits_{m,k,\ell,r\geq 0}a_{m,k,\ell,r}Q^m(y^{n})^ky^\ell(y^{-1})^r,
    \end{equation}
where $a_{m,k,\ell,r}$ is the number of integer partitions of $m$ with $k$ distinct sizes of part, ``minimal" partition totalling $\ell$ and ``partition of partial sums" totalling $r$, as described above.
\begin{rem}
The way we have written the generating function in \eqref{eq: gen func J is i} suggests that it might be natural to consider a four variable generating function here instead. Or possibly a three variable generating function with the power of one variable giving the difference between the ``minimal" partition and ``partition of partial sums", to ensure convergence.  Perhaps there is a generalisation of the Ising model we have considered here whose stationary measure naturally gives the equivalent generating function, with more variables, to \eqref{eq: gen func J is i}.
\end{rem}
\newpage From the above we see that $Z_n(Q,y)$ is a 2 variable generating function for integer partitions with the power of $y$ encoding properties about the partition, namely the number of distinct sizes of part, the sizes of the ``minimal" associated partition and the ``partition of partial sums". Now we look to give a direct combinatorial interpretation to,
\begin{align*}
    &Z_n(Q,y) = 1 + \sum_{L,R>0} y^{nL+(n-1)R}\sum_{\substack{\ell_1,\dots,\ell_L\geq 1 \\ m_1,\dots,m_{R}\geq 1}}\1_{\big\{\sum_{j=1}^R m_j= \sum_{j=1}^L\ell_j\big\}} \prod_{j=1}^{L} \frac{Q^{\tfrac12 \ell_{j}(\ell_{j}-1)}\big(y^{-(L-j+1)}Q^{\ell_{j}+\cdots + \ell_{L}}\big)^{\ell_{j-1}+1}}{1-y^{-(L-j+1)}Q^{\ell_{j}+\cdots + \ell_{L}}} \\
    &\qquad\qquad\qquad\prod_{j=1}^{R}\frac{Q^{\tfrac12 m_{j}(m_{j}-1)}\big(y^{R-j+1}Q^{m_{j}+\cdots+m_{R}}\big)^{m_{j-1}+1}}{1-y^{R-j+1}Q^{m_{j}+\cdots+m_{R}}}\bigg(y^{-n}+\big(1-y^{-n}\big)y^{R-L}Q^{\ell_1+\cdots + \ell_L}Q^{m_{1}+\cdots+m_{R}}\bigg)
\end{align*}   
with $\ell_0=-1$ and $m_0=0$. Just as before, it will be useful to pull the constant 1 into the sum, we see that it is simply the value of the summand when $L=R=0$.
\par We note that this looks very similar to $Z(Q,z,y)$, in particular the powers of $Q$ give the same partitions as we saw in Section \ref{sec: comb j is 1} (see Figure \ref{fig: J is 1 Z partition} for example). Notice that, $Z_n(Q,y)$ is only a 2 variable function; the absence of a $z$ variable and also the indicator that $\sum\limits_{j=1}^Rm_j=\sum\limits_{j=1}^L\ell_j$, tells us that we are now looking at just the constant in $z$ (viewing these as similar to $Z(Q,z,y)$) and thus by the ``General Principle" of Andrews \cite{Andrews_GFP}, we are only considering generalised Frobenius partitions with offset 0. Lastly we note that the power of $y$ here is very different to the power of $y$ in $Z(Q,z,y)$ and so counts something different for these GFPs. 
\par Let us consider the power of $y$ in a general term of, 
\begin{align}
\sum\limits_{L,R\geq0}y^{nL+(n-1)R}\sum_{\substack{\ell_1,\dots,\ell_L\geq 1 \\ m_1,\dots,m_{R}\geq 1}}&\1_{\big\{\sum_{j=1}^R m_j= \sum_{j=1}^L\ell_j\big\}} \prod_{j=1}^{L} \frac{Q^{\tfrac12 \ell_{j}(\ell_{j}-1)}\big(y^{-(L-j+1)}Q^{\ell_{j}+\cdots + \ell_{L}}\big)^{\ell_{j-1}+1}}{1-y^{-(L-j+1)}Q^{\ell_{j}+\cdots + \ell_{L}}} \notag \\
&\cdot\prod_{j=1}^{R}\frac{Q^{\tfrac12 m_{j}(m_{j}-1)}\big(y^{R-j+1}Q^{m_{j}+\cdots+m_{R}}\big)^{m_{j-1}+1}}{1-y^{R-j+1}Q^{m_{j}+\cdots+m_{R}}}. \label{eq: J=i Z part 1}
\end{align}
For some $L,R>0$, $\alpha_1,\cdots,\alpha_L, \beta_1, \cdots, \beta_R\geq0$, and $\ell_1,\cdots,\ell_L,m_1,\cdots, m_R\geq 1$ such that $\sum\limits_{j=1}^Rm_j=\sum\limits_{j=1}^L\ell_j$, also let $\ell_0=-1$ and $m_0=0$, we have 
\begin{align*}
&nL+(n-1)R-\sum\limits_{j=1}^L (\alpha_j+l_{j-1}+1)(L-j+1)+\sum\limits_{k=1}^R(\beta_k+m_{k-1}+1)(R-k+1)\\
&=n(L+R) + \left\{\sum\limits_{k=1}^R(\beta_k+m_{k-1}+1)(R-k+1)-R\right\}-\sum\limits_{j=1}^L (\alpha_j+l_{j-1}+1)(L-j+1).
\end{align*}
We see that this can be split into three parts:
\begin{itemize}
    \item $n(L+R)$, which is $n$ times the number of runs of parts which differ in size by exactly one (each row of the GFP treated separately), as we have seen previously. 
    \item $\sum\limits_{k=1}^R(\beta_k+m_{k-1}+1)(R-k+1)-R$, which can be seen as a ``minimal" partition corresponding to the partition $\sum\limits_{k=1}^R \frac{m_k(m_k-1)}{2}+(\beta_k+1+m_{k-1})(m_k+\cdots+m_R)$. That is, looking at the repeated parts only, i.e.\ $\sum\limits_{k=1}^R (\beta_k+1+m_{k-1})(m_k+\cdots+m_R)$, the minimal number we can partition is by taking $m_k=1$ for all $k\in\{1,\cdots,R\}$ (noting that we also remove one part of size $R$). For example see Figure \ref{fig: y power as partitions J is i} below.
    \item $-\sum\limits_{j=1}^L (\alpha_j+l_{j-1}+1)(L-j+1)$, is the negative of the ``minimal" partition (in the same way as above) corresponding to the partition $\sum\limits_{j=1}^L\frac{\ell_j(\ell_j-1)}{2}+(\alpha_j+1+\ell_{j-1})(\ell_j+\cdots+\ell_L)$.
\end{itemize}

\begin{figure}[H]
    \centering
    \begin{tikzpicture}[scale=0.5]
        \draw (0,0) rectangle (5,-4);
        \draw (0,-4) rectangle (4.5,-7);
        \filldraw[black] (0,-7.5) circle (2pt);
        \filldraw[black] (0,-8) circle (2pt);
        \filldraw[black] (0,-8.5) circle (2pt);
        \draw (0,-9.5) rectangle (0.5,-12);
        \node(a) at (-1,-2.5) [label=\tiny{$\beta_1$}]{};
        \node(a) at (-1.5,-6) [label=\tiny{$\beta_2+m_1+1$}]{};
        \node(a) at (-2,-11.5) [label=\tiny{$\beta_R+m_{R-1}+1$}]{};
        \node (a) at (2.5,-0.2) [label=\tiny{$R$}]{};
        \node (a) at (2.5,-4.2) [label=\tiny{$R-1$}]{};
        \node (a) at (0.25,-9.7) [label=\tiny{$1$}]{};
    \end{tikzpicture}
    \caption{The partition, $\sum\limits_{k=1}^R(\beta_k+m_{k-1}+1)(R-k+1)-R$.}
    \label{fig: y power as partitions J is i}
\end{figure}
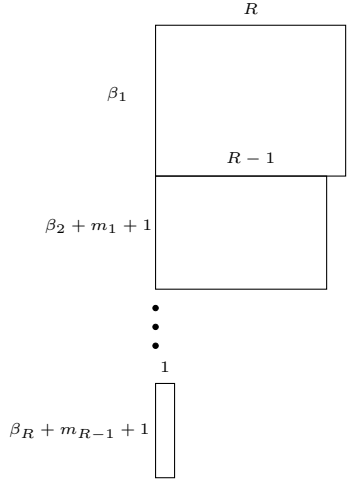
\par Now we consider the other term,
\begin{align}
\sum\limits_{L,R\geq0}y^{nL+(n-1)R}\sum_{\substack{\ell_1,\dots,\ell_L\geq 1 \\ m_1,\dots,m_{R}\geq 1}}&\1_{\big\{\sum_{j=1}^R m_j= \sum_{j=1}^L\ell_j\big\}} \prod_{j=1}^{L} \frac{Q^{\tfrac12 \ell_{j}(\ell_{j}-1)}\big(y^{-(L-j+1)}Q^{\ell_{j}+\cdots + \ell_{L}}\big)^{\ell_{j-1}+1}}{1-y^{-(L-j+1)}Q^{\ell_{j}+\cdots + \ell_{L}}} \notag\\
&\cdot\prod_{j=1}^{R}\frac{Q^{\tfrac12 m_{j}(m_{j}-1)}\big(y^{R-j+1}Q^{m_{j}+\cdots+m_{R}}\big)^{m_{j-1}+1}}{1-y^{R-j+1}Q^{m_{j}+\cdots+m_{R}}}\cdot y^{R-L}Q^{\ell_1+\cdots+\ell_L}Q^{m_1+\cdots+\m_R}. \label{eq: J=i Z part 2}
\end{align}
As in Section \ref{sec: comb j is 1} when we multiply through by $Q^{\ell_1+\cdots+\ell-L}Q^{m_1+\cdots+m_R}$ this means that the conjugate partitions of the rows of the GFP have at least one part of size $m_1+\cdots+m_R$ in the first row and $\ell_1+\cdots+\ell_L$ in second row as opposed to potentially none. We also multiply through by $y^{R-L}$, this accounts for the fact that now the ``minimal" partitions for each row, as described above, have an extra part of size $R$ or $L$.
\par So $Z_n(Q,y)=y^{-n}\times\eqref{eq: J=i Z part 1}+(1-y^{-n})\times\eqref{eq: J=i Z part 2}$ is a linear combination of three 2-variable generating functions. First, $y^{-n}\times\eqref{eq: J=i Z part 1}$ is the generating function for offset 0 generalised Frobenius partitions where the rows are the conjugate partitions of the partitions as given in Figure \ref{fig: J is 1 Z partition}. The power of $y$ can be split into three parts; $n(L+R-1)$, where $L+R$ is the number of runs of parts that differ by 1 (each row treated separately), the total of ``minimal" partition associated to the first row and minus of the total of the ``minimal" partition associated to the second row (as given in Figure \ref{fig: y power as partitions J is i}). Next we have, $(1-y^{-n})\times\eqref{eq: J=i Z part 2}$, which is a sum of two generating functions. Both are again generating functions for offset 0 GFPs with rows the conjugates of partitions as in Figure \ref{fig: J is 1 Z partition} but with at least one part of size $\ell_1+\cdots+\ell_L$ (or equivalently $m_1+\cdots+m_R$) rather than potentially none. Again the power of $y$ in each case can be split into three parts. As before we have the parts that give the totals of ``minimal" partitions associated to each row. The third part is either $n(L+R)$ or $n(L+R-1)$, where again $L+R$ is the number of runs of parts which differ by exactly 1 (each row treated separately).
\par \vspace{3mm} So the identities (one for each $n\in\mathbb{Z}$) we have proved in Theorem \ref{thm: J is i identity} can be seen as a rather surprising equivalence of generating functions for certain integer and offset 0 generalised Frobenius partitions. One side of the identities, 
$$\sum\limits_{z\in\Omega}\prod\limits_{i=1}^\infty y^{\1_{\{z_{-i}>0\}}(n+i-\sum\limits_{j=i}^\infty z_{-j})}Q^{iz_{-i}}=\sum\limits_{m,k,\ell,r\geq 0}a_{m,k,\ell,r}Q^m(y^{n})^ky^\ell(y^{-1})^r$$
is the two variable generating function for integer partitions, with $a_{m,k,\ell,r}$ the number of integer partitions of $m$ with $k$ distinct sizes of parts, ``minimal partition" of size $\ell$ and ``partition of partial sums" of size $r$. The other side, $Z_n(Q,y)$ is a linear combination of generating functions for offset 0 generalised Frobenius partitions with rows forming partitions like in Figure \ref{fig: J is 1 Z partition}. The power of $y$ encodes information about the number of runs of parts differing in size by 1 (each row treated separately) as well as ``minimal" partition sizes for each row of the GFP as discussed above, see Figure \ref{fig: y power as partitions J is i} for example.

\section{The Ising model with long range interactions and Kawasaki Dynamics}

We have seen above that the Ising model with inhomogeneous nearest neighbour coupling constants and inhomogeneous external field can be related to a process of interacting particles which is precisely the zero-range process when $\beta=0$. In this section we consider the Ising model with long-range interactions, find reversible dynamics for the associated Kawasaki dynamics, and using the stand up map $T^n$ relate this model to a particle model with somewhat unusual long-range jumps that inherits reversibility from the Ising model. We then take the opposite approach by defining a desirable dynamics for the particle system and use the stand up map $T^n$ to find a measure on $\Bcal^n$ for which the inherited dynamics are reversible, this can then be transferred back to the particle system. 

To begin, consider the following hamiltonian on $\Omega^{\Is}$ given by
\begin{equation}
H_{J}(\sigma)=\tfrac{1}{2}\sum_{i,j\in\Z}J(i,j)\1_{\{\sigma_i\neq \sigma_j\}},
\end{equation}
where $J:\Z\times\Z\to\R$ with $J(i,j)=J(j,i)\geq 0$ for all $i,j\in \Z$ and $J(i,i)=0$. We could relax the requirements of symmetry and positivity of $J$ and obtain conditions for concentration on $\Bcal$ as in Lemma \ref{lem:longrangeconcentration} that are less concise. For simplicity we assume that $J$ is positive and symmetric in its arguments.

We also introduce the inhomogeneous external field from above. Recall the function $f_c(\sigma)$ defined in \eqref{eq:fc}, we consider the probability measure
\begin{equation}\label{eq: long range mu}
\mu^c_{J,\beta,q}(\sigma)=\mu^c_J(\sigma)\propto e^{-\beta H_{J}(\sigma)}q^{f_c(\sigma)}.
\end{equation}
As above, we want to show that this measure is concentrated on $\Bcal$ under some conditions on $J$. We additionally need that $H_{J}(\sigma)$ is finite for $\sigma\in \Bcal$, this was automatic in the case of nearest neighbour interactions but in the case of long-range interactions we require some decay on $J(i,j)$ as $|i-j|\to\infty$. A sufficient condition is stated in Lemma \ref{lem:longrangeconcentration} below.

Now we want to find reversible dynamics for $\mu^c_J$, finding these dynamics is very similar to the case of nearest-neighbour spins. Suppose that $\sigma$ and $\sigma^\prime$ agree except at sites $k$ and $\ell$, where $\sigma_k=1,\sigma_\ell=-1$ and $\sigma^\prime_k=-1,\sigma^\prime_\ell=1$. We have that
\begin{equation}
\frac{\mu^c_J(\sigma^\prime)}{\mu^c_J(\sigma)}=\frac{\exp\big(-\beta\sum_{i\neq k,\ell}J(i,\ell)\1_{\{\sigma_i=-1\}}-\beta\sum_{j\neq k,\ell}J(j,k)\1_{\{\sigma_j=1\}}\big)q^{2(k-c)}}{\exp\big(-\beta\sum_{i\neq k,\ell}J(i,\ell)\1_{\{\sigma_i=1\}}-\beta\sum_{j\neq k,\ell}J(j,k)\1_{\{\sigma_j=-1\}}\big)q^{2(l-c)}}
\end{equation}
where we used that $\sigma$ and $\sigma^\prime$ agree for $i\neq k,\ell$ and also cancelled the $J(k,\ell)$ term that appears in $\mu^c_J(\sigma^\prime)$ and $\mu^c_J(\sigma)$.
 By the detailed balance equation and using that $\1_{\{\sigma_i=\pm1\}}=(1\pm\sigma_i)/2$ we find that the required rates satisfy
$$
\frac{w(\sigma,\sigma^\prime)}{w(\sigma^\prime,\sigma)}=\exp\bigg(\hspace{-2pt}-\beta\hspace{-2pt}\sum_{i\neq \ell,k}(J(i,k)-J(i,\ell))\sigma_i\hspace{-2pt}\bigg)q^{2(k-l)} =\frac{\exp\bigg(-\frac{\beta}{2}\sum_{i\neq \ell,k}\big(J(i,k)\sigma_k\sigma_i + J(i,\ell)\sigma_\ell\sigma_i\big)\bigg)q^{k-\ell}}{\exp\bigg(-\frac{\beta}{2}\sum_{i\neq \ell,k}\big(J(i,k)\sigma_k^\prime\sigma_i^\prime + J(i,\ell)\sigma_\ell^\prime\sigma_i^\prime\big)\bigg)q^{\ell-k}} .
$$

This suggests that for $\sigma$ and $\sigma^\prime$ as above we should take
\begin{equation}
w(\sigma,\sigma^{\prime})=\frac{1}{2}q^{k-\ell}\bigg(1-\tanh\Big(\frac{\beta}{2}\sum_{i\neq \ell,k}\big(J(i,k)\sigma_k\sigma_i + J(i,\ell)\sigma_\ell\sigma_i\big)\Big)\bigg).
\end{equation}

\begin{lem}\label{lem:longrangeconcentration}
Let $c\in \R$, $\beta\geq 0$ and $q\in(0,1)$. Let $J:\Z\times\Z\to\R$ with $J(i,j)=J(j,i)\geq 0$ for all $i,j\in \Z$ and $J(i,i)=0$. The measure $\mu^c_J$ is stationary and reversible for the Kawasaki dynamics described above with rates given by $w(\sigma,\sigma^{\prime})$. Suppose that 
\begin{align*}
&\sum_{i\leq0}e^{\beta\sum_{j<i< k}J(j,k)}\bigg(1+e^{\beta J(i-1,i)}q^{2(i-c)}
\Big(e^{\beta\sum_{j>i}J(i,j)} + q^{2}e^{\beta\sum_{j<i-1}J(i-1,j)}\Big)^{-1}\bigg)^{-1}
\\
&+ \sum_{i>0}e^{\beta\sum_{j<i< k}J(j,k)}\bigg(1+e^{\beta J(i-1,i)}q^{-2(i-c)}
\Big(e^{\beta\sum_{j>i}J(i,j)} + q^{-2}e^{\beta\sum_{j<i-1}J(i-1,j)}\Big)^{-1}\bigg)^{-1} <\infty
\end{align*}
then $H_{J}(\sigma)<\infty$ for every $\sigma\in\Bcal$ and the measure $\mu^c_J$ concentrates on $\Bcal$.
\end{lem}
\begin{rem}
As noted above, the assumption that $J$ is symmetric and non-negative can be relaxed to give conditions for concentration on $\Bcal$. 
In the case $J(j,k)=0$ whenever $|j-k|>1$ this condition reduces to the condition for nearest neighbour interactions in Lemma \ref{lem:nearestneighbourconcentration}.
In general, a sufficient condition for the sum to be finite is that $J(i,j)=O(|i-j|^{-(2+\eps)})$ for some $\eps>0$ and $J(i-1,i)$ satisfies the condition in Lemma \ref{lem:nearestneighbourconcentration}. 
\end{rem}
\begin{proof}
Define 
$$
\Ical_i=\{\sigma\in\Omega^{\Is}\,:\, \sigma_i\neq \sigma_{i-1}\}.
$$
In analogy with Lemma \ref{lem:nearestneighbourconcentration} we also define $\Omega^{<i}=\{\pm1\}^{\{\dots,i-2,i-1\}}$ and $\Omega^{\geq i}=\{\pm1\}^{\{i,i+1,\dots\}}$ and
$$
S^{(i)}(\pm1)=\sum_{\sigma\in\Omega^{<i}}e^{-\beta H_{J}^{<i}(\sigma)}q^{f_c^{<i}(\sigma)}\1_{\{\sigma_{i-1}=\pm1\}},\qquad 
T^{(i)}(\pm1)=\sum_{\sigma\in\Omega^{\geq i}}e^{-\beta H_{J}^{\geq i}(\sigma)}q^{f_c^{\geq i}(\sigma)}\1_{\{\sigma_{i}=\pm1\}}
$$
where $H_{J}^{< i}(\sigma)=\tfrac12\sum_{j,k <i}J(k,j)\1_{\{\sigma_k\neq \sigma_j\}}$, $H_{J}^{\geq i}(\sigma)=\tfrac12\sum_{j,k \geq i}J(k,j)\1_{\{\sigma_k\neq \sigma_j\}}$, $f_c^{<i}=-\sum_{j=-\infty}^{i-1}(j-c)(1+\sigma_j)$, and finally $f_c^{\geq i}(\sigma)=f_c(\sigma)-f_c^{<i}(\sigma)$. With this notation in place we have that
\begin{equation}
\begin{aligned}
\sum_{\sigma\in\Ical_i}e^{-\beta H_{J}(\sigma)}q^{f_c(\sigma)}
&=\sum_{\sigma\in\Ical_i} 
e^{-\beta\sum_{j<i\leq k}J(j,k)\1_{\{\sigma_j\neq\sigma_k\}}}
e^{-\beta H_{J}^{<i}(\sigma)}
q^{f_c^{<i}(\sigma)}
e^{-\beta H_{J}^{\geq i}(\sigma)}
q^{f_c^{\geq i}(\sigma)}
\\
&\leq e^{-\beta J(i-1,i)}\big( S^{(i)}(-1)T^{(i)}(1) + S^{(i)}(1)T^{(i)}(-1)\big)
\end{aligned}
\end{equation}
and similarly
\begin{equation}
\begin{aligned}
\sum_{\sigma\in\Omega^{\Is}}e^{-\beta H_{J}(\sigma)}q^{f_c(\sigma)}\geq &e^{-\beta\sum_{j<i\leq k}J(j,k)}\big( S^{(i)}(-1)T^{(i)}(1) + S^{(i)}(1)T^{(i)}(-1)\big) 
\\
&+ e^{-\beta\sum_{j<i\leq k}J(j,k)+\beta J(i-1,i)}\big( S^{(i)}(-1)T^{(i)}(-1) + S^{(i)}(1)T^{(i)}(1)\big).
\end{aligned}
\end{equation}
For $i\leq 0$, similarly to the proof of Lemma \ref{lem:nearestneighbourconcentration} this gives us the following bounds
\begin{align}
T^{(i)}(1)\leq& T^{(i)}(-1)q^{-2(i-c)}e^{\beta\sum_{j>i}J(i,j)},
    \\
S^{(i)}(1)\leq& S^{(i)}(-1)q^{-2(i-1-c)}e^{\beta \sum_{j<i-1}J(i-1,j)}.    
\end{align}
Our bound on $\mu_J^c(\Ical_i)$ then becomes
\begin{equation}
\begin{aligned}
\mu^c_J(\Ical_i)&\leq e^{\beta\sum_{j<i\leq k}J(j,k)}\left(e^{\beta J(i-1,i)}+e^{2\beta J(i-1,i)}\left(\frac{T^{(i)}(1)}{T^{(i)}(-1)} + \frac{S^{(i)}(1)}{S^{(i)}(-1)}\right)^{-1}\right)^{-1}
\\
&\leq e^{\beta\sum_{j<i< k}J(j,k)}\left(1+e^{\beta J(i-1,i)}\left(q^{-2(i-c)}e^{\beta\sum_{j>i}J(i,j)} + q^{-2(i-1-c)}e^{\beta\sum_{j<i-1}J(i-1,j)}\right)^{-1}\right)^{-1}
\\
&= e^{\beta\sum_{j<i< k}J(j,k)}\bigg(1+e^{\beta J(i-1,i)}q^{2(i-c)}
\Big(e^{\beta\sum_{j>i}J(i,j)} + q^{2}e^{\beta\sum_{j<i-1}J(i-1,j)}\Big)^{-1}\bigg)^{-1}.
\end{aligned}
\end{equation}
For $i>0$ we have
$$
\mu^c_J(\Ical_i)\leq  e^{\beta\sum_{j<i< k}J(j,k)}\bigg(1+e^{\beta J(i-1,i)}q^{-2(i-c)}
\Big(e^{\beta\sum_{j>i}J(i,j)} + q^{-2}e^{\beta\sum_{j<i-1}J(i-1,j)}\Big)^{-1}\bigg)^{-1}.
$$
\end{proof}

\par \noindent As for nearest-neighbour interactions there is a unique stationary distribution for these dynamics on $\Bcal^n$.

\begin{prop}\label{stationary dist on B^n}
The unique stationary distribution on $\mathcal{B}^n$ is given by, 
$$\nu^n_{J,\beta,q}(\sigma)\defeq \mu^c_{J,\beta,q}(\sigma|N(\sigma)=n)= \frac{\mu^c_J(\sigma)\1_{\{N(\sigma)=n\}}}{\mu^c_J(\{N=n\})}.$$
\end{prop}

\subsection{Transferring dynamics to a particle system}~

The stand up map, $T^n:\Bcal^n\to \Omega$ (see Definition \ref{def: stand up}) now transfers these long range interactions for the Ising chain to a particle system with long range jumps.

The long-range dynamics described above correspond to an intriguing dynamics for the particle system. We begin with the general case and then look at restricted dynamics that are more simple and potentially more natural.

Let us first consider the `bulk' dynamics, that is all particle jumps that do not include particles leaving or entering through the boundary to the right of site $-1$. Before giving the precise description of the dynamics, let us explain a step of the dynamics in an informal way. First we take the convention that for a site with more than one particle, the particles below a given particle as counted as being on `the right' of the particle and the particles above it are counted as being on `the left' of the particle. This will streamline the description of our dynamics slightly. 
A particle can move (hop) up to a certain distance, which is a function of the number of sites and the number of particles between its start and end location. When it moves (hops), potentially from a position in the middle of a stack of particles, it `lands' on the top of the stack of particles at its target site. If the hopping particle moved left/right then the particles left/right of the hopping particle, up to but not including those at the target site, move one site to the left/right. The moving particles that are at the site adjacent to the target site move on top of the hopping particle.

Now we give a precise description of the `bulk' dynamics. Consider the $m^{th}$ particle on site $-i$ which has $\omega_{-i}\geq m$ particles. As part of one step of the dynamics this particle may hop to a site $-j$, taking position $\omega_{-j}+1$ at site $-j$. If the particle has hopped to the right (resp. left) then the particles in positions $1,\dots,m-1$ (resp. $m+1,\dots, \omega_{-i}$) at site $-i$, and each particle on a site strictly between sites $-i$ and $-j$ move one site in the direction of the hopping particle (i.e. one site to the left if $-i>-j$ and one site to the right if $-i<-j$). This includes the particles at the site neighbouring site $-j$ which hop onto site $-j$, taking positions $\omega_{-j}+2,\dots$ and merging the stacks of particles.
The \emph{length} or \emph{distance} of this hop is defined to be $\sum\limits_{s=min\{-i,-j\}+1}^{max\{-i,-j\}-1}(\omega_{-s}+ 1) + (\omega_{-j}+ 1) + m-1$ if the particle hops to the right and $\sum\limits_{s=min\{-i,-j\}+1}^{max\{-i,-j\}-1}(\omega_{-s}+ 1) + (\omega_{-j}+ 1) + \omega_{-i}-m$ if the particle hops to the left. 

By using the mapping between the particle system and the Ising model we see that this corresponds to exchanging a $+1$ and $-1$ spin in the Ising chain whose distance is the distance of the hop. For examples of these dynamics and the spin swaps in the Ising chain that they arise from see Figures \ref{fig: make stack} and \ref{fig:long dynamics} below.

Call the initial configuration $\omega$ and the final configuration $\omega^\prime$ and denote by $d(\omega,\omega^\prime)$ the distance of the hop taking $\omega$ to $\omega^\prime$. 
Note that if $S^{(n)}_r(\sigma)$ is the location of the $r^{th}$ positive spin from the left for $\sigma\in\Bcal^n$, as defined in Definition \ref{def: stand up}, the stand up map $T^n:\Bcal_n\to\Omega$ gives an expression $S_r^{(n)}(\omega) = n + r + \sum\limits_{i=r}^\infty \omega_{-i}$ where $\omega=T^n(\sigma)$. Hence the dynamics inherited from the Ising model via the stand up map $T^n$ have a dependence on $n$. 

To find the rates of jumps on $\Omega $ we first massage the jump rates for the Ising chain slightly.  Suppose that $\sigma$ and $\sigma^\prime$ agree except at sites $\ell$ and $k$, where $\sigma_k=1,\sigma_\ell=-1$ and $\sigma^\prime_k=-1,\sigma^\prime_\ell=1$, then
\begin{equation}
w(\sigma,\sigma^{\prime})=\frac{1}{2}q^{k-\ell}\Bigg(1-\tanh\bigg(\frac{\beta}{2}\sum_{i\in \Z}\Big[(J(i,\ell)-J(i,k)) + 2(J(i,k)-J(i,\ell))\1_{\{\sigma_i=1\}} \Big] \bigg)\Bigg).
\end{equation}
Thus the step of the dynamics described above occurs with rate
\par\vspace{5mm}\begin{equation}
    \begin{aligned}
        &w^{(n)}(\omega,\omega^\prime) = \frac12 q^{d(\omega,\omega^\prime)}
     \Bigg\{ 1- \tanh\Bigg[\frac{\beta}{2}\sum_{a\in\Z}\bigg(J\Big(a,n+i-\sum_{r=i}^\infty\omega_{-r}+m\Big)-J\Big(a,n+j-\sum_{r=j}^\infty\omega_{-r}\Big)\Bigg) 
        \\
        &\,\,+ \beta \sum_{r=1}^\infty\bigg(J\Big(n+r-\sum_{s=r}^\infty\omega_{-s},n+j-\sum_{s=j}^\infty\omega_{-s}\Big) - J\Big(n+r-\sum_{s=r}^\infty\omega_{-s},n+i-\sum_{s=i}^\infty \omega_{-s} + m\Big)\bigg)\Bigg]\Bigg\}.
    \end{aligned}
\end{equation}

\newpage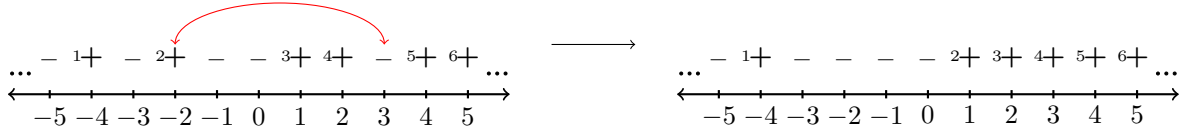
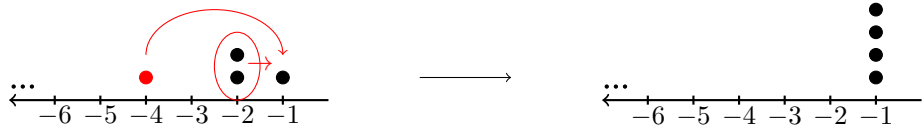
\begin{figure}[H]
    \centering
    \begin{subfigure}[b]{\textwidth}
    \centering
    \begin{tikzpicture}[scale=0.55]
    \draw[thick, <->] (-6,0.8)--(6,0.8);
\foreach \x in {-5,-4,-3,-2,-1,0,1,2,3,4,5}
    \draw[thick, -](\x cm, 0.9)--(\x cm, 0.7) node[anchor=north]{$\x$};
 
 \filldraw [black] (-5.5,1.3) circle (1pt);
\filldraw [black] (-5.7,1.3) circle (1pt);
\filldraw [black] (-5.9,1.3) circle (1pt);   
\node (a) at (-5,1) [label=\textbf{--}]{};
\node (a) at (-4,1) [label=\textbf{+}]{};
\node (a) at (-4.35,1.15) [label=\tiny{1}]{};
\node (a) at (-3,1) [label=\textbf{--}]{};
\node (a) at (-2,1) [label=\textbf{+}]{};
\node (a) at (-2.35,1.15) [label=\tiny{2}]{};
\node (a) at (-1,1) [label=\textbf{--}]{};
\node (a) at (0,1) [label=\textbf{--}]{};
\node (a) at (1,1) [label=\textbf{+}]{};
\node (a) at (0.65,1.15) [label=\tiny{3}]{};
\node (a) at (2,1) [label=\textbf{+}]{};
\node (a) at (1.65,1.15) [label=\tiny{4}]{};
\node (a) at (3,1) [label=\textbf{--}]{};
\node (a) at (4,1) [label=\textbf{+}]{};
\node (a) at (3.65,1.15) [label=\tiny{5}]{};
\node (a) at (5,1) [label=\textbf{+}]{};
\node (a) at (4.65,1.15) [label=\tiny{6}]{};
 \filldraw [black] (5.5,1.3) circle (1pt);
\filldraw [black] (5.7,1.3) circle (1pt);
\filldraw [black] (5.9,1.3) circle (1pt); 
\draw[<->, red] (3,2) arc
    [
        start angle=0,
        end angle=180,
        x radius=2.5cm,
        y radius =1cm
    ] ;
\draw [->] (7,2)--(9,2);
    \draw[thick, <->] (10,0.8)--(22,0.8);
\foreach \x in {11,12,13,14,15,16,17,18,19,20,21}
    \draw[thick, -](\x cm, 0.9)--(\x cm, 0.7) node[anchor=north]{};
\node (a) at (11,-0.41) [label=$-5$]{}; 
\node (a) at (12,-0.41) [label=$-4$]{}; 
\node (a) at (13,-0.41) [label=$-3$]{}; 
\node (a) at (14,-0.41) [label=$-2$]{}; 
\node (a) at (15,-0.41) [label=$-1$]{}; 
\node (a) at (16,-0.35) [label=$0$]{}; 
\node (a) at (17,-0.35) [label=$1$]{}; 
\node (a) at (18,-0.35) [label=$2$]{}; 
\node (a) at (19,-0.35) [label=$3$]{}; 
\node (a) at (20,-0.35) [label=$4$]{}; 
\node (a) at (21,-0.35) [label=$5$]{}; 
 \filldraw [black] (10.5,1.3) circle (1pt);
\filldraw [black] (10.3,1.3) circle (1pt);
\filldraw [black] (10.1,1.3) circle (1pt);   
\node (a) at (11,1) [label=\textbf{--}]{};
\node (a) at (12,1) [label=\textbf{+}]{};
\node (a) at (11.65,1.15) [label=\tiny{1}]{};
\node (a) at (13,1) [label=\textbf{--}]{};
\node (a) at (14,1) [label=\textbf{--}]{};
\node (a) at (15,1) [label=\textbf{--}]{};
\node (a) at (16,1) [label=\textbf{--}]{};
\node (a) at (17,1) [label=\textbf{+}]{};
\node (a) at (16.65,1.15) [label=\tiny{2}]{};
\node (a) at (18,1) [label=\textbf{+}]{};
\node (a) at (17.65,1.15) [label=\tiny{3}]{};
\node (a) at (19,1) [label=\textbf{+}]{};
\node (a) at (18.65,1.15) [label=\tiny{4}]{};
\node (a) at (20,1) [label=\textbf{+}]{};
\node (a) at (19.65,1.15) [label=\tiny{5}]{};
\node (a) at (21,1) [label=\textbf{+}]{};
\node (a) at (20.65,1.15) [label=\tiny{6}]{};
 \filldraw [black] (21.5,1.3) circle (1pt);
\filldraw [black] (21.7,1.3) circle (1pt);
\filldraw [black] (21.9,1.3) circle (1pt); 
    \end{tikzpicture}
    \caption{An example of a long range spin swap in Ising between spins at sites $-2$ and  $3$.}
    \end{subfigure}
    \par \vspace{5mm}
    \begin{subfigure}[b]{\textwidth}
    \centering
    \begin{tikzpicture}[scale=0.6]
    \draw[thick, <-] (-7,-1)--(0,-1);
\foreach \x in {-6,-5,-4,-3,-2,-1}
    \draw[thick, -](\x cm, -1.1)--(\x cm, -0.9) node[anchor=north]{$\x$};

\filldraw [black] (-2,-0.5) circle (4pt);
\filldraw [black] (-2,0) circle (4pt);
\filldraw [black] (-1,-0.5) circle (4pt);
\filldraw [red] (-4,-0.5) circle (4pt);
 \filldraw [black] (-6.5,-0.7) circle (1pt);
\filldraw [black] (-6.7,-0.7) circle (1pt);
\filldraw [black] (-6.9,-0.7) circle (1pt);
\draw[red] (-2,-0.25) ellipse (0.5cm and 0.75cm);
\node(a) at (-1.5,-0.75) [label=\textcolor{red}{$\rightarrow$}]{};
\draw[->, red] (-4,0) arc
    [
        start angle=180,
        end angle=0,
        x radius=1.5cm,
        y radius =1cm
    ] ;
\draw [->] (2,-0.5)--(4,-0.5);
\draw[thick, <-] (6,-1)--(13,-1);
\foreach \x in {7,8,9,10,11,12}
    \draw[thick, -](\x cm, -1.1)--(\x cm, -0.9) node[anchor=north]{};
\node (a) at (7,-2) [label=$-6$]{};    
\node (a) at (8,-2) [label=$-5$]{}; 
\node (a) at (9,-2) [label=$-4$]{}; 
\node (a) at (10,-2) [label=$-3$]{}; 
\node (a) at (11,-2) [label=$-2$]{}; 
\node (a) at (12,-2) [label=$-1$]{}; 
\filldraw [black] (12,-0.5) circle (4pt);
\filldraw [black] (12,0) circle (4pt);
\filldraw [black] (12,0.5) circle (4pt);
\filldraw [black] (12,1) circle (4pt);
 \filldraw [black] (6.5,-0.7) circle (1pt);
\filldraw [black] (6.3,-0.7) circle (1pt);
\filldraw [black] (6.1,-0.7) circle (1pt);
    \end{tikzpicture}
        \caption{The equivalent particle jump in the stood up process}
    \end{subfigure}
    \caption{An example of a particle stack forming coming from the long range Ising dynamics.}
    \label{fig: make stack}
\end{figure}

\begin{figure}[H]
    \centering
    \begin{subfigure}[b]{\textwidth}
    \centering
    \begin{tikzpicture}[scale=0.55]
    \draw[thick, <->] (-6,0.8)--(6,0.8);
\foreach \x in {-5,-4,-3,-2,-1,0,1,2,3,4,5}
    \draw[thick, -](\x cm, 0.9)--(\x cm, 0.7) node[anchor=north]{$\x$};
 
 \filldraw [black] (-5.5,1.3) circle (1pt);
\filldraw [black] (-5.7,1.3) circle (1pt);
\filldraw [black] (-5.9,1.3) circle (1pt);   
\node (a) at (-5,1) [label=\textbf{--}]{};
\node (a) at (-4,1) [label=\textbf{+}]{};
\node (a) at (-4.35,1.15) [label=\tiny{1}]{};
\node (a) at (-3,1) [label=\textbf{+}]{};
\node (a) at (-3.35,1.15) [label=\tiny{2}]{};
\node (a) at (-2,1) [label=\textbf{+}]{};
\node (a) at (-2.35,1.15) [label=\tiny{3}]{};
\node (a) at (-1,1) [label=\textbf{--}]{};
\node (a) at (0,1) [label=\textbf{--}]{};
\node (a) at (1,1) [label=\textbf{+}]{};
\node (a) at (0.65,1.15) [label=\tiny{4}]{};
\node (a) at (2,1) [label=\textbf{--}]{};
\node (a) at (3,1) [label=\textbf{--}]{};
\node (a) at (4,1) [label=\textbf{+}]{};
\node (a) at (3.65,1.15) [label=\tiny{5}]{};
\node (a) at (5,1) [label=\textbf{+}]{};
\node (a) at (4.65,1.15) [label=\tiny{6}]{};
 \filldraw [black] (5.5,1.3) circle (1pt);
\filldraw [black] (5.7,1.3) circle (1pt);
\filldraw [black] (5.9,1.3) circle (1pt); 
\draw[<->, red] (2,2) arc
    [
        start angle=0,
        end angle=180,
        x radius=2.5cm,
        y radius =1cm
    ] ;
\draw [->] (7,2)--(9,2);
    \draw[thick, <->] (10,0.8)--(22,0.8);
\foreach \x in {11,12,13,14,15,16,17,18,19,20,21}
    \draw[thick, -](\x cm, 0.9)--(\x cm, 0.7) node[anchor=north]{};
\node (a) at (11,-0.41) [label=$-5$]{}; 
\node (a) at (12,-0.41) [label=$-4$]{}; 
\node (a) at (13,-0.41) [label=$-3$]{}; 
\node (a) at (14,-0.41) [label=$-2$]{}; 
\node (a) at (15,-0.41) [label=$-1$]{}; 
\node (a) at (16,-0.35) [label=$0$]{}; 
\node (a) at (17,-0.35) [label=$1$]{}; 
\node (a) at (18,-0.35) [label=$2$]{}; 
\node (a) at (19,-0.35) [label=$3$]{}; 
\node (a) at (20,-0.35) [label=$4$]{}; 
\node (a) at (21,-0.35) [label=$5$]{}; 
 \filldraw [black] (10.5,1.3) circle (1pt);
\filldraw [black] (10.3,1.3) circle (1pt);
\filldraw [black] (10.1,1.3) circle (1pt);   
\node (a) at (11,1) [label=\textbf{--}]{};
\node (a) at (13,1) [label=\textbf{--}]{};
\node (a) at (12,1) [label=\textbf{+}]{};
\node (a) at (11.65,1.15) [label=\tiny{1}]{};
\node (a) at (14,1) [label=\textbf{+}]{};
\node (a) at (13.65,1.15) [label=\tiny{2}]{};
\node (a) at (15,1) [label=\textbf{--}]{};
\node (a) at (16,1) [label=\textbf{--}]{};
\node (a) at (17,1) [label=\textbf{+}]{};
\node (a) at (16.65,1.15) [label=\tiny{3}]{};
\node (a) at (18,1) [label=\textbf{+}]{};
\node (a) at (17.65,1.15) [label=\tiny{4}]{};
\node (a) at (19,1) [label=\textbf{--}]{};
\node (a) at (20,1) [label=\textbf{+}]{};
\node (a) at (19.65,1.15) [label=\tiny{5}]{};
\node (a) at (21,1) [label=\textbf{+}]{};
\node (a) at (20.65,1.15) [label=\tiny{6}]{};
 \filldraw [black] (21.5,1.3) circle (1pt);
\filldraw [black] (21.7,1.3) circle (1pt);
\filldraw [black] (21.9,1.3) circle (1pt); 
    \end{tikzpicture}
    \caption{An example of a long range spin swap in Ising between spins at sites $-3$ and  $2$.}
    \end{subfigure}
    \par\vspace{5mm}
    \begin{subfigure}[b]{\textwidth}
    \centering
    \begin{tikzpicture}[scale=0.6]
    \draw[thick, <-] (-7,-1)--(0,-1);
\foreach \x in {-6,-5,-4,-3,-2,-1}
    \draw[thick, -](\x cm, -1.1)--(\x cm, -0.9) node[anchor=north]{$\x$};

\filldraw [black] (-3,-0.5) circle (4pt);
\filldraw [black] (-3,0) circle (4pt);
\filldraw [red] (-4,-0.5) circle (4pt);
\filldraw [black] (-4,0) circle (4pt);
 \filldraw [black] (-6.5,-0.7) circle (1pt);
\filldraw [black] (-6.7,-0.7) circle (1pt);
\filldraw [black] (-6.9,-0.7) circle (1pt);
\draw[red] (-3,-0.25) ellipse (0.5cm and 0.75cm);
\node(a) at (-2.5,-0.75) [label=\textcolor{red}{$\rightarrow$}]{};
\draw[->, red] (-4,0.5) arc
    [
        start angle=180,
        end angle=0,
        x radius=1.5cm,
        y radius =1cm
    ] ;
\draw [->] (2,-0.5)--(4,-0.5);
\draw[thick, <-] (6,-1)--(13,-1);
\foreach \x in {7,8,9,10,11,12}
    \draw[thick, -](\x cm, -1.1)--(\x cm, -0.9) node[anchor=north]{};
\node (a) at (7,-2) [label=$-6$]{};    
\node (a) at (8,-2) [label=$-5$]{}; 
\node (a) at (9,-2) [label=$-4$]{}; 
\node (a) at (10,-2) [label=$-3$]{}; 
\node (a) at (11,-2) [label=$-2$]{}; 
\node (a) at (12,-2) [label=$-1$]{}; 
\filldraw [black] (12,-0.5) circle (4pt);
\filldraw [black] (11,-0.5) circle (4pt);
\filldraw [black] (11,-0) circle (4pt);
\filldraw [black] (9,-0.5) circle (4pt);
 \filldraw [black] (6.5,-0.7) circle (1pt);
\filldraw [black] (6.3,-0.7) circle (1pt);
\filldraw [black] (6.1,-0.7) circle (1pt);
    \end{tikzpicture}
        \caption{The equivalent particle jump in the stood up process}
    \end{subfigure}
    \caption{An example of the particle dynamics coming from the long range Ising dynamics.}
    \label{fig:long dynamics}
\end{figure}
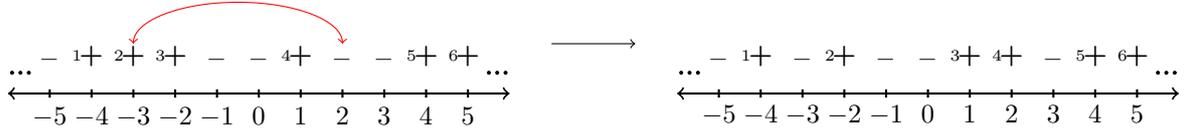
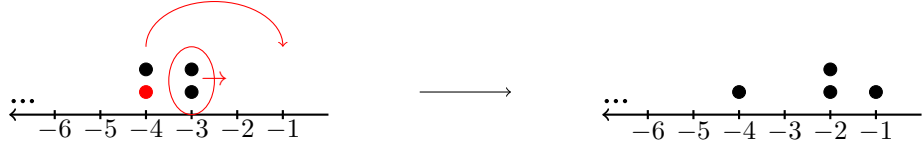

\par Now we will describe the boundary dynamics. First let us consider particles that jump out of the system through the boundary to the right of site $-1$. Just like in the bulk, the $m^\text{th}$ particle at site $-i$ for $i\geq 1$ might jump $i$ sites to the right, meaning it leaves the system through the boundary. When it does this all the particles to its right shift one site to the right and any particles at site $-1$ also leave the system through the boundary. This boundary dynamic is equivalent to the spin swap in the corresponding Ising configuration between the first positive spin and the $m^\text{th}$ negative spin to the right of the $i^\text{th}$ positive spin (note this will be to the left of the $(i+1)^\text{th}$ positive spin). For an example of this and its corresponding Ising spin swap see Figure \ref{fig: system to boundary}.
\par Let $\omega$ denote the state of the system and $\omega'$ be the state after the $m^\text{st}$ particle at site $-i$ has jumped into the boundary; such a boundary jump happens with the following rate,
\begin{multline*}
w^{(n)}(\omega,\omega')=\frac{1}{2}q^{-\sum\limits_{j=1}^{i-1}(\omega_{-j}+1)-m}\Bigg\{1-\tanh \Bigg[\frac{\beta}{2}\sum\limits_{a\in\mathbb{Z}}\bigg(J\Big(a,n+i+m-\sum\limits_{s=i}^\infty\omega_{-s}\Big)-J\Big(a,n+1-\sum\limits_{s=1}^\infty \omega_{-s}\Big)\bigg)\\+\beta\sum\limits_{r=1}^\infty\bigg(J\Big(n+r-\sum\limits_{s=r}^\infty\omega_{-s},n+1-\sum\limits_{s=1}^\infty \omega_{-s}\Big)-J\Big(n+r-\sum\limits_{s=r}^\infty\omega_{-s},n+i+m-\sum\limits_{s=i}^\infty\omega_{-s}\Big)\bigg)\Bigg]\Bigg\}.
\end{multline*}
Here we used that in the corresponding Ising configuration the first +1 spin is at site $S^{(n)}_1=n+1-\sum\limits_{s=1}^\infty \omega_{-s}$.
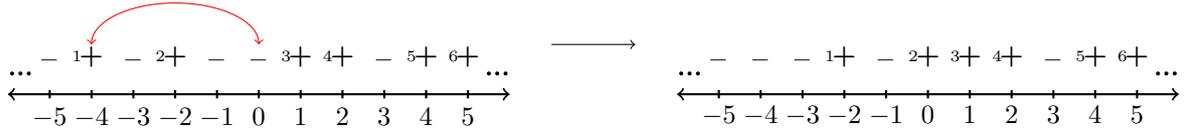
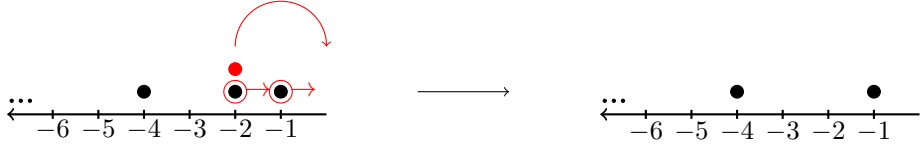
\begin{figure}[H]
    \centering
    \begin{subfigure}[b]{\textwidth}
    \centering
    \begin{tikzpicture}[scale=0.55]
    \draw[thick, <->] (-6,0.8)--(6,0.8);
\foreach \x in {-5,-4,-3,-2,-1,0,1,2,3,4,5}
    \draw[thick, -](\x cm, 0.9)--(\x cm, 0.7) node[anchor=north]{$\x$};
 
 \filldraw [black] (-5.5,1.3) circle (1pt);
\filldraw [black] (-5.7,1.3) circle (1pt);
\filldraw [black] (-5.9,1.3) circle (1pt);   
\node (a) at (-5,1) [label=\textbf{--}]{};
\node (a) at (-4,1) [label=\textbf{+}]{};
\node (a) at (-4.35,1.15) [label=\tiny{1}]{};
\node (a) at (-3,1) [label=\textbf{--}]{};
\node (a) at (-2,1) [label=\textbf{+}]{};
\node (a) at (-2.35,1.15) [label=\tiny{2}]{};
\node (a) at (-1,1) [label=\textbf{--}]{};
\node (a) at (0,1) [label=\textbf{--}]{};
\node (a) at (1,1) [label=\textbf{+}]{};
\node (a) at (0.65,1.15) [label=\tiny{3}]{};
\node (a) at (2,1) [label=\textbf{+}]{};
\node (a) at (1.65,1.15) [label=\tiny{4}]{};
\node (a) at (3,1) [label=\textbf{--}]{};
\node (a) at (4,1) [label=\textbf{+}]{};
\node (a) at (3.65,1.15) [label=\tiny{5}]{};
\node (a) at (5,1) [label=\textbf{+}]{};
\node (a) at (4.65,1.15) [label=\tiny{6}]{};
 \filldraw [black] (5.5,1.3) circle (1pt);
\filldraw [black] (5.7,1.3) circle (1pt);
\filldraw [black] (5.9,1.3) circle (1pt); 
\draw[<->, red] (0,2) arc
    [
        start angle=0,
        end angle=180,
        x radius=2cm,
        y radius =1cm
    ] ;
\draw [->] (7,2)--(9,2);
    \draw[thick, <->] (10,0.8)--(22,0.8);
\foreach \x in {11,12,13,14,15,16,17,18,19,20,21}
    \draw[thick, -](\x cm, 0.9)--(\x cm, 0.7) node[anchor=north]{};
\node (a) at (11,-0.41) [label=$-5$]{}; 
\node (a) at (12,-0.41) [label=$-4$]{}; 
\node (a) at (13,-0.41) [label=$-3$]{}; 
\node (a) at (14,-0.41) [label=$-2$]{}; 
\node (a) at (15,-0.41) [label=$-1$]{}; 
\node (a) at (16,-0.35) [label=$0$]{}; 
\node (a) at (17,-0.35) [label=$1$]{}; 
\node (a) at (18,-0.35) [label=$2$]{}; 
\node (a) at (19,-0.35) [label=$3$]{}; 
\node (a) at (20,-0.35) [label=$4$]{}; 
\node (a) at (21,-0.35) [label=$5$]{}; 
 \filldraw [black] (10.5,1.3) circle (1pt);
\filldraw [black] (10.3,1.3) circle (1pt);
\filldraw [black] (10.1,1.3) circle (1pt);   
\node (a) at (11,1) [label=\textbf{--}]{};
\node (a) at (12,1) [label=\textbf{--}]{};
\node (a) at (13,1) [label=\textbf{--}]{};
\node (a) at (14,1) [label=\textbf{+}]{};
\node (a) at (13.65,1.15) [label=\tiny{1}]{};
\node (a) at (15,1) [label=\textbf{--}]{};
\node (a) at (16,1) [label=\textbf{+}]{};
\node (a) at (15.65,1.15) [label=\tiny{2}]{};
\node (a) at (17,1) [label=\textbf{+}]{};
\node (a) at (16.65,1.15) [label=\tiny{3}]{};
\node (a) at (18,1) [label=\textbf{+}]{};
\node (a) at (17.65,1.15) [label=\tiny{4}]{};
\node (a) at (19,1) [label=\textbf{--}]{};
\node (a) at (20,1) [label=\textbf{+}]{};
\node (a) at (19.65,1.15) [label=\tiny{5}]{};
\node (a) at (21,1) [label=\textbf{+}]{};
\node (a) at (20.65,1.15) [label=\tiny{6}]{};
 \filldraw [black] (21.5,1.3) circle (1pt);
\filldraw [black] (21.7,1.3) circle (1pt);
\filldraw [black] (21.9,1.3) circle (1pt); 
    \end{tikzpicture}
    \caption{An example of a long range spin swap in Ising between spins at sites $-4$ and  $0$.}
    \end{subfigure}
    \par \hspace{5mm}
    \begin{subfigure}[b]{\textwidth}
    \centering
    \begin{tikzpicture}[scale=0.6]
    \draw[thick, <-] (-7,-1)--(0,-1);
\foreach \x in {-6,-5,-4,-3,-2,-1}
    \draw[thick, -](\x cm, -1.1)--(\x cm, -0.9) node[anchor=north]{$\x$};

\filldraw [black] (-2,-0.5) circle (4pt);
\filldraw [red] (-2,0) circle (4pt);
\filldraw [black] (-1,-0.5) circle (4pt);
\filldraw [black] (-4,-0.5) circle (4pt);
 \filldraw [black] (-6.5,-0.7) circle (1pt);
\filldraw [black] (-6.7,-0.7) circle (1pt);
\filldraw [black] (-6.9,-0.7) circle (1pt);
\draw[red] (-2,-0.5) ellipse (0.25cm and 0.25cm);
\node(a) at (-1.5,-1) [label=\textcolor{red}{$\rightarrow$}]{};
\draw[red] (-1,-0.5) ellipse (0.25cm and 0.25cm);
\node(a) at (-0.5,-1) [label=\textcolor{red}{$\rightarrow$}]{};
\draw[->, red] (-2,0.5) arc
    [
        start angle=180,
        end angle=0,
        x radius=1cm,
        y radius =1cm
    ] ;

\draw [->] (2,-0.5)--(4,-0.5);
\draw[thick, <-] (6,-1)--(13,-1);
\foreach \x in {7,8,9,10,11,12}
    \draw[thick, -](\x cm, -1.1)--(\x cm, -0.9) node[anchor=north]{};
\node (a) at (7,-2) [label=$-6$]{};    
\node (a) at (8,-2) [label=$-5$]{}; 
\node (a) at (9,-2) [label=$-4$]{}; 
\node (a) at (10,-2) [label=$-3$]{}; 
\node (a) at (11,-2) [label=$-2$]{}; 
\node (a) at (12,-2) [label=$-1$]{};
\filldraw [black] (12,-0.5) circle (4pt);
\filldraw [black] (9,-0.5) circle (4pt);
 \filldraw [black] (6.5,-0.7) circle (1pt);
\filldraw [black] (6.3,-0.7) circle (1pt);
\filldraw [black] (6.1,-0.7) circle (1pt);
    \end{tikzpicture}
        \caption{The equivalent particle jump in the stood up process}
    \end{subfigure}
    \caption{An example of a particle hop into the boundary coming from the long range Ising dynamics.}
    \label{fig: system to boundary}
\end{figure}

\par Particles can also enter the system from an infinite well of particles in the boundary. A particle enters the system and jumps to site $-i$, then all the particles to the right of this site move one site to the left. Any particles that were at site $-i+1$ join to form a stack at site $-i$. As this left shift happens some number of particles can be brought in from the boundary to site $-1$. This boundary dynamic is equivalent to a spin swap in the long range Ising model where the $i^\text{th}$ positive spin swaps with a negative spin that is to the left of the first positive spin. The number of particles that enter to site $-1$ is equal to the number of negative spins between the new first and second positive spins after this spin swap. For an example of this and its corresponding Ising spin swap see Figure \ref{fig:boundary to system}.
\par Let $\omega$ denote the state of the system and $\omega'$ the state reached after a particle from the boundary jumped to site $-i$ and brought in another $m-1$ particles with it. Such a boundary jump happens with the following rate,
\begin{multline*}
w^{(n)}(\omega,\omega')=\frac{1}{2}q^{\sum\limits_{j=1}^{i-1}(\omega_{-j}+1)+m}\Bigg\{1-\tanh \Bigg[\frac{\beta}{2}\sum\limits_{a\in\mathbb{Z}}\bigg(J\Big(a,n+1-m-\sum\limits_{s=1}^\infty\omega_{-s}\Big)-J\Big(a,n+i-\sum\limits_{s=i}^\infty \omega_{-s}\Big)\bigg)\\+\beta\sum\limits_{r=1}^\infty\bigg(J\Big(n+r-\sum\limits_{s=r}^\infty\omega_{-s},n+i-\sum\limits_{s=i}^\infty \omega_{-s}\Big)-J\Big(n+r-\sum\limits_{s=r}^\infty\omega_{-s},n+1-m-\sum\limits_{s=1}^\infty\omega_{-s}\Big)\bigg)\Bigg]\Bigg\}.
\end{multline*}
Again we used that in the corresponding Ising configuration the first positive spin is at site $S^{(n)}_1=n+1-\sum\limits_{s=1}^\infty \omega_{-s}$. 

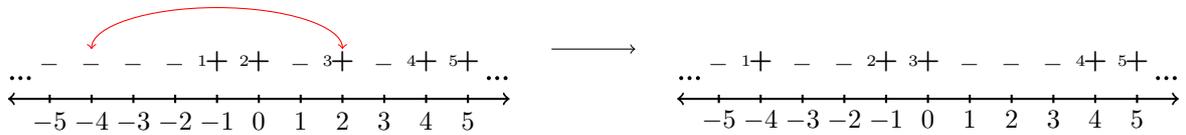
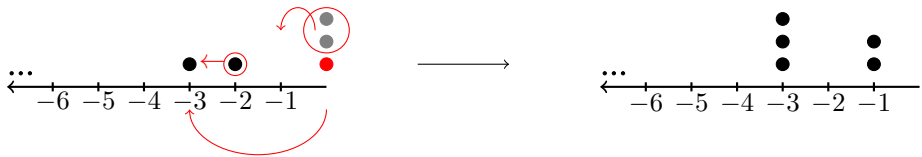
\begin{figure}[H]
    \centering
    \begin{subfigure}[b]{\textwidth}
    \centering
    \begin{tikzpicture}[scale=0.55]
    \draw[thick, <->] (-6,0.8)--(6,0.8);
\foreach \x in {-5,-4,-3,-2,-1,0,1,2,3,4,5}
    \draw[thick, -](\x cm, 0.9)--(\x cm, 0.7) node[anchor=north]{$\x$};
 
 \filldraw [black] (-5.5,1.3) circle (1pt);
\filldraw [black] (-5.7,1.3) circle (1pt);
\filldraw [black] (-5.9,1.3) circle (1pt);   
\node (a) at (-5,1) [label=\textbf{--}]{};
\node (a) at (-4,1) [label=\textbf{--}]{};
\node (a) at (-1.35,1.15) [label=\tiny{1}]{};
\node (a) at (-3,1) [label=\textbf{--}]{};
\node (a) at (-2,1) [label=\textbf{--}]{};
\node (a) at (-0.35,1.15) [label=\tiny{2}]{};
\node (a) at (-1,1) [label=\textbf{+}]{};
\node (a) at (0,1) [label=\textbf{+}]{};
\node (a) at (1,1) [label=\textbf{--}]{};
\node (a) at (2,1) [label=\textbf{+}]{};
\node (a) at (3,1) [label=\textbf{--}]{};
\node (a) at (1.65,1.15) [label=\tiny{3}]{};
\node (a) at (4,1) [label=\textbf{+}]{};
\node (a) at (3.65,1.15) [label=\tiny{4}]{};
\node (a) at (5,1) [label=\textbf{+}]{};
\node (a) at (4.65,1.15) [label=\tiny{5}]{};
 \filldraw [black] (5.5,1.3) circle (1pt);
\filldraw [black] (5.7,1.3) circle (1pt);
\filldraw [black] (5.9,1.3) circle (1pt); 
\draw[<->, red] (2,2) arc
    [
        start angle=0,
        end angle=180,
        x radius=3cm,
        y radius =1cm
    ] ;
\draw [->] (7,2)--(9,2);
    \draw[thick, <->] (10,0.8)--(22,0.8);
\foreach \x in {11,12,13,14,15,16,17,18,19,20,21}
    \draw[thick, -](\x cm, 0.9)--(\x cm, 0.7) node[anchor=north]{};
\node (a) at (11,-0.41) [label=$-5$]{}; 
\node (a) at (12,-0.41) [label=$-4$]{}; 
\node (a) at (13,-0.41) [label=$-3$]{}; 
\node (a) at (14,-0.41) [label=$-2$]{}; 
\node (a) at (15,-0.41) [label=$-1$]{}; 
\node (a) at (16,-0.35) [label=$0$]{}; 
\node (a) at (17,-0.35) [label=$1$]{}; 
\node (a) at (18,-0.35) [label=$2$]{}; 
\node (a) at (19,-0.35) [label=$3$]{}; 
\node (a) at (20,-0.35) [label=$4$]{}; 
\node (a) at (21,-0.35) [label=$5$]{}; 
 \filldraw [black] (10.5,1.3) circle (1pt);
\filldraw [black] (10.3,1.3) circle (1pt);
\filldraw [black] (10.1,1.3) circle (1pt);   
\node (a) at (11,1) [label=\textbf{--}]{};
\node (a) at (13,1) [label=\textbf{--}]{};
\node (a) at (12,1) [label=\textbf{+}]{};
\node (a) at (11.65,1.15) [label=\tiny{1}]{};
\node (a) at (14,1) [label=\textbf{--}]{};
\node (a) at (15,1) [label=\textbf{+}]{};
\node (a) at (16,1) [label=\textbf{+}]{};
\node (a) at (17,1) [label=\textbf{--}]{};
\node (a) at (14.65,1.15) [label=\tiny{2}]{};
\node (a) at (18,1) [label=\textbf{--}]{};
\node (a) at (15.65,1.15) [label=\tiny{3}]{};
\node (a) at (19,1) [label=\textbf{--}]{};
\node (a) at (20,1) [label=\textbf{+}]{};
\node (a) at (19.65,1.15) [label=\tiny{4}]{};
\node (a) at (21,1) [label=\textbf{+}]{};
\node (a) at (20.65,1.15) [label=\tiny{5}]{};
 \filldraw [black] (21.5,1.3) circle (1pt);
\filldraw [black] (21.7,1.3) circle (1pt);
\filldraw [black] (21.9,1.3) circle (1pt); 
    \end{tikzpicture}
    \caption{An example of a long range spin swap in Ising between spins at sites $-4$ and  $2$.}
    \end{subfigure}
    \par\vspace{5mm}
    \begin{subfigure}[b]{\textwidth}
    \centering
    \begin{tikzpicture}[scale=0.6]
    \draw[thick, <-] (-7,-1)--(0,-1);
\foreach \x in {-6,-5,-4,-3,-2,-1}
    \draw[thick, -](\x cm, -1.1)--(\x cm, -0.9) node[anchor=north]{$\x$};

\filldraw [red] (0,-0.5) circle (4pt);
\filldraw [gray] (0,0) circle (4pt);
\filldraw [gray] (0,0.5) circle (4pt);
\filldraw [black] (-2,-0.5) circle (4pt);
\filldraw [black] (-3,-0.5) circle (4pt);
 \filldraw [black] (-6.5,-0.7) circle (1pt);
\filldraw [black] (-6.7,-0.7) circle (1pt);
\filldraw [black] (-6.9,-0.7) circle (1pt);
\draw[red] (-2,-0.5) ellipse (0.25cm and 0.25cm);
\node(a) at (-2.5,-1) [label=\textcolor{red}{$\leftarrow$}]{};
\draw[->, red] (0,-1.5) arc
    [
        start angle=360,
        end angle=180,
        x radius=1.5cm,
        y radius =1cm
    ] ;
\draw[red] (0,0.25) ellipse (0.5cm and 0.5cm);
\draw[->, red] (-0.25,0.25) arc
    [
        start angle=0,
        end angle=180,
        x radius=0.375cm,
        y radius =0.5cm
    ] ;
\draw [->] (2,-0.5)--(4,-0.5);
\draw[thick, <-] (6,-1)--(13,-1);
\foreach \x in {7,8,9,10,11,12}
    \draw[thick, -](\x cm, -1.1)--(\x cm, -0.9) node[anchor=north]{};
\node (a) at (7,-2) [label=$-6$]{};    
\node (a) at (8,-2) [label=$-5$]{}; 
\node (a) at (9,-2) [label=$-4$]{}; 
\node (a) at (10,-2) [label=$-3$]{}; 
\node (a) at (11,-2) [label=$-2$]{}; 
\node (a) at (12,-2) [label=$-1$]{}; 
\filldraw [black] (12,-0.5) circle (4pt);
\filldraw [black] (12,0) circle (4pt);
\filldraw [black] (10,0) circle (4pt);
\filldraw [black] (10,-0.5) circle (4pt);
\filldraw [black] (10,0.5) circle (4pt);
 \filldraw [black] (6.5,-0.7) circle (1pt);
\filldraw [black] (6.3,-0.7) circle (1pt);
\filldraw [black] (6.1,-0.7) circle (1pt);
    \end{tikzpicture}
        \caption{The equivalent particle jump in the stood up process}
    \end{subfigure}
    \caption{An example of particles entering from the boundary coming from the long range Ising dynamics.}
    \label{fig:boundary to system}
\end{figure}
\begin{rem}
   We note that when the target site of a boundary jump is $-1$, this can be seen as some number, $m$, of particles jumping in from the boundary to site $-1$. This boundary dynamic is equivalent to the first positive spin in the corresponding Ising state jumping $m$ sites to the left. 
   \par Let $\omega$ denote the starting state of the system and $\omega'$ the state reached after $m$ particles enter from the boundary all to site $-1$; such a boundary jump happens at rate,
    \begin{multline*}
w^{(n)}(\omega,\omega')=\frac{q^{m}}{2}\Bigg\{1-\tanh \Bigg[\frac{\beta}{2}\sum\limits_{a\in\mathbb{Z}}\bigg(J\Big(a,n+1-m-\sum\limits_{s=1}^\infty\omega_{-s}\Big)-J\Big(a,n+1-\sum\limits_{s=1}^\infty \omega_{-s}\Big)\bigg)\\+\beta\sum\limits_{r=1}^\infty\bigg(J\Big(n+r-\sum\limits_{s=r}^\infty\omega_{-s},n+1-\sum\limits_{s=1}^\infty \omega_{-s}\Big)-J\Big(n+r-\sum\limits_{s=r}^\infty\omega_{-s},n+1-m-\sum\limits_{s=1}^\infty\omega_{-s}\Big)\bigg)\Bigg]\Bigg\}.
\end{multline*}
\end{rem}

\begin{prop}\label{prop: meas of long range stood up}
    For the process on $\Omega$ described by the jump rates $w^{(n)}(\cdot,\cdot)$ the following measure is stationary and reversible, 
    $$\pi^{(n)}(\omega)=\nu^n((T^{n})^{-1}(\omega)).$$
\end{prop}
\begin{proof}
    This holds since the bijection $T^n$ gives an equivalence of reversible stationary measures for the two processes.
\end{proof}
An explicit expression for $\pi^{(n)}(\omega)$ is extremely cumbersome, but we note that after some calculation we find
\begin{equation}\label{eq:fcparticle}
\begin{aligned}
&f_c\big((T^n)^{-1}(\omega)\big) = \sum_{r\geq 0}\bigg[\1_{\{\sum_{j\geq r+1}\omega_{-j}\leq n+r-1\}}\Big(n+r-\sum_{j\geq r+1}\omega_{-j}+1-\big[1\vee(n+r+1-\sum_{j\geq r}\omega_{-j})\big]\Big)
\\
&\Big(n+r-\sum_{j\geq r+1}\omega_{-j}+\big[1\vee(n+r+1-\sum_{j\geq r}\omega_{-j})\big]-2c\Big)-\1_{\{\sum_{j\geq r}\omega_{-j}\geq n+r\}}2\Big(n+r-\sum_{j\geq r}\omega_{-j}-c\Big)\bigg]
\\
&\qquad\qquad\qquad\,\,=:\sum_{r\geq 0}g_c^{(r)}(\omega)
\end{aligned}
\end{equation}
so that, due to the terms of the form $\omega_{-r}\omega_{-j}$ and the indicators with conditions on partial sums of $\omega_{-j}$'s, $\pi^{(n)}(\omega)$ cannot be written as a product of independent variables on each site.

\subsection{Restricted long-range dynamics}\label{sec:restricted dynamics}

We can also restrict the dynamics to a case that is simpler to describe and, perhaps, more physically relevant. First we consider the corresponding dynamics for the long-range Ising model. Suppose that a negative spin at $j\in \Z$ may swap with either of the positive spins at the ends of the run of consecutive negative spins that contains $j$. Suppose these positive spins are at sites $k$ and $\ell$. If the negative spin at $j$ hops, then it exchanges with the positive spin at $k$ with probability proportional to the probability this swap would occur in the unrestricted dynamics, given the negative spin at $j$ will make a swap, and exchanges with the positive spin at $\ell$ with the complementary probability.
In the particle system, this corresponds to stacks of particles splitting and their top halves moving to a neighbouring site, merging with the (possibly empty) stack of particles there (see Figure \ref{fig:stacks_merging} for an example). These dynamics are not quite an extension of the zero-range model to allow multiple particles to move at the same time due to the rates of the dynamics having a dependence on the configuration after a proposed step is made.
\begin{prop}
    For the long range Ising process with the restricted dynamics described above, the measure $\mu^c_{J,\beta,q}$ given in \eqref{eq: long range mu}, is still reversible and stationary. Similarly, for the ``stood up" process with these restricted dynamics the measure $\pi^{(n)}$ given in Proposition \ref{prop: meas of long range stood up}, is still reversible and stationary.
\end{prop}
\begin{proof}
    This holds since for a reversible Markov chain with some transitions removed independently, the ``cut" chain is also reversible with the same stationary measure (see for example Liggett \cite{liggett_book} II. Proposition 5.10).
\end{proof}
\newpage 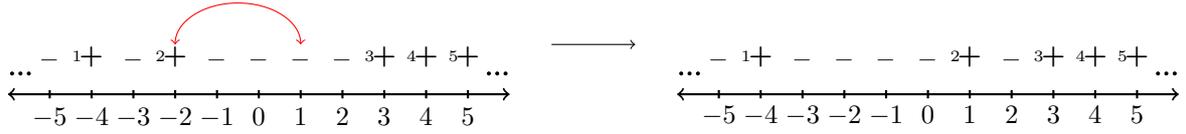
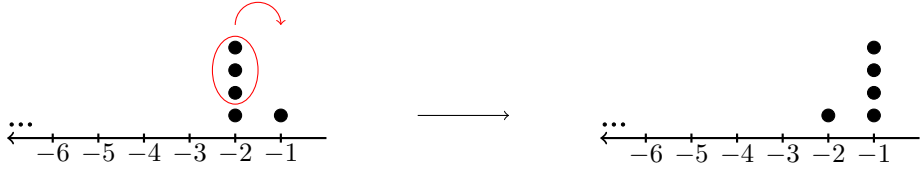
\begin{figure}[H]
    \centering
    \begin{subfigure}[b]{\textwidth}
    \centering
    \begin{tikzpicture}[scale=0.55]
    \draw[thick, <->] (-6,0.8)--(6,0.8);
\foreach \x in {-5,-4,-3,-2,-1,0,1,2,3,4,5}
    \draw[thick, -](\x cm, 0.9)--(\x cm, 0.7) node[anchor=north]{$\x$};
 
 \filldraw [black] (-5.5,1.3) circle (1pt);
\filldraw [black] (-5.7,1.3) circle (1pt);
\filldraw [black] (-5.9,1.3) circle (1pt);   
\node (a) at (-5,1) [label=\textbf{--}]{};
\node (a) at (-4,1) [label=\textbf{+}]{};
\node (a) at (-4.35,1.15) [label=\tiny{1}]{};
\node (a) at (-3,1) [label=\textbf{--}]{};
\node (a) at (-2,1) [label=\textbf{+}]{};
\node (a) at (-2.35,1.15) [label=\tiny{2}]{};
\node (a) at (-1,1) [label=\textbf{--}]{};
\node (a) at (0,1) [label=\textbf{--}]{};
\node (a) at (1,1) [label=\textbf{--}]{};
\node (a) at (2,1) [label=\textbf{--}]{};
\node (a) at (3,1) [label=\textbf{+}]{};
\node (a) at (2.65,1.15) [label=\tiny{3}]{};
\node (a) at (4,1) [label=\textbf{+}]{};
\node (a) at (3.65,1.15) [label=\tiny{4}]{};
\node (a) at (5,1) [label=\textbf{+}]{};
\node (a) at (4.65,1.15) [label=\tiny{5}]{};
 \filldraw [black] (5.5,1.3) circle (1pt);
\filldraw [black] (5.7,1.3) circle (1pt);
\filldraw [black] (5.9,1.3) circle (1pt); 
\draw[<->, red] (1,2) arc
    [
        start angle=0,
        end angle=180,
        x radius=1.5cm,
        y radius =1cm
    ] ;
\draw [->] (7,2)--(9,2);
    \draw[thick, <->] (10,0.8)--(22,0.8);
\foreach \x in {11,12,13,14,15,16,17,18,19,20,21}
    \draw[thick, -](\x cm, 0.9)--(\x cm, 0.7) node[anchor=north]{};
\node (a) at (11,-0.41) [label=$-5$]{}; 
\node (a) at (12,-0.41) [label=$-4$]{}; 
\node (a) at (13,-0.41) [label=$-3$]{}; 
\node (a) at (14,-0.41) [label=$-2$]{}; 
\node (a) at (15,-0.41) [label=$-1$]{}; 
\node (a) at (16,-0.35) [label=$0$]{}; 
\node (a) at (17,-0.35) [label=$1$]{}; 
\node (a) at (18,-0.35) [label=$2$]{}; 
\node (a) at (19,-0.35) [label=$3$]{}; 
\node (a) at (20,-0.35) [label=$4$]{}; 
\node (a) at (21,-0.35) [label=$5$]{}; 
 \filldraw [black] (10.5,1.3) circle (1pt);
\filldraw [black] (10.3,1.3) circle (1pt);
\filldraw [black] (10.1,1.3) circle (1pt);   
\node (a) at (11,1) [label=\textbf{--}]{};
\node (a) at (13,1) [label=\textbf{--}]{};
\node (a) at (12,1) [label=\textbf{+}]{};
\node (a) at (11.65,1.15) [label=\tiny{1}]{};
\node (a) at (14,1) [label=\textbf{--}]{};
\node (a) at (15,1) [label=\textbf{--}]{};
\node (a) at (16,1) [label=\textbf{--}]{};
\node (a) at (17,1) [label=\textbf{+}]{};
\node (a) at (16.65,1.15) [label=\tiny{2}]{};
\node (a) at (18,1) [label=\textbf{--}]{};
\node (a) at (18.65,1.15) [label=\tiny{3}]{};
\node (a) at (19,1) [label=\textbf{+}]{};
\node (a) at (20,1) [label=\textbf{+}]{};
\node (a) at (19.65,1.15) [label=\tiny{4}]{};
\node (a) at (21,1) [label=\textbf{+}]{};
\node (a) at (20.65,1.15) [label=\tiny{5}]{};
 \filldraw [black] (21.5,1.3) circle (1pt);
\filldraw [black] (21.7,1.3) circle (1pt);
\filldraw [black] (21.9,1.3) circle (1pt); 
    \end{tikzpicture}
    \caption{An example of a long range spin swap in Ising between spins at sites $-2$ and  $1$.}
    \end{subfigure}
    \par\vspace{5mm}
    \begin{subfigure}[b]{\textwidth}
    \centering
    \begin{tikzpicture}[scale=0.6]
    \draw[thick, <-] (-7,-1)--(0,-1);
\foreach \x in {-6,-5,-4,-3,-2,-1}
    \draw[thick, -](\x cm, -1.1)--(\x cm, -0.9) node[anchor=north]{$\x$};

\filldraw [black] (-1,-0.5) circle (4pt);
\filldraw [black] (-2,-0.5) circle (4pt);
\filldraw [black] (-2,0.5) circle (4pt);
\filldraw [black] (-2,0) circle (4pt);
\filldraw [black] (-2,1) circle (4pt);
 \filldraw [black] (-6.5,-0.7) circle (1pt);
\filldraw [black] (-6.7,-0.7) circle (1pt);
\filldraw [black] (-6.9,-0.7) circle (1pt);
\draw[red] (-2,0.5) ellipse (0.5cm and 0.75cm);
\draw[->, red] (-2,1.5) arc
    [
        start angle=180,
        end angle=0,
        x radius=0.5cm,
        y radius =0.5cm
    ] ;
\draw [->] (2,-0.5)--(4,-0.5);
\draw[thick, <-] (6,-1)--(13,-1);
\foreach \x in {7,8,9,10,11,12}
    \draw[thick, -](\x cm, -1.1)--(\x cm, -0.9) node[anchor=north]{};
\node (a) at (7,-2) [label=$-6$]{};    
\node (a) at (8,-2) [label=$-5$]{}; 
\node (a) at (9,-2) [label=$-4$]{}; 
\node (a) at (10,-2) [label=$-3$]{}; 
\node (a) at (11,-2) [label=$-2$]{}; 
\node (a) at (12,-2) [label=$-1$]{}; 
\filldraw [black] (12,-0.5) circle (4pt);
\filldraw [black] (12,0.5) circle (4pt);
\filldraw [black] (12,0) circle (4pt);
\filldraw [black] (12,1) circle (4pt);
\filldraw [black] (11,-0.5) circle (4pt);
 \filldraw [black] (6.5,-0.7) circle (1pt);
\filldraw [black] (6.3,-0.7) circle (1pt);
\filldraw [black] (6.1,-0.7) circle (1pt);
    \end{tikzpicture}
        \caption{The equivalent particle jump in the stood up process}
    \end{subfigure}
    \caption{An example of the particle dynamics coming from the long range Ising dynamics.}
    \label{fig:stacks_merging}
\end{figure}

\subsection{Natural reversible dynamics for the particle system}\label{sec:nicelong}

In the previous sections we defined a natural long range dynamics for the Ising model and then used the stand up mapping $T^n:\Bcal_n\to\Omega$ to transfer these dynamics to a rather unusual particle dynamics on $\Omega$. It is also possible to begin with a natural dynamics for the particle system, with the moves of the restricted dynamics described in Section \ref{sec:restricted dynamics} and rates with a simple form, and use the inverse map $(T^n)^{-1}$ to find a measure on $\Bcal_n$ for which the transferred dynamics are reversible. This leads to a rather unusual measure on $\Bcal_n$ but the task of proving reversibility for the transferred dynamics, and therefore for the particle dynamics, will still be relatively straightforward. 

Consider a configuration $\omega\in\Omega$ and for $i\in \N$ define $r_i(\omega) =\omega_{-i}\vee 1$. If $1\leq k \leq \omega_{-i}$ then $k$ particles at site $-i<-1$ jump to $-i+1$, resulting in configuration $\omega^\prime\in\Omega$, at rate
\begin{equation}\label{eq:particlerate1}
w(\omega,\omega^\prime)= \frac{q^{-k}}{r_i(\omega)r_{i+1}(\omega)}.
\end{equation}
$k$ particles at site $-i$ jump to $-i-1$, resulting in configuration $\omega^\prime\in\Omega$, at rate
\begin{equation}\label{eq:particlerate2}
w(\omega,\omega^\prime)= \frac{q^{k}}{r_i(\omega)r_{i-1}(\omega)}.
\end{equation}
Lastly, if $1\leq k \leq \omega_{-1}$ then $k$ particles jump out of the system resulting in configuration $\omega^\prime\in\Omega$, at rate
\begin{equation}\label{eq:particlerate3}
w(\omega,\omega^\prime)= \frac{q^{-k}}{r_1(\omega)},
\end{equation}
and $k$ particles jump into the system from the right onto site -1 resulting in configuration $\omega^\prime\in\Omega$, at rate
\begin{equation}\label{eq:particlerate4}
w(\omega,\omega^\prime)= \frac{q^{k}}{r_1(\omega)}.
\end{equation}
In order to find a measure for which these dynamics are reversible, we transfer these dynamics to an Ising chain and find a reversible measure on $\Omega^{\Is}$. For $\omega\in\Omega$ let $\sigma=(T^n)^{-1}(\omega)$ and, committing a slight abuse of notation, we define $r_i(\sigma)=r_i(T^n(\sigma))=r_i(\omega)$ to be the maximum of 1 and the number of negative spins between the $i^{th}$ and $(i+1)^{st}$ positive spin in $\sigma$, counted from the left. We further define $r_0(\sigma)=1$. We take the probability measure $\lambda^{(n)}$ on $\Omega^{\Is}$ by

\begin{equation}
\lambda^{(n)}(\sigma)=\frac{1}{Z}q^{f_c(\sigma)}\prod_{i\geq 0}r_i(\sigma).
\end{equation}
The corresponding dynamics have the same transitions as those described in Section \ref{sec:restricted dynamics}. A negative spin may swap places with either the closest positive spin to its left or the closest positive spin to its right. Equivalently a positive spin may swap with a negative spin at an arbitrary distance to its left or right, provided that this swap maintains the relative order of positive spins. Suppose this positive spin is the $i^{th}$ positive spin in $\sigma$ counted from the left. This positive spin swaps places with the negative spin at distance $k$ to its right, provided this swap maintains the relative order of positive spins in $\sigma$, at rate
$$
\frac{q^{-k}}{r_{i-1}(\sigma)r_{i}(\sigma)},
$$
and with the negative spin at distance $k$ to its left, provided this swap maintains the relative order of positive spins in $\sigma$, at rate 
$$
\frac{q^{k}}{r_{i-1}(\sigma)r_{i}(\sigma)}.
$$
If the proposed swap would not maintain the relative order of positive spins in $\sigma$ then it is not allowed. It can easily be seen that these dynamics are reversible for $\lambda^{(n)}$ and hence that the dynamics for the particle system are reversible for the measure
\begin{equation}
\kappa^{(n)}(\omega) := \lambda^{(n)}\big((T^n)^{-1}(\omega)\big).
\end{equation}
Recall from \eqref{eq:fcparticle} that we have
\begin{equation}
f_c\big((T^n)^{-1}(\omega)\big) =\sum_{r\geq 0}g_c^{(r)}(\omega)
\end{equation}
where $g_c^{(r)}(\omega)$ depends on the occupation of every site left of $-r$. We also have $\prod_{i\geq0}r_i\big((T^n)^{-1}(\omega)\big)=\prod_{i\geq 1}r_i(\omega)$. Putting this together gives
\begin{equation}
  \kappa^{(n)}(\omega) \propto \prod_{i\geq 1}  q^{g_c^{(i)}(\omega)}r_i(\omega),
\end{equation}
where we defined $r_0(\omega)= 1$.
Note that, due to the dependencies within $g_c^{(i)}$, this measure is not a product of independent measures for each site.
We summarise these results in the following lemma.

\begin{lem}\label{lem:particletoIsingconcentration}
For every $c\in\mathbb{R}$ and $q\in(0,1)$, the measure $\kappa^{(n)}$ is stationary and reversible for the
dynamics described above with rates given by \eqref{eq:particlerate1}-\eqref{eq:particlerate4}. Moreover, $\kappa^{(n)}$ concentrates on $\Omega$.
\end{lem}
\begin{proof}
The measure $\lambda^{(n)}$ and rates are chosen specifically to satisfy detailed balance as a straightforward calculation will verify. Reversibility with respect to $\kappa^{(n)}$ is hence inherited from the dynamics on $\Omega^{\Is}$  through the mapping $(T^n)^{-1}$.

It remains to show concentration on $\Omega$. Using $T^n$ it is sufficient to show concentration of $\lambda^{(n)}$ on $\Bcal$.
For $j>0$ consider the event $\{\sigma_j=-1\}$ and let $F_j: \Omega^{\Is}\to\Omega^{\Is}$ be the mapping that flips the spin at $j$. For $\sigma\in\Omega^{\Is}$ such that $\sigma_j=-1$ and $j^\prime$ such that $j$ lies in the run of negative spins corresponding to $r_{j^\prime}(\sigma)$ we have 
\begin{equation}
\lambda^{(n)}(\sigma)= q^{2(j-c)}\frac{r_{j^\prime}(\sigma)}{(r_{j^\prime}(\sigma)-1)\vee 1}\lambda^{(n)}\big(F_j(\sigma)\big)\leq 2q^{2(j-c)}\lambda^{(n)}\big(F_j(\sigma)\big).
\end{equation}
Now as $F_j$ is a bijection we have that $\lambda^{(n)}(\{\sigma_j=-1\})\leq 2q^{2(j-c)}\lambda^{(n)}(\{\sigma_j=1\})\leq 2q^{2(j-c)}$. The upper bound is summable, hence by Borel-Cantelli there are $\lambda^{(n)}$-a.s. only finitely many $j>0$ such that $\sigma_j=-1$.

Now for $j<k<0$ suppose that $\sigma_j=\sigma_k=1$ so that $\sigma$ has an entire (at least one) run of negative spins left of 0. Let $r_{j^\prime-1}(\sigma)$ and $r_{j^\prime}(\sigma)$ correspond to the runs of negative spins immediately to the left and right of $j$, respectively. Note that $r_{j^\prime}(\sigma)\leq k-j-1$. We have that
\begin{align*}
\lambda^{(n)}(\{\sigma_j=\sigma_k=1\})&=\sum_{\sigma\in\Omega^{\Is}}\1_{\{\sigma_j=\sigma_k=1\}}q^{-2(j-c)}q^{f_c(F_j(\sigma))}\frac{r_{j^\prime-1}(\sigma)r_{j^\prime}(\sigma)}{r_{j^\prime-1}(\sigma)+r_{j^\prime}(\sigma)+1}\prod_{i\geq1} r_i\big(F_j(\sigma)\big)
\\
&\leq q^{-2(j-c)}\sum_{\sigma\in\Omega^{\Is}}\1_{\{\sigma_j=\sigma_k=1\}}(k-j-1)(\sigma)\lambda^{(n)}\big(F_j(\sigma)\big)
\\
&= (k-j-1)q^{-2(j-c)} \lambda^{(n)}\big(\{\sigma_j=-1,\sigma_k=1\}\big).
\end{align*}
Hence the probability of having a run of negative spins starting immediately to the left of $k<0$ is bounded above by
\begin{equation}
\sum_{j<k}(k-j-1)q^{-2(j-c)}=\frac{q^{2(2+c-2k)}}{(1-q^2)^2}.
\end{equation}
This probability is summable over $k<0$, hence by Borel-Cantelli there are $\lambda^{(n)}$-a.s. only finitely many $k<0$ such that a run of negative spins starts at $k-1$.
\end{proof}

\end{document}